\documentclass[11pt,reqno]{amsart}
\usepackage[margin=1.18in]{geometry}
\usepackage[OT1]{fontenc}

\usepackage{amsmath,amssymb,amsfonts,mathtools,mathrsfs,bm}
\usepackage{enumitem}
\usepackage{microtype}
\usepackage[dvipsnames]{xcolor}
\usepackage{tikz}
\definecolor{inkblue}{RGB}{25,53,89}
\usepackage[colorlinks=true,citecolor=MidnightBlue,linkcolor=BrickRed,
 urlcolor=RoyalBlue]{hyperref}
\hypersetup{pdftitle={Large deviations of the spectral radius of iid subgaussian random matrices},
 pdfauthor={Yi Han},
 pdfsubject={Upper tails at speed n and quadratic exponential lower-tail bounds}}
\usepackage{bookmark}
\usepackage{needspace}
\usepackage{etoolbox}
\usepackage{longtable,booktabs,array}
\numberwithin{equation}{section}
\newtheorem{theorem}{Theorem}[section]
\newtheorem{proposition}[theorem]{Proposition}
\newtheorem{lemma}[theorem]{Lemma}
\newtheorem{fact}[theorem]{Fact}
\newtheorem{corollary}[theorem]{Corollary}
\theoremstyle{definition}
\newtheorem{definition}[theorem]{Definition}
\newtheorem{algorithms}[theorem]{Algorithm}
\newtheorem{example}[theorem]{Example}
\theoremstyle{remark}
\newtheorem{remark}[theorem]{Remark}
\providecommand{\tightlist}{\setlength{\itemsep}{0pt}\setlength{\parskip}{0pt}}
\allowdisplaybreaks[2]
\makeatletter
\renewcommand{\paragraph}{\@startsection{paragraph}{4}\z@
 {.5\linespacing\@plus.5\linespacing}{.25\linespacing}
 {\normalfont\bfseries}}
\makeatother
\title[Large deviations of the spectral radius]{Large deviations of the spectral radius of iid subgaussian random matrices}
\author{Yi Han}
\thanks{Department of Mathematics, Massachusetts Institute of Technology,
Cambridge, MA.\newline Email: \texttt{hanyi16@mit.edu}.\newline A Lean formalization of the whole paper is available, see Section \ref{lean}.}
\date{}
\newcommand{\E}{\mathbb E}\newcommand{\Prob}{\mathbb P}
\newcommand{\R}{\mathbb R}\newcommand{\C}{\mathbb C}
\newcommand{\F}{\mathbb F}\newcommand{\norm}[1]{\lVert#1\rVert}
\newcommand{\HS}{\mathrm{HS}}
\newcommand{\op}{\mathrm{op}}\newcommand{\dd}{\,\mathrm d}
\newcommand{\tr}{\operatorname{Tr}}

\begin{document}
\begin{abstract}
We study large deviations of the spectral radius of
$X_n=n^{-1/2}(\xi_{ij})_{i,j=1}^n$, where the entries are iid, centered,
variance-one, and subgaussian, with zero pseudovariance
($\E\xi_{ij}^2=0$) in the complex case. We write $\rho(X_n)$ for its
spectral radius, i.e. the largest modulus among its eigenvalues.
For every fixed $r>1$, we prove
\[
 \liminf_{n\to\infty}\frac1n\log\Prob\{\rho(X_n)>r\}
 \ge -\frac\beta2\bigl(r^2-1-2\log r\bigr),
\]
where $\beta=1$ for real entries and $\beta=2$ for complex entries.
We prove a matching upper bound for real symmetric laws with
Gaussian-dominated even moments and for complex laws satisfying the
sharp Gaussian Laplace-transform bound. These classes include discrete
distributions. In the complex sharp class, we also identify the
exponential rate for an eigenvalue to enter a fixed disk outside the
unit disk. We then turn to lower deviations and establish a
quadratic-speed bound. Under the additional assumption that the entry
law $\mu$ has a bounded density, we prove
\[
 \Prob\{\rho(X_n)\le r\}\le e^{-c_{\mu,r}n^2}
\]
for every fixed $0<r<1$ and all sufficiently large $n$, where $c_{\mu,r}>0$.
For the matching upper-tail classes with a bounded density, these bounds
give a full speed-$n$ large deviation principle, with the displayed
upper-tail rate for $r\ge1$ and infinite rate for $r<1$.
The upper-tail proof uses a change of measure that preserves the entry
support and creates an outlying eigenvalue. The lower-tail proof develops
a weighted comparison for adaptive orthonormal observations and applies
it to Arnoldi residuals.
\end{abstract}
\maketitle
\tableofcontents

\section{Introduction and main results}\label{sec:introduction}\label{sec:main-results}\label{u:sec:statements}
Let $X_n=n^{-1/2}(\xi_{ij})_{i,j=1}^n$ be a random matrix with iid
centered, variance-one real or complex entries. Its spectral radius
$\rho(X_n)$, the largest modulus among its eigenvalues, governs the asymptotic
stability of discrete-time linear systems \cite[Section~1]{Geman1986}.
Write $\mu$ for the common entry law and $\xi\sim\mu$ for a generic entry.
The study of Gaussian
non-Hermitian spectra dates back to Ginibre's work in 1965
\cite{Ginibre1965}. For matrices with independent non-Gaussian entries,
classical spectral-radius estimates were obtained in 1986 by Geman
\cite{Geman1986} and Bai--Yin \cite{BaiYin}, who proved
$\limsup_{n\to\infty}\rho(X_n)\le1$ almost surely under suitable
moment assumptions.

The circular law identifies the limiting empirical eigenvalue distribution
as uniform on the unit disk; its finite-variance form was proved by
Tao--Vu, with an appendix by Krishnapur \cite{TaoVuCircular2010}. In particular,
$\liminf_{n\to\infty}\rho(X_n)\ge1$ almost surely.
Further work on spectral-radius upper bounds has addressed weaker moment
assumptions and inhomogeneous variance profiles. In the iid setting,
Bordenave--Chafa\"i--Garc\'ia-Zelada \cite{BCG2022} proved
$\rho(X_n)\to1$ in probability under only the second-moment assumptions.
Alt--Erd\H{o}s--Kr\"uger \cite{AEK2021} established quantitative
high-probability bounds on the spectral radius under stronger hypotheses,
also allowing
a nonconstant variance profile. Hachem--Louvaris \cite{HachemLouvaris2026}
treat general variance profiles under weak moment assumptions.
Han \cite{Han} gives spectral-radius concentration bounds under general
tail assumptions.

At the finer fluctuation scale, Rider \cite{Rider2003} and
Rider--Sinclair \cite{RiderSinclair2014} established Gumbel laws for the
complex and real Ginibre spectral radii, respectively.
Cipolloni--Erd\H{o}s--Xu obtained precise spectral-radius asymptotics
\cite{CEYPrecise} and proved Gumbel universality and a Poisson limit for
the extremal eigenvalues of proper complex iid matrices under moment
assumptions \cite{CEYUniversality}. These results describe shrinking
neighborhoods of the unit circle. For fixed deviations, Xu--Zeng
\cite[Theorem~1.2]{XuZeng} established speed-$n$ large deviation principles
for the spectral radius and the rightmost eigenvalue of both real and
complex Ginibre ensembles. We study fixed-radius deviations beyond
Gaussian entries: the events $\rho(X_n)>r$ for $r>1$ and
$\rho(X_n)\le r$ for $0\le r<1$. Throughout this paper the entries are
subgaussian; in the complex case we also assume zero pseudovariance,
$\E\xi^2=0$.

The two deviations have different mechanisms. An upper deviation can be
created by a single outlying eigenvalue, with a logarithmic probability
cost of order $n$. A lower deviation confines every eigenvalue to a
smaller disk, and our probability bound has a quadratic exponent.
The hypotheses needed for the conclusions also differ.
For upper tails, we prove a universal lower bound under subgaussianity
and identify the exact Gaussian rate under the stronger moment or Laplace
comparison assumptions of Theorem~\ref{u:thm:matching}:
\[
 \lim_{n\to\infty}\frac1n\log\Prob\{\rho(X_n)>r\}=-I_\beta(r),
 \qquad I_\beta(r)=\frac\beta2\bigl(r^2-1-2\log r\bigr),\quad r>1,
\]
where $\beta=1$ for real entries and $\beta=2$ for complex entries.
These upper-tail results allow atomic laws. For lower tails, the common
subgaussian assumptions together with a bounded density give
\[
 \Prob\{\rho(X_n)\le r\}\le e^{-c_{\mu,r}n^2},
 \qquad 0\le r<1,\quad n\ge n_0(\mu,r),
\]
in both the real and complex cases, with $c_{\mu,r}>0$.
We also obtain local
complex eigenvalue deviations outside the unit disk.
Under the combined upper- and lower-tail hypotheses, we obtain a full
speed-$n$ large deviation principle for the spectral radius
(Theorem~\ref{thm:full-ldp}); its rate is infinite below one.
The statements below collect these results with precise hypotheses,
followed by examples and a discussion of their scope.
A proof overview in Subsection~\ref{sec:proof-overview} then explains
how the two probability scales arise. The detailed proofs follow in two parts.

\subsection{Relation to earlier work}

\paragraph{Largest-eigenvalue and operator-norm large deviations for Wigner matrices.}
For Wigner matrices, Guionnet and Husson \cite{GuionnetHusson} obtained
the Gaussian speed-$n$ upper-tail rate for the largest eigenvalue,
and hence for the operator norm, under sharp-subgaussian hypotheses
and suitable regularity assumptions.
Augeri--Guionnet--Husson \cite{AugeriGuionnetHusson} and
Cook--Ducatez--Guionnet \cite{CookDucatezGuionnet} developed nonuniversal
largest-eigenvalue rates for more general subgaussian laws, using
spherical integrals.
For nonconstant variance profiles, Husson \cite{HussonVariance2022} and,
under more general profile assumptions, Ducatez--Guionnet--Husson
\cite{DucatezGuionnetHusson2024} proved speed-$n$ largest-eigenvalue LDPs
for sharp-subgaussian entries, with rates depending on the profile.
For tails heavier than Gaussian, Augeri \cite{AugeriHeavy2016} proved
a largest-eigenvalue LDP at speed $n^{\alpha/2}$ when the entry tails
are stretched exponential with exponent $\alpha\in(0,2)$, under suitable
tail and angular assumptions. In the sparse setting,
Ganguly--Hiesmayr--Nam \cite{GangulyHiesmayrNam2024} studied the largest
eigenvalue of weighted Erd\H{o}s--R\'enyi graphs with constant average
degree, including weights with tails heavier than Gaussian.

\paragraph{Large deviations of empirical spectral measures for Hermitian matrices.}
For Hermitian random matrices, the empirical spectral distribution records
the collective behavior of the eigenvalues and typically has a different deviation scale from an
individual extreme eigenvalue. Ben Arous and Guionnet \cite{BAG1997}
established its speed-$n^2$ LDP for Gaussian Wigner matrices, with a
logarithmic-energy rate function minimized by the semicircle law.
Hiai and Petz proved related quadratic-speed LDPs for Wishart matrices
\cite{HiaiPetzWishart} and, in the unitary setting, for empirical eigenvalue
distributions on the unit circle \cite{HiaiPetzUnitary}.
For general subgaussian Wigner entries, the concentration estimates of
Guionnet and Zeitouni \cite{GZ2000}, combined with truncation, give
quadratic exponential upper bounds for fixed weak deviations from the
semicircle law; such bounds alone do not identify a full LDP or its rate
function.

Beyond this setting, both the speed and the mechanism can change.
Bordenave and Caputo \cite{BordenaveCaputo2014} proved an empirical-measure
LDP at speed $n^{1+\alpha/2}$ for Wigner matrices with
stretched-exponential entry tails of exponent $\alpha\in(0,2)$,
under their tail and angular hypotheses. Augeri \cite{AugeriSparse2025}
obtained a speed-$n^2p_n$ LDP for supercritical sparse Wigner matrices
with bounded entry weights and sparsity satisfying
$\log n\ll np_n\ll n$.
For power-law tails, Bordenave--Guionnet--Male \cite{BordenaveGuionnetMale2024}
proved speed-$n$ empirical-measure LDPs for L\'evy matrices under their
natural normalization, as well as for weighted sparse matrices with
constant average degree.
These results show how collective spectral
deviations depend on the entry tails and sparsity. For a broader account
of random-matrix large deviations, see Guionnet \cite{GuionnetICM2023}.

\paragraph{Outliers of non-Hermitian deformations.}
Our proof of the lower bound for upper-tail probabilities uses an
exponential tilt to make an eigenvalue outside the unit disk typical.
To locate this outlier under the tilted law, we use the theory of
non-Hermitian deformations; see Tao \cite{Tao} for bounded-rank perturbations of iid matrices,
O'Rourke--Renfrew \cite{OR2014} for elliptic matrices, and
Bordenave--Capitaine \cite{BC2016} for additive iid deformations.
Han \cite{HanBand} extends spectral confinement and bounded-rank outlier
results to inhomogeneous band and elliptic models under quantitative
assumptions. In our argument, the required statement also accommodates
a deterministic right deformation bounded in Hilbert--Schmidt norm
and a correlated error vanishing in operator norm; we prove this form
in Lemma~\ref{u:lem:outlier}. The rate calculation additionally requires
the likelihood cost of producing that deformation.
Characteristic-polynomial methods provide another approach to controlling
eigenvalue locations; see \cite{BCG2022,MR2025,HanFiniteRank2025}.

\paragraph{Further large-deviation results for Gaussian and integrable ensembles.}
Complementing the Ginibre results of Xu--Zeng \cite{XuZeng} discussed above,
Byun--Lee--Oh \cite{ByunLeeOh} study upper deviations for elliptic
Ginibre ensembles; their complex Ginibre specialization gives the
Gaussian case of our local eigenvalue-deviation theorem.

For inward deviations, the speed-$n^2$ empirical-measure LDPs of
Ben Arous--Zeitouni \cite{BAZ1998} and Petz--Hiai \cite{PetzHiai1998}
for real and complex Ginibre matrices already imply quadratic
exponential bounds for disk confinement. More precise lower-tail
asymptotics are available in Gaussian and radial integrable models
\cite{CMV2016,Noda2026}. For iid matrices with bounded-density
subgaussian entries, \hyperref[part:lower]{Part~\ref*{part:lower}} proves the spectral-radius
lower-tail estimate $\Prob\{\rho(X_n)\le r\}\le e^{-c_{\mu,r}n^2}$
for every fixed $0<r<1$ and all sufficiently large $n$. A speed-$n^2$ LDP
for the empirical spectral measure remains open in this generality.
\subsection{Common setting}
Let $\F=\R$ or $\C$, and let $\beta=\dim_\R\F\in\{1,2\}$.
We distinguish the entry matrix from its normalization:
\begin{equation}\label{eq:common-normalization}
 \mathsf X_n=(\xi_{ij})_{i,j=1}^n,
 \qquad X_n=\frac{\mathsf X_n}{\sqrt n},
 \qquad \rho(X_n)=\max_{\lambda\in\operatorname{spec}(X_n)}|\lambda|.
\end{equation}
Here $\operatorname{spec}(M)$ denotes the spectrum of a square matrix
$M$, taken in $\C$ even when its entries are real.
The entries are iid with a law $\mu$ independent of $n$;
$\xi\sim\mu$ denotes a generic entry and $\mathbb P_n$ the law of $X_n$.
Our common assumptions are
\begin{equation}\label{u:eq:basic}
 \E\xi=0,\qquad \E|\xi|^2=1,\qquad
 \E e^{c_0|\xi|^2}<\infty\quad\text{for some }c_0>0.
\end{equation}
In the complex case we additionally require zero pseudovariance:
\begin{equation}\label{u:eq:properness}
 \E\xi^2=0.
\end{equation}
Throughout the paper, all complex entry laws satisfy \eqref{u:eq:properness}.
For upper deviations, define
\begin{equation}\label{eq:common-rate}
 I_\beta(r)=\frac\beta2\bigl(r^2-1-2\log r\bigr),\qquad r>1.
\end{equation}
We set $I_\beta(1)=0$ by continuity.

Unmarked vector norms are Euclidean and unmarked matrix norms are operator
norms; $\|v\|_\infty=\max_j|v_j|$ and
$\|M\|_{\HS}^2=\sum_{i,j}|M_{ij}|^2$.
We write $^T$ and $^*$ for transpose and conjugate transpose,
$I$ for the identity of the indicated size, and $e_j$ for a coordinate
unit vector. The real inner product is
$\langle u,w\rangle=\operatorname{Re}(u^*w)$.
For Hermitian matrices, $A\preceq B$ means that $B-A$ is positive
semidefinite, and $A>0$ means positive definite.
The notation $\mathcal L(Z)$ denotes the law of $Z$, and $\mathbf1_E$ the
indicator of an event $E$. For a probability law $\nu$,
$\E_\nu f=\int f\dd\nu$. Integrals in real coordinates use Lebesgue
measure unless specified otherwise. Constants $C,c$ may change from line
to line; subscripts indicate dependence.
A standard complex Gaussian has independent $N(0,1/2)$ real and imaginary
parts, so its second absolute moment is one.

\subsection{Upper tails}

We begin with upper deviations: under the common assumptions above,
their probabilities admit a universal lower bound at speed $n$
with the Gaussian rate $I_\beta$.

\begin{theorem}[Universal lower bound for upper deviations]\label{u:thm:lower}
Under \eqref{u:eq:basic} and, in the complex case, \eqref{u:eq:properness},
\[
 \liminf_{n\to\infty}\frac1n\log\Prob\{\rho(X_n)>r\}\ge-I_\beta(r),
 \qquad r>1.
\]
In the real case the same lower bound holds for the event that $X_n$
has a positive real eigenvalue larger than $r$.
\end{theorem}

In the real moment condition below,
$(2m-1)!!=1\cdot3\cdots(2m-1)$ is the odd double factorial.

\Needspace{5\baselineskip}
\begin{theorem}[Matching upper-tail rates]\label{u:thm:matching}
Suppose that one of the following sets of assumptions holds:
\begin{enumerate}
\item $\F=\R$, $\xi\overset{d}=-\xi$, $\E\xi^2=1$, and
      $\E\xi^{2m}\le(2m-1)!!$ for every integer $m\ge1$;
\item $\F=\C$, $\E\xi=0$, $\E|\xi|^2=1$, condition \eqref{u:eq:properness}, and
\begin{equation}\label{u:eq:sharp}
 \E e^{2\operatorname{Re}(\bar u\xi)}\le e^{|u|^2},\qquad u\in\C.
\end{equation}
\end{enumerate}
Then the strict and non-strict right tails both satisfy
\begin{equation}\label{u:eq:matching}
 \lim_{n\to\infty}\frac1n\log\Prob\{\rho(X_n)>r\}
 =\lim_{n\to\infty}\frac1n\log\Prob\{\rho(X_n)\ge r\}
 =-I_\beta(r),\qquad r>1.
\end{equation}

Neither case requires a density. In the complex case the real and
imaginary components may be dependent.
\end{theorem}

Both sets of assumptions imply \eqref{u:eq:basic} and allow discrete
entry laws; examples and their applicable conclusions are collected in
Subsection~\ref{sec:examples}.

\subsection{Lower tails}
Unlike the preceding upper-tail theorems, which allow atomic entry laws,
the following lower-tail bound requires a bounded density on $\R$ or $\C$.

\begin{theorem}[Quadratic lower-tail bound over $\R$ and $\C$]\label{l:thm:MT}
Assume \eqref{u:eq:basic} and, when $\F=\C$, \eqref{u:eq:properness}.
Suppose in addition that $\mu$ has a bounded density with respect to
Lebesgue measure on $\F$ (planar Lebesgue measure when $\F=\C$).
For every fixed $0\le r<1$, there are constants $c_{\mu,r}>0$ and
$n_0(\mu,r)$ such that
\begin{equation}\label{l:eq:M.2}
 \Prob\{\rho(X_n)\le r\}\le e^{-c_{\mu,r}n^2},
 \qquad n\ge n_0(\mu,r).
\end{equation}
\end{theorem}
No symmetry of the entry law, positive lower bound on its density, or
smoothness of the density is required.
This is an exponential upper bound for a lower-tail event; it does not
assert a matching lower bound or an exact lower-tail rate function.
The atomic examples below do not satisfy this density hypothesis.
The real uniform law on $[-\sqrt3,\sqrt3]$ and the complex uniform law
on the disk of radius $\sqrt2$ satisfy the hypotheses for both tails.
The proof is given in \hyperref[part:lower]{Part~\ref*{part:lower}}.

\begin{remark}[Zero pseudovariance in the lower-tail bound]\label{rem:lower-pseudovariance}
We expect the complex lower-tail bound in Theorem~\ref{l:thm:MT} to remain
valid without \eqref{u:eq:properness}, retaining centering, unit variance,
subgaussianity, and a bounded planar density. Zero pseudovariance makes
the real covariance a scalar multiple of the identity and substantially
simplifies the Gaussian comparison. Removing it would require a more
involved treatment of general real covariance in the comparison and
Gaussian estimates. We do not prove this extension here.
\end{remark}

\begin{remark}[Bernoulli lower tails]\label{rem:bernoulli}
Symmetric Bernoulli entries satisfy the upper-tail hypotheses, but our
lower-tail argument encounters two obstructions in this case.
First, the comparison for arbitrary observation events in
Theorem~\ref{l:thm:AC} cannot hold for Bernoulli laws: a nonzero deterministic
linear observation has finite support under Bernoulli entries and a
continuous law under Gaussian entries. Taking the event to be that
finite support gives probabilities one and zero, respectively. This
obstruction persists for a flat observation such as
$h^{-1/2}\sum_{j=1}^h\xi_j$.
Second, the border estimate in Proposition~\ref{l:thm:MB} fixes the
$(n-1)\times(n-1)$ core $A$ and averages over the remaining $2n-1$
entries. For independent symmetric Bernoulli entries, every border
event $E$ therefore satisfies
\[
 \Prob(E\mid A)=0
 \quad\text{or}\quad
 \Prob(E\mid A)\ge 2^{-(2n-1)}.
\]
Consequently, a conditional bound of order $e^{-cn^2}$ would require
that no Bernoulli border completion satisfy the event, rather than
merely that its continuous volume be small. An extension thus needs
additional control of the arithmetic feasibility of the constraints
as the core varies.

We also explored a multipoint arithmetic approach, motivated by the
two-point quantitative invertibility estimates of
\cite[Theorems 1.5 and 1.9]{HanTwoPoint}. For real Bernoulli matrices,
this approach would require upper bounds of the form $e^{-c_m n}$ for
simultaneous small-singular-value events at $m$ nonreal shifts, with
$c_m\to\infty$ as $m\to\infty$. We did not obtain such estimates.
An additive error $e^{-cn}$ with $c$ independent of $m$ would prevent
this improvement. These limitations of the available arguments do
not constitute counterexamples to a Bernoulli lower-tail bound;
no such bound is proved here.
\end{remark}

\subsection{A full large deviation principle}\label{sec:full-ldp}
Combining the exact upper-tail rate with the quadratic lower-tail bound
gives a full LDP on the entire state space of the spectral radius.

\begin{theorem}[Full spectral-radius large deviation principle]\label{thm:full-ldp}
Suppose that $\mu$ satisfies either set of hypotheses in
Theorem~\ref{u:thm:matching} and has a bounded density with respect to
Lebesgue measure on $\F$. Then $(\rho(X_n))_{n\ge1}$ satisfies a large
deviation principle on $[0,\infty)$ with speed $n$ and good rate function
\begin{equation}\label{eq:full-ldp-rate}
 J_\beta(r)=
 \begin{cases}
  +\infty, & 0\le r<1,\\[2pt]
  \dfrac\beta2\bigl(r^2-1-2\log r\bigr), & r\ge1.
 \end{cases}
\end{equation}
Precisely, for every Borel set $A\subseteq[0,\infty)$,
\begin{equation}\label{eq:full-ldp-borel}
\begin{split}
 -\inf_{r\in A^\circ}J_\beta(r)
 &\le \liminf_{n\to\infty}\frac1n\log\Prob\{\rho(X_n)\in A\}\\
 &\le \limsup_{n\to\infty}\frac1n\log\Prob\{\rho(X_n)\in A\}
 \le -\inf_{r\in\overline A}J_\beta(r).
\end{split}
\end{equation}
Here interior and closure are relative to $[0,\infty)$, and
$\inf\varnothing=+\infty$.
\end{theorem}

Thus the complex case requires the sharp Laplace bound, zero
pseudovariance, and a bounded planar density, together with centering
and variance normalization. The real case uses symmetry and
Gaussian-dominated even moments with a bounded density; a real sharp
Laplace bound alone is not asserted to suffice. The value $J_\beta(r)=+\infty$
for $r<1$ records superexponential decay at speed $n$. It does not
identify the finer lower-tail rate at speed $n^2$.

\begin{proof}
Write $R_n=\rho(X_n)$. The hypotheses of Theorem~\ref{u:thm:matching}
imply the common subgaussian assumptions, so Theorem~\ref{l:thm:MT}
also applies. For each fixed $a\in[0,1)$, it gives
\begin{equation}\label{eq:ldp-superexponential}
 \limsup_{n\to\infty}\frac1n\log\Prob\{R_n\le a\}=-\infty.
\end{equation}
Together with the upper-tail estimate \eqref{u:eq:matching}, this implies
$R_n\to1$ in probability. Also, $I_\beta$ is continuous and strictly
increasing on $(1,\infty)$, tends to zero at $1$, and tends to infinity
at infinity. Consequently $J_\beta$ is lower semicontinuous and has
compact sublevel sets.

For the open-set lower bound, let $O\subseteq[0,\infty)$ be open.
If $1\in O$, convergence in probability gives $\Prob\{R_n\in O\}\to1$.
For any $r\in O\cap(1,\infty)$, choose $1<a<r<b$ with $(a,b)\subseteq O$.
By \eqref{u:eq:matching},
\[
 \Prob\{a<R_n<b\}
 =\Prob\{R_n>a\}-\Prob\{R_n\ge b\}
 =\exp\{-nI_\beta(a)+o(n)\},
\]
since $I_\beta(b)>I_\beta(a)$ makes the second term exponentially
smaller than the first. Letting $a\uparrow r$ yields the lower bound
$-J_\beta(r)$; taking the supremum over such $r$, and including $r=1$
when present, gives $-\inf_O J_\beta$. If $O\cap[1,\infty)=\varnothing$,
the required lower bound is $-\infty$ and is automatic.

For the closed-set upper bound, let $F\subseteq[0,\infty)$ be closed.
The bound is automatic if $1\in F$ or $F=\varnothing$. Otherwise,
closedness separates $F$ from $1$, so choose $a\in[0,1)$ with
$F\cap[0,1)\subseteq[0,a]$. If $F\cap(1,\infty)\ne\varnothing$,
its minimum $b$ exists and is strictly larger than $1$.
Then $F\subseteq[0,a]\cup[b,\infty)$, and the union bound gives
\[
 \Prob\{R_n\in F\}
 \le\Prob\{R_n\le a\}+\Prob\{R_n\ge b\}.
\]
Equations~\eqref{eq:ldp-superexponential} and~\eqref{u:eq:matching}
give the upper bound $-I_\beta(b)=-\inf_F J_\beta$.
If $F\cap(1,\infty)=\varnothing$, omit the second term and obtain
$-\infty=-\inf_F J_\beta$ from \eqref{eq:ldp-superexponential}.
This treats unbounded closed sets as well. Applying these two bounds
to $A^\circ\subseteq A\subseteq\overline A$ proves
\eqref{eq:full-ldp-borel}.
\end{proof}

\subsection{Local complex eigenvalue deviations}
The next theorem gives the local form of the complex upper-tail estimate.
For complex Ginibre matrices, the corresponding disk rate is a special
case of \cite[Theorem 2.2, ellipticity parameter $\tau=0$]{ByunLeeOh}.
Here the entry law may be non-Gaussian and atomic.

\begin{theorem}[Local eigenvalue deviations in the complex exterior]\label{u:thm:complex-local}
Suppose $\F=\C$ and \eqref{u:eq:basic}--\eqref{u:eq:properness} hold.
For $z\in\C$ and $\varepsilon>0$, write
$B(z,\varepsilon)=\{w\in\C:|w-z|<\varepsilon\}$ and let
\[
 \mathcal E_n(z,\varepsilon)
 =\{\operatorname{spec}(X_n)\cap B(z,\varepsilon)\ne\varnothing\}.
\]
For every $|z|>1$ and $0<\varepsilon<|z|-1$,
\begin{equation}\label{u:eq:complex-local-lower}
 \liminf_{n\to\infty}\frac1n\log\Prob(\mathcal E_n(z,\varepsilon))
 \ge-I_2(|z|-\varepsilon).
\end{equation}
If, in addition, \eqref{u:eq:sharp} holds, then
\begin{equation}\label{u:eq:complex-local}
 \lim_{n\to\infty}\frac1n\log\Prob(\mathcal E_n(z,\varepsilon))
 =-I_2(|z|-\varepsilon).
\end{equation}
In particular, under the sharp bound,
\[
 \lim_{\varepsilon\downarrow0}\lim_{n\to\infty}
 \frac1n\log\Prob(\mathcal E_n(z,\varepsilon))=-I_2(|z|).
\]
\end{theorem}

The event concerns an actual eigenvalue near the prescribed complex point,
not merely a small singular value of $X_n-zI$.
The result also applies to the non-rotationally-invariant complex Bernoulli
laws in Subsection~\ref{sec:examples}. Its proof is given in Subsection~\ref{u:sec:complex-local}.

\subsection{Examples and scope}\label{sec:examples}
The examples below distinguish the hypotheses needed for the two tails.
We first give laws satisfying the matching upper-tail assumptions of
Theorem~\ref{u:thm:matching}; a bounded density additionally gives
Theorems~\ref{l:thm:MT} and~\ref{thm:full-ldp}.
Theorem~\ref{u:thm:complex-local} also applies to the complex examples.
The final remark explains why ordinary subgaussianity alone does not
give the Gaussian upper-tail rate.

\paragraph{Bounded-density examples: both tails and the full LDP.}
\textbf{Real uniform law.}
The uniform law on $[-\sqrt3,\sqrt3]$ satisfies the real moment assumptions:
more generally, every
symmetric variance-one law supported on that interval satisfies
$\E\xi^{2m}\le3^{m-1}\le(2m-1)!!$ for $m\ge1$.
The uniform law has a bounded density, as does the following example.

\smallskip\noindent\textbf{A real continuous density.}
A non-Gaussian example is
\[
 f(x)=\frac{3}{4\sqrt5}\max\{1-x^2/5,0\},\qquad
 \E\xi^{2m}=\frac{3\cdot5^m}{(2m+1)(2m+3)}\le(2m-1)!!.
\]
Indeed, the density is symmetric and integrates to one, and the ratio of
successive even moments is $5(2m+1)/(2m+5)\le2m+1$, with $\E\xi^2=1$.

\smallskip\noindent\textbf{Complex uniform disk law.}
The uniform law on the disk of radius $\sqrt2$ has bounded planar density
$1/(2\pi)$. Write $\xi=Re^{i\Theta}$, where $\Theta$ is uniform on
$[0,2\pi)$ and independent of $R$. Then $\E R^2=1$ and
$\E R^{2m}=2^m/(m+1)\le m!$ for $m\ge1$. Angular averaging gives
centering, \eqref{u:eq:properness}, and
\[
 \E e^{2\operatorname{Re}(\bar u\xi)}
 =\sum_{m\ge0}\frac{\E R^{2m}|u|^{2m}}{(m!)^2}\le e^{|u|^2}.
\]
The same calculation applies to any independent radial variable $R\ge0$
with $\E R^2=1$ and $\E R^{2m}\le m!$.

\paragraph{Discrete and singular examples: upper-tail results.}
The Bernoulli laws, Steinhaus law, and radial mixture below have no
Lebesgue density on $\F$. Thus the density assumption of our lower-tail
and full-LDP theorems is not met.

\smallskip\noindent\textbf{Real Bernoulli signs.}
For $\Prob(\xi=1)=\Prob(\xi=-1)=1/2$, symmetry and
$\E\xi^{2m}=1\le(2m-1)!!$ verify the real assumptions.

\smallskip\noindent\textbf{Complex Bernoulli laws: no rotational invariance.}
Let $\varepsilon_1,\varepsilon_2$ be independent
symmetric signs. The complex Bernoulli law
$\xi=(\varepsilon_1+i\varepsilon_2)/\sqrt2$ and its rotation uniform on
$\{1,-1,i,-i\}$ are centered, have variance one, and satisfy
\eqref{u:eq:properness}, although neither law is rotationally invariant.
Writing $u=a+ib$, we have
\[
 \E e^{2\operatorname{Re}(\bar u\xi)}
 =\cosh(\sqrt2a)\cosh(\sqrt2b)\le e^{a^2+b^2},
\]
which proves \eqref{u:eq:sharp}; deterministic rotations preserve all these
conditions. The same factorization works for $(U+iV)/\sqrt2$ whenever
$U,V$ are independent centered variance-one real variables satisfying
$\E e^{tU},\E e^{tV}\le e^{t^2/2}$ for every $t\in\R$.

\smallskip\noindent\textbf{Singular rotationally invariant complex laws.}
The radial calculation above covers the \emph{Steinhaus law}
(uniform on the unit circle, $R=1$) and the \emph{radial mixture}
$\Prob(R=0)=\Prob(R=\sqrt2)=1/2$, for which
$\E R^{2m}=2^{m-1}\le m!$. Both laws are singular with respect to
planar Lebesgue measure.

\paragraph{Sums of independent signs.}
Both Bernoulli examples are special cases of the normalized sign sum
\[
 \xi=\sum_{j\ge1}c_j\varepsilon_j,\qquad
 c_j\in\R^\beta,\qquad
 \sum_{j\ge1}c_jc_j^{\mathsf T}=\beta^{-1}\mathrm{Id}_\beta,
\]
where $(\varepsilon_j)_{j\ge1}$ are independent symmetric signs:
$\Prob(\varepsilon_j=1)=\Prob(\varepsilon_j=-1)=1/2$.
We identify $\F$ with $\R^\beta$ and interpret the sum in $L^2$.
Real Bernoulli uses $c_1=1$; complex Bernoulli uses
$c_1=(1,0)/\sqrt2$ and $c_2=(0,1)/\sqrt2$, with all other coefficients zero.
For a finite sum, fix $a\in\R^\beta$ and an integer $m\ge1$.
Expand the $2m$-th moment of the real projection $\langle a,\xi\rangle$:
only even multiplicities survive, and replacing signs by independent
standard real Gaussians increases every surviving term. Thus
\[
 \E\langle a,\xi\rangle^{2m}
 \le(2m-1)!!\left(\frac{\norm a^2}{\beta}\right)^m,
 \qquad
 \E e^{\langle a,\xi\rangle}
 =\E\cosh\langle a,\xi\rangle
 \le e^{\norm a^2/(2\beta)}.
\]
The last bound follows by expanding $\cosh$.
For infinite sums, Fatou's lemma along an almost surely convergent subsequence
of the $L^2$ partial sums gives the same moment bounds; symmetry and
covariance pass to the limit. These bounds give the real moment
condition and, with $a=2(\operatorname{Re}u,\operatorname{Im}u)$,
the complex condition \eqref{u:eq:sharp}; covariance normalization gives
variance one and \eqref{u:eq:properness} in the complex case.
Finite sign sums are atomic. Infinite sign sums need not be singular;
if their law has a bounded density, the lower-tail bound and full LDP
also apply.

\begin{remark}[Why ordinary subgaussianity does not determine the upper-tail rate]\label{u:rem:ordinary-subgaussian}
The sharp assumptions in Theorem~\ref{u:thm:matching} cannot in general
be replaced by ordinary subgaussianity. A single large diagonal entry
can create an outlier at a cost smaller than the Gaussian rate $I_\beta$.
The following example makes this explicit.

Let $g$ be a standard real or proper complex Gaussian with $\E|g|^2=1$,
and set $\xi=\sqrt V\,g$, where $V$ is independent of $g$, with
$\Prob\{V=4/7\}=7/8$ and $\Prob\{V=4\}=1/8$. Since $\E V=1$, this entry law is centered and has variance one. It has
a bounded smooth density, is subgaussian, and has zero pseudovariance
in the complex case. The variance-$4$ component makes large entries
less costly than under the standard Gaussian law: for every fixed $b>2$,
\[
 \Prob\{|\xi_{11}/\sqrt n-b|\le n^{-1/4}\}
 \ge\exp\bigl[-\beta b^2n/8+o(n)\bigr].
\]
On this event, the normalized matrix has a diagonal entry close to $b$.
Replace this entry by an independent copy to obtain an iid matrix $Y_n$.
Then $X_n=Y_n+b e_1e_1^*+\mathcal E_n$, where $e_1$ is the first coordinate
vector and $\|\mathcal E_n\|_{\op}\to0$ in conditional probability.
Lemma~\ref{u:lem:outlier} gives an eigenvalue near $b$, again with
conditional probability tending to one. Since $b>2$, this gives a lower
bound for $\Prob\{\rho(X_n)\ge2\}$. Letting $b\downarrow2$ yields
\[
 \liminf_{n\to\infty}\frac1n\log\Prob\{\rho(X_n)\ge2\}
 \ge-\frac\beta2
 >-\beta\left(\frac32-\log2\right)=-I_\beta(2).
\]
Thus the upper tail can be strictly larger, on the exponential scale,
than the Gaussian prediction. The universal lower bound in
Theorem~\ref{u:thm:lower} still applies, but need not be sharp.
For Wigner matrices, Cook--Ducatez--Guionnet \cite{CookDucatezGuionnet}
established full speed-$n$ LDPs for the largest eigenvalue beyond sharp
subgaussianity, when the log-Laplace transform of the entry law has bounded
second derivative. We expect an analogous full speed-$n$ LDP for the
spectral radius in the iid non-Hermitian setting beyond the sharp class.
We do not yet have a conjectural formula for the law-dependent rate in
this generality.
\end{remark}

\subsection{The two proof mechanisms}\label{sec:proof-overview}
Theorems~\ref{u:thm:lower} and~\ref{l:thm:MT} involve different scales:
an upper deviation can be produced by one outlying eigenvalue, whereas
a lower deviation confines the whole spectrum. We now explain the main
constructions behind these two bounds, before developing their proofs
in \hyperref[part:upper]{Part~\ref*{part:upper}} and \hyperref[part:lower]{Part~\ref*{part:lower}}.

\textbf{Creating an outlier by a soft eigenvector tilt.}
To prove Theorem~\ref{u:thm:lower}, we seek a change of measure under
which an eigenvalue outside the disk of radius $r>1$ becomes likely,
at a logarithmic cost of order $nI_\beta(r)$. Fix a target $b>r$ and a
unit vector $v$. The eigenvector equation would be $(X_n-bI)v=0$;
we favor matrices with a small residual by weighting them with
\[
 \exp\left\{-\frac{\beta n}{2\eta^2}
                   \|(X_n-bI)v\|^2\right\},\qquad \eta>0,
\]
and normalizing to obtain a probability law. This is the soft eigenvector
tilt: the parameter $\eta$ controls the penalty for failing the eigenvector
equation. We average over suitably flat directions, whose coordinates
are individually small, so that each observation combines many independent
entries. Since the weight is positive, this change of measure preserves
the entry support and applies also to discrete laws.
For related work on non-Hermitian spherical integrals, see
Bousseyroux--Potters \cite{BousseyrouxPottersR,BousseyrouxPottersPhase}.

A small residual alone does not ensure a nearby eigenvalue for a
nonnormal matrix. To obtain an actual outlier, sequential coupling
expresses the tilted matrix as an iid matrix with a deterministic right
deformation, a rank-one signal, and an error small in operator norm.
Smooth replacement follows the Lindeberg principle \cite{Chatterjee2006},
and matrix Freedman \cite{Tropp} controls the coupling error.
Stability of the exterior resolvent then turns this representation into
an eigenvalue near $b$; the required deformation estimate is
Lemma~\ref{u:lem:outlier}. Finally, the normalization and likelihood cost
of the tilt give the original-law probability bound. Sending the smoothing
parameter to zero after $n\to\infty$, and then $b\downarrow r$, yields
$I_\beta(r)$.
For matching upper bounds, the complex argument controls the annealed
soft spherical integral, while the real argument compares moments of
matrix powers with the Gaussian case. These are distinct uses of the
two hypotheses in Theorem~\ref{u:thm:matching}.

\Needspace{11\baselineskip}
\textbf{Accumulating precision through Arnoldi links.}
For the lower tail, decompose the unnormalized matrix $\mathsf X_n=\sqrt n X_n$ into
an $(n-1)\times(n-1)$ core and an independent bordering row and column.
The classical Arnoldi iteration \cite{Arnoldi1951} is run on the core
with an additional independent starter; its Gaussian Hessenberg
structure is related to \cite{KV2014}.
Writing $R_j$ for the Arnoldi residual lengths, the relevant statistic is
\[
 W_k=\prod_{j=1}^{k-1}\left(\frac{R_j}{\sqrt n}\right)^{\min\{j,k-j\}},
 \qquad \sum_{j=1}^{k-1}\min\{j,k-j\}=\frac{k^2}{4},
 \quad k\ \text{even}.
\]
We take $k$ proportional to $n$, so the total weight has order $n^2$.
A characteristic-polynomial evaluation determinant converts a sufficiently
large link product into
a conditional confinement bound of the form $\exp(-ck^2+Cn)$.
Only the $2n-1$ border entries are Gaussianized in this step.
To show that the core has the required links outside an event of
probability $e^{-cn^2}$, we compare adaptive orthonormal observations
under the original and Gaussian laws. A coordinate deletion argument
controls the concentration of the query frame; tilted fiber-density
bounds and a sequential transport coupling yield the weighted comparison.
The reduction is given in Subsection~\ref{l:sec:real_main};
Section~\ref{sec:comparison-links} proves the original-law link estimate.

Both proofs couple variables as they are revealed sequentially and control a
change of measure without an additive error larger than the probability
of interest. Their couplings serve different purposes: \hyperref[part:upper]{Part~\ref*{part:upper}}
constructs a tilted matrix with an outlier, whereas
\hyperref[part:lower]{Part~\ref*{part:lower}} transfers rare observation events between entry laws.

\Needspace{12\baselineskip}
\subsection{Organization}
\hyperref[part:upper]{Part~\ref*{part:upper}}
(Sections~\ref{sec:upper-proofs}--\ref{u:sec:technical-proofs}) proves
the upper-tail results. In \hyperref[part:lower]{Part~\ref*{part:lower}},
Sections~\ref{sec:lower-reduction}--\ref{sec:comparison-links} prove the
real lower-tail bound, and Section~\ref{l:sec:complex} gives the complex
extension.

\subsection{Generative AI usage and Lean formalization.}\label{lean}
The author used ChatGPT Codex to assist with some technical arguments and
manuscript preparation. Several constructions and observations in the
lower-tail proof were suggested by Codex. The author has checked the
arguments and takes full responsibility for their correctness.
The Lean formalizations are complete.\footnote{Upper tails:
\url{https://github.com/hanyi162013-Yihan/spectral-radius-upper-tail-lean}.
Lower tails:
\url{https://github.com/hanyi162013-Yihan/spectral-radius-lower-tail-lean}.}

\part{Upper tails}\label{part:upper}
\section{Outlier construction and the universal lower bound}\label{sec:upper-proofs}
The lower bound for an upper-tail probability has two steps: make an
outlier typical under a tilted law, then compute the cost of returning
to the original law. This section carries out these steps and proves
Theorem~\ref{u:thm:lower}. Section~\ref{sec:upper-matching} proves the
matching upper bounds and the local complex result. The supporting
coupling and Schur calculations are collected in
Section~\ref{u:sec:technical-proofs}.

\subsection{External inputs and common estimates}\label{u:sec:inputs}
We retain the setting of Section~\ref{sec:main-results}.
The exterior resolvent estimates below support the outlier construction.
The bulk determinant estimates have failure probability $Ce^{-cn^2}$,
which stays negligible under the later change of measure, whose density
is bounded by $e^{O(n)}$ for fixed parameters.
All finite-dimensional manipulations and new reductions used below are
proved here. The following established results are external inputs.
We use the deterministic bilinear coefficient estimate from
\cite[Section 2 and Lemma 2.3]{Tao}, and
the classical Bai--Yin fixed-power norm upper bound \cite[Theorem 2.1]{BaiYin};
its precise weak form is also recorded in \cite[Section 1.3, (3)]{Xiang}.
The original Bai--Yin proof is written for real entries. For complex entries,
use the conjugate transpose in the Gram trace and bound each mixed scalar
moment by its absolute moment. Singleton cancellation, the directed-entry
graph count, and the truncation argument are unchanged; no independence
between real and imaginary parts is required.
In particular, for compact
$K\subset\{|z|>1\}$ there are deterministic $M_K,M<\infty$ such that
\begin{equation}\label{u:eq:resolvent}
 \begin{gathered}
 \Prob\left\{\sup_{z\in K}\norm{(zI-X_n)^{-1}}\le M_K,
                 \ \norm{X_n}\le M\right\}\longrightarrow1,\\
 \sup_{z\in K}\left|p_n^*(zI-X_n)^{-1}q_n-\frac{p_n^*q_n}{z}\right|
       \xrightarrow{\Prob}0
 \end{gathered}
\end{equation}
for any deterministic unit-vector sequences $p_n,q_n\in\F^n$.
The same inputs give, uniformly on each such compact $K$,
\begin{equation}\label{u:eq:exterior-logdet}
 \frac1n\log|\det(zI-X_n)|\longrightarrow\log|z|
 \quad\hbox{in probability}.
\end{equation}
Here is a common derivation of \eqref{u:eq:resolvent}--\eqref{u:eq:exterior-logdet}.
For each fixed integer $k\ge1$, the cited inputs give
$p_n^*X_n^kq_n\xrightarrow{\Prob}0$ and
$\Prob\{\norm{X_n^k}\le k+2\}\to1$.
Also $\E|n^{-1}\tr X_n^k|^2\le C_k/n^2$: expand into two closed
directed walks. A nonzero term has at most $k$ distinct directed entries,
each occurring at least twice, and at most $k$ vertices since each has
an outgoing edge. There are $O_k(n^k)$ labelings, with bounded moments
and normalization $n^{-k-2}$.

Choose $1<r_1<r:=\inf_K|z|$ and a fixed integer $m$ with $m+2<r_1^m$.
With probability tending to one, $\norm{X_n}\le M$ and
\[
 \norm{X_n^{am+b}}\le r_1^{am}M^b\quad(a\ge0,\ 0\le b<m).
\]
On this event $\rho(X_n)<r$, and both series
\begin{align*}
 (zI-X_n)^{-1}&=\sum_{k\ge0}\frac{X_n^k}{z^{k+1}},\\
 \frac1n\log|\det(zI-X_n)|
 &=\log|z|-\operatorname{Re}\sum_{k\ge1}\frac{\tr X_n^k}{nkz^k}
\end{align*}
have uniformly geometric tails on $K$. The first series bounds the
resolvent norm. In its bilinear forms and in the second series, each
fixed term with $k\ge1$ tends to zero by the preceding estimates.
First take $n\to\infty$ at fixed truncation order, then send that order
to infinity. This proves both claims uniformly on $K$.

We also use Talagrand's product-space convex-distance inequality
\cite[Theorem 4.1.1 and Lemma 4.1.2]{Talagrand}, Tropp's rectangular matrix
Freedman inequality \cite[Corollary 1.3]{Tropp}, and the real Gaussian
right-tail upper bound \cite[Theorem 1.2]{XuZeng}.
The real Schur change of variables
\cite[Theorem 5.1, Lemma 5.2 and Section 6.1]{Edelman}
is derived in Subsection~\ref{u:sec:schur-distribution}.

\begin{lemma}[Convex concentration for independent blocks]\label{u:lem:concentration}
Let $Z_1,\ldots,Z_{n^2}$ be iid vectors in a fixed Euclidean space with
$\E e^{c\norm{Z_1}^2}<\infty$. For fixed $L,\delta>0$, every convex
$L/n$-Lipschitz function $f_n$ satisfies
\[
 \Prob\{|f_n(Z)-\E f_n(Z)|>\delta\}\le C e^{-c_\delta n^2}
\]
for all sufficiently large $n$, uniformly over such functions.
Here $C>0$ is universal. With $L$ and the block law fixed,
$c_\delta>0$ depends only on $\delta$, independently of $n$ and $f_n$;
dependence on these fixed data is suppressed in the notation.
\end{lemma}
\begin{proof}
Set $Z_j^{(R)}=Z_j\mathbf1_{\norm{Z_j}\le R}$. For fixed $R$,
Talagrand's inequality \cite[Theorem 4.1.1 and Lemma 4.1.2]{Talagrand},
with whole blocks as product coordinates, gives the stated bound for
$f_n(Z^{(R)})$ at any fixed deviation after the standard mean--median
comparison. No independence within a block is needed. Moreover,
\[
 |f_n(Z)-f_n(Z^{(R)})|\le
 \Delta_n:=\frac L n\left(\sum_{j=1}^{n^2}\norm{Z_j}^2
                  \mathbf1_{\norm{Z_j}>R}\right)^{1/2}.
\]
For $0<\theta<c$, dominated convergence gives
$\log\E e^{\theta\norm{Z_1}^2\mathbf1_{\norm{Z_1}>R}}\to0$
as $R\to\infty$. Thus a fixed large $R$ ensures
$\E\Delta_n\le\delta/4$ and, by Chernoff's bound,
$\Prob\{\Delta_n>\delta/4\}\le e^{-c_\delta n^2}$, uniformly in $n$.
Apply the bounded-block result at deviation $\delta/2$ and use the
triangle inequality, including the centering error
$|\E f_n(Z)-\E f_n(Z^{(R)})|\le\E\Delta_n$.
\end{proof}

Define, for $\tau>0$,
\[
 F_{\tau,n}(z)=\frac1n\log\det((X_n-zI)^*(X_n-zI)+\tau^2 I).
\]
\begin{lemma}[Regularized bulk events]\label{u:lem:bulk}
Under \eqref{u:eq:basic}, for each $b\in\C$ with $|b|>1$ and fixed $\tau,\delta>0$,
\begin{equation}\label{u:eq:bulk-lower}
 \Prob\{F_{\tau,n}(b)<2\log|b|-\delta\}\le Ce^{-cn^2}.
\end{equation}
For each compact annulus $K\subset\{|z|>1\}$, a constant $M_K$
depending only on $\mu,K$ satisfies, for fixed $\tau,\delta>0$,
\begin{equation}\label{u:eq:bulk-upper}
 \Prob\left\{\sup_{z\in K}[F_{\tau,n}(z)-2\log|z|]
                   >M_K^2\tau^2+\delta\right\}\le Ce^{-cn^2}.
\end{equation}
Here $c,C>0$ are independent of $n$, but may depend on
$\mu,b,\tau,\delta$ in \eqref{u:eq:bulk-lower} and on
$\mu,K,\tau,\delta$ in \eqref{u:eq:bulk-upper}; they need not be the same
in the two bounds.
\end{lemma}
\begin{proof}
Write $\log(s^2+\tau^2)=2\log\tau+h_1(s)-h_2(s)$, $s\ge0$, with
$h_1(0)=h_2(0)=0$ and
\[
 h_1'(s)=\begin{cases}2s/(s^2+\tau^2),&s\le\tau,\\1/\tau,&s>\tau,
 \end{cases}\qquad h_2'(s)=h_1'(s)-2s/(s^2+\tau^2).
\]
Both functions are convex, nondecreasing and $1/\tau$-Lipschitz.
Write $s_1(A)\ge\cdots\ge s_n(A)\ge0$ for the singular values of $A$.
To see convexity of the spectral sums, use
\[
 \sum_j(s_j(A)-a)_+
 =\max_{0\le k\le n}\left(\sum_{j=1}^k s_j(A)-ka\right),\qquad
 \sum_{j=1}^k s_j(A)=\max_{U^*U=V^*V=I_k}\operatorname{Re}\tr(U^*AV).
\]
Here $(x)_+=\max(x,0)$ and $U,V$ have $k$ orthonormal columns.
Each $h_\ell$ is a nonnegative integral of the functions $(s-a)_+$,
since $h_\ell'(0)=0$ and $h_\ell''\ge0$.
The singular-value inequality
$\sum_j|s_j(A)-s_j(B)|^2\le\norm{A-B}_{\HS}^2$ also shows that
$n^{-1}\sum_jh_\ell(s_j(X_n-zI))$ is $1/(n\tau)$-Lipschitz in the
original entries.
Lemma~\ref{u:lem:concentration} therefore gives concentration of
$F_{\tau,n}(z)$ at fixed deviations with failure $Ce^{-cn^2}$.

By \eqref{u:eq:exterior-logdet}, the deterministic floor
$F_{\tau,n}\ge2\log\tau$ and the inequality
\[
 F_{\tau,n}(b)\ge2n^{-1}\log|\det(bI-X_n)|
\]
yield
$\liminf_n\E F_{\tau,n}(b)\ge2\log|b|$.
On the resolvent event \eqref{u:eq:resolvent},
\[
 F_{\tau,n}(z)\le2n^{-1}\log|\det(zI-X_n)|+M_K^2\tau^2.
\]
Passing to expectations is justified by the common lower floor and the
uniformly square-integrable upper envelope
\[
 F_{\tau,n}(z)\le
 \log\left(2n^{-2}\sum_{i,j=1}^n|\xi_{ij}|^2+2\sup_K|z|^2+\tau^2\right).
\]
This proves the required upper and lower mean estimates. Concentration
gives \eqref{u:eq:bulk-lower}. Finally
$|F_{\tau,n}(z)-F_{\tau,n}(w)|\le|z-w|/\tau$; a fixed finite net
of $K$ gives \eqref{u:eq:bulk-upper}.
\end{proof}

The next estimate is used in the change of measure and the complex upper
bound; its derivative bounds also enter the coupling proof. The proof
is given in Subsection~\ref{u:sec:coupling-details}.
The replacement estimate uses the Lindeberg principle: matching first
and second moments cancel the first two Taylor terms, leaving a cubic
remainder. See Chatterjee \cite{Chatterjee2006} for a general smooth
invariance principle. We give the Gaussian-kernel and score estimates
needed here explicitly.

\begin{lemma}[Smooth Gaussian replacement and score]\label{u:lem:soft-replacement}
Assume \eqref{u:eq:basic} and, in the complex case, \eqref{u:eq:properness}.
Let $g_j$ be independent standard real or proper complex Gaussians
with $\E|g_j|^2=1$,
and let $\xi_j$ be independent copies of the original entry. Take a finite,
possibly empty coefficient family $(v_j)\subset\F$ with $q:=\sum_j|v_j|^2\le1$,
and set $\epsilon_3:=\sum_j|v_j|^3$. For $a>0$, define
\[
 \mathcal Z_\mu(s)=\E e^{-|s-\sum_jv_j\xi_j|^2/a},\qquad
 \mathcal Z_G(s)=\E e^{-|s-\sum_jv_jg_j|^2/a},\qquad d_{\mathrm G}=a+2q/\beta.
\]
All derivatives are with respect to $s\in\F\simeq\R^\beta$.
In particular, $\nabla$ is the \emph{real gradient}, identified with an
element of $\F$; when $s=x+iy$, $\nabla f=\partial_xf+i\partial_yf$.
With $m_3=\E|\xi|^3+\E|g_1|^3$, one has
\begin{equation}\label{u:eq:smooth-replacement}
 \sup_{s\in\F}\bigl(|\mathcal Z_\mu(s)-\mathcal Z_G(s)|
       +|\nabla\mathcal Z_\mu(s)-\nabla\mathcal Z_G(s)|\bigr)
 \le C_a m_3\epsilon_3.
\end{equation}
Moreover, including $q=0$,
\begin{equation}\label{u:eq:gaussian-soft-formula}
 \mathcal Z_G(s)=\left(\frac{a}{d_{\mathrm G}}\right)^{\beta/2}e^{-|s|^2/d_{\mathrm G}},
 \qquad \nabla\log\mathcal Z_G(s)=-2s/d_{\mathrm G}.
\end{equation}
For each fixed $K<\infty$,
\begin{equation}\label{u:eq:compact-score}
 \sup_{|s|\le K}\left|\nabla\log\mathcal Z_\mu(s)+\frac{2s}{d_{\mathrm G}}\right|
 \le C_{a,K}m_3\epsilon_3.
\end{equation}
The constants are independent of the number and choice of coefficients.
\end{lemma}

\subsection{Support-preserving coupling for upper deviations}\label{u:sec:coupling}
Assume \eqref{u:eq:basic} and, in the complex case, \eqref{u:eq:properness}.
The smoothing parameter $\eta$ is kept fixed while $n\to\infty$ and
sent to zero afterwards.
Fix $\eta>0$, $L\ge1$, $T_0\ge0$, and a deterministic unit vector
$v\in\F^n$ with $\max_j|v_j|\le L/\sqrt n$. For $|t|\le T_0$, define
the scalar normalizer
\[
 \mathcal Z^{\mathrm{row}}_v(t):=\int_{\F^n}
 \exp\left[-\frac\beta{2\eta^2}\left|\sum_{j=1}^n v_jx_j-t\right|^2\right]
 \,\mu^{\otimes n}(\dd x)
\]
and the tilted row law on $\F^n$ by
\begin{equation}\label{u:eq:row-tilt}
 \frac{d\mu_{v,t}}{d\mu^{\otimes n}}(x)
 =\mathcal Z^{\mathrm{row}}_v(t)^{-1}\exp\left(-\frac\beta{2\eta^2}
                  \left|\sum_{j=1}^nv_jx_j-t\right|^2\right).
\end{equation}
Its entries need not be independent. Dependence on $\eta,\beta,n$ is
suppressed in $\mu_{v,t}$ and $\mathcal Z^{\mathrm{row}}_v(t)$. Jensen gives
$e^{-\beta(1+T_0^2)/(2\eta^2)}\le \mathcal Z^{\mathrm{row}}_v(t)\le1$,
so the row density is uniformly bounded.
Constants in the coupling estimates may depend on $\mu,\eta,L,T_0$,
but not on $n,j,v,t$.

For the outlier construction we will take $t_i=\sqrt n\,b v_i$,
with $b\in\F$. The product of the row weights is then the soft
eigenvector weight $e^{-\beta n\|(X-bI)v\|^2/(2\eta^2)}$ for
$X=n^{-1/2}(x_{ij})$.
With this choice, the rank-one term in the coupling below is
$bvv^*/(1+\eta^2)$; we choose $|b|/(1+\eta^2)>1$ to place its outlier
outside the unit disk. Flatness of $v$ ensures $|t_i|\le |b|L$.

\begin{lemma}[Triangular coupling]\label{u:lem:coupling}
Let $v\in\F^n$ be a deterministic unit vector satisfying
\[
 \max_{1\le j\le n}|v_j|\le\frac{L}{\sqrt n}.
\]
Let $\bm t=(t_1,\ldots,t_n)^{\mathsf T}\in\F^n$ be deterministic with
$|t_i|\le T_0$, and let
$\bm x_i=(x_{ij})_{j=1}^n\sim\mu_{v,t_i}$ be independent rows as in \eqref{u:eq:row-tilt}.
There is a coupling to an iid array
$(\zeta_{ij})$ of law $\mu$ such that
\begin{equation}\label{u:eq:coupling}
 n^{-1/2}(x_{ij})=Y_nT+\frac{n^{-1/2}\bm t v^*}{1+\eta^2}+E_n,
 \quad Y_n=n^{-1/2}(\zeta_{ij}),\quad\norm{E_n}_{\op}\xrightarrow{\Prob}0
\end{equation}
uniformly over the specified deterministic parameters.
The error $E_n$ may depend on $Y_n$ and need not have independent entries.
Here, with $d_j=\eta^2+\sum_{k=j}^n|v_k|^2$ for $1\le j\le n+1$
(so $d_{n+1}=\eta^2$),
\begin{equation}\label{u:eq:T}
 T_{jj}=1,\quad T_{kj}=-v_k\bar v_j/d_{k+1}\ (k<j),\quad
 T_{kj}=0\ (k>j),\qquad \norm{T-I}_{\HS}^2\le(2\eta^4)^{-1}.
\end{equation}
\end{lemma}
\paragraph{Proof outline.} Reveal each tilted row sequentially and couple its conditional
entries to independent original entries. The conditional means give an
upper triangular regression system, whose inverse is $T$. The centered
coupling errors form a matrix martingale; truncation and matrix Freedman
control their operator norm, while the regression remainder is controlled
in Hilbert--Schmidt mean square.
The complete proof is in Subsection~\ref{u:sec:coupling-details}.

The finite-rank resolvent mechanism is classical in outlier theory
\cite{Tao,OR2014,BC2016}; see also \cite{HanBand} for band and elliptic
models. The coupling controls $T-I$ by a uniform Hilbert--Schmidt
bound. The following version treats such right deformations together
with the possibly correlated small error in \eqref{u:eq:coupling}.

\begin{lemma}[Right deformations and actual outliers]\label{u:lem:outlier}
Let $Y_n$ have the same law as $X_n$.
If $D_n$ is deterministic with $\norm{D_n}_{\HS}\le C$, then
$Y_n(I+D_n)$ has the exterior resolvent bounds and bilinear limits
\eqref{u:eq:resolvent}, with constants also depending on $C$.
For every random matrix sequence $\mathcal E_n^{\rm pert}$ satisfying
$\norm{\mathcal E_n^{\rm pert}}_{\op}\xrightarrow{\Prob}0$, they also
hold for $Y_n(I+D_n)+\mathcal E_n^{\rm pert}$.
For each fixed $s\in\F$ with $|s|>1$, the matrix
\[
 Y_n(I+D_n)+s v_nv_n^*+\mathcal E_n^{\rm pert}
\]
has an eigenvalue converging in probability to $s$,
for deterministic unit vectors $v_n$.
No independence from $Y_n$, or among the entries of
$\mathcal E_n^{\rm pert}$, is required.
\end{lemma}
\begin{proof}
Fix a compact set $K\subset\{|z|>1\}$. We use $m$ for a
\emph{truncation rank}, not for a matrix-power exponent; $n$ remains
the matrix dimension. For each fixed $m\ge1$ and $n>m$, truncate the
singular-value decomposition of $D_n$:
\[
 D_n=D_n^{(m)}+E_n^{(m)},\qquad
 D_n^{(m)}=\mathsf U_{n,m}\mathsf V_{n,m},
\]
where $\mathsf U_{n,m}\in\F^{n\times m}$, $\mathsf V_{n,m}\in\F^{m\times n}$,
$\norm{\mathsf U_{n,m}}\le C$, $\norm{\mathsf V_{n,m}}\le1$, and
$\norm{E_n^{(m)}}\le C/\sqrt{m+1}$.
The last bound follows from
$(m+1)s_{m+1}(D_n)^2\le\norm{D_n}_{\HS}^2\le C^2$.
Thus $D_n^{(m)}$ has fixed rank but need not have small norm, whereas
the remainder $E_n^{(m)}$ has small operator norm for large $m$.
Define the baseline and truncated resolvents by
\[
 R(z)=(zI-Y_n)^{-1},\qquad
 R^{(m)}(z)=(zI-Y_n(I+D_n^{(m)}))^{-1}.
\]
Since $R(z)Y_n=zR(z)-I$, the deterministic bilinear limits in
\eqref{u:eq:resolvent} imply
\[
 \mathsf V_{n,m}R(z)Y_n\mathsf U_{n,m}
 =z\mathsf V_{n,m}R(z)\mathsf U_{n,m}-\mathsf V_{n,m}\mathsf U_{n,m},
\]
and, for each fixed $m$,
\[
 \sup_{z\in K}\norm{\mathsf V_{n,m}R(z)Y_n\mathsf U_{n,m}}_{\op}
 \xrightarrow{\Prob}0.
\]
Indeed, all $m^2$ scalar entries converge uniformly to zero in
probability, and $m$ is fixed. With probability tending to one,
$I_m-\mathsf V_{n,m}R(z)Y_n\mathsf U_{n,m}$ is invertible and its
inverse has norm at most $2$, uniformly on $K$. Woodbury gives
\[
 R^{(m)}(z)=R(z)+R(z)Y_n\mathsf U_{n,m}
       (I_m-\mathsf V_{n,m}R(z)Y_n\mathsf U_{n,m})^{-1}\mathsf V_{n,m}R(z).
\]
Its norm is at most $C_0=M_K+2M_K^2MC$, independent of $m$, with
probability tending to one for every fixed $m$. For deterministic
unit vectors $p_n,q_n$, use $R(z)Y_n=zR(z)-I$, bound the middle
inverse in the Woodbury correction by $2$, and use
$\norm{\mathsf V_{n,m}R(z)q_n}\le
\norm{\mathsf V_{n,m}}\norm{R(z)}\norm{q_n}\le M_K$
to obtain, on the above events,
\[
 \sup_{z\in K}|p_n^*(R^{(m)}(z)-R(z))q_n|
 \le 2M_K\sup_{z\in K}\norm{p_n^*(zR(z)-I)\mathsf U_{n,m}}
 \xrightarrow{\Prob}0.
\]
Indeed, the row vector inside the norm has $m$ coordinates, each tending
uniformly to zero by \eqref{u:eq:resolvent}; the columns of
$\mathsf U_{n,m}$ are deterministic and bounded, and $m$ is fixed.

To restore the remainder, write $R_D(z)=(zI-Y_n(I+D_n))^{-1}$ and
$\varepsilon_m=MC/\sqrt{m+1}$. Choose $m$ with $C_0\varepsilon_m<1/2$.
Since $\norm{Y_nE_n^{(m)}}\le\varepsilon_m$ with probability tending
to one, the Neumann series gives, uniformly on $K$,
\[
 R_D=(I-R^{(m)}Y_nE_n^{(m)})^{-1}R^{(m)},\qquad
 \norm{R_D}\le\frac{C_0}{1-C_0\varepsilon_m}\le2C_0.
\]
The resolvent identity then yields
\[
 \sup_{z\in K}\norm{R_D(z)-R^{(m)}(z)}
 =\sup_{z\in K}\norm{R_D(z)Y_nE_n^{(m)}R^{(m)}(z)}
 \le2C_0^2\varepsilon_m.
\]
Consequently, on these events,
\begin{align*}
 \sup_{z\in K}\left|p_n^*R_D(z)q_n-\frac{p_n^*q_n}{z}\right|
 &\le \sup_{z\in K}|p_n^*(R^{(m)}(z)-R(z))q_n|\\
 &\quad+\sup_{z\in K}\left|p_n^*R(z)q_n-\frac{p_n^*q_n}{z}\right|
       +2C_0^2\varepsilon_m.
\end{align*}
For fixed $m$, the first two terms tend to zero in probability as
$n\to\infty$; then $m\to\infty$ removes the last term. This proves
both the resolvent bound and the bilinear limits. Another Neumann
series absorbs $\mathcal E_n^{\rm pert}$ using
$\norm{\mathcal E_n^{\rm pert}}_{\op}\xrightarrow{\Prob}0$, with no
independence requirement.

For the outlier assertion, use the perturbation from the statement and set
$W_n=Y_n(I+D_n)+\mathcal E_n^{\rm pert}$. We seek an outlier of
$W_n+s v_nv_n^*$. On a small closed disk about
$s$ contained in $\{|z|>1\}$, $W_n$ has no resolvent poles and
$1-sv_n^*(zI-W_n)^{-1}v_n\to1-s/z$ uniformly. The determinant lemma
and Rouch\'e's theorem give exactly one eigenvalue in the disk.
For real data it is real by conjugation symmetry. For fixed $K,C,s$,
these probability bounds are uniform over the permitted deterministic
matrices and unit vectors, provided the perturbation error converges
uniformly over these data. This follows by a subsequence contradiction;
the argument does not require simultaneous control over directions
selected from the data.
\end{proof}

\subsection{Likelihood cost and the upper-tail lower bound}
The coupling makes the outlier typical under a tilted law. To return to
the original law, we need the averaged normalizer of the tilt and an
upper bound on its weight for a typical tilted matrix.
Let $\sigma_n$ be normalized Haar measure on the unit sphere in $\F^n$.

The following lemma collects the deterministic sphere estimates used
in the lower and upper bounds for upper-tail probabilities.
\begin{lemma}[Spherical integral identity and bounds]\label{u:lem:spherical-polar}
Write $p=\beta n/2$. We use the gamma function
$\Gamma(s)=\int_0^\infty x^{s-1}e^{-x}\dd x$ for $s>0$, so
$\Gamma(n)=(n-1)!$, and write $\lambda_{\min}(H)$ for the smallest
eigenvalue of $H$. For every $A=A^*>0$ in $\F^{n\times n}$,
\begin{equation}\label{u:eq:spherical-negative-moment}
 \int(v^*Av)^{-p}\dd\sigma_n(v)=(\det A)^{-\beta/2}.
\end{equation}
For $H\ge0$ and $\eta>0$, set
$\mathcal I_\eta(H)=\int e^{-pv^*Hv/\eta^2}\dd\sigma_n(v)$.
Then
\begin{equation}\label{u:eq:spherical-upper}
 \frac1n\log\mathcal I_\eta(H)
 \le\frac\beta2\left(\log\eta^2+\eta^2-1\right)
       -\frac\beta{2n}\log\det(H+\eta^4I).
\end{equation}
In the complex case, for $n\ge2$, one also has
\begin{equation}\label{u:eq:witness}
 \mathcal I_\eta(H)
 \ge\frac{\Gamma(n)}{n^{n-1}}
       \frac{\eta^{2n}}{\det(H+2\eta^2I)}
       e^{-n\lambda_{\min}(H)/\eta^2}.
\end{equation}
\end{lemma}
\begin{proof}
\emph{Identity and upper bound.}
Diagonalization and Gaussian integration give
\[
 \frac{\int_{\F^n}e^{-x^*Ax}\dd x}
      {\int_{\F^n}e^{-\norm{x}^2}\dd x}=(\det A)^{-\beta/2}.
\]
In polar coordinates $x=rv$, the radial integral in the numerator is
$\tfrac12\Gamma(p)(v^*Av)^{-p}$. Dividing by the case $A=I$ proves
\eqref{u:eq:spherical-negative-moment}. Finally, $\log u\le u-1$ gives
$e^{-px}\le e^{p(\eta^2-1)}(x+\eta^2)^{-p}$ for $x\ge0$.
Apply this with $x=v^*Hv/\eta^2$ and use
\eqref{u:eq:spherical-negative-moment} for $A=H+\eta^4I$.

\emph{Lower bound.}
Suppose $\F=\C$ and diagonalize $H$ with eigenvalues
$0\le h_1\le\cdots\le h_n$. Normalizing a standard complex Gaussian
vector gives $w_j=|v_j|^2=Y_j/\sum_iY_i$ for independent rate-one
exponentials $Y_j$. The substitution $Y_j=sw_j$ has Jacobian $s^{n-1}$;
integrating over $s>0$ gives density $\Gamma(n)$ in $(w_2,\ldots,w_n)$ on
$\mathcal S=\{w_j\ge0:\sum_{j=2}^n w_j\le1\}$.
Set $\lambda_j=n(h_j+2\eta^2)/\eta^2$ for $j\ge2$. Then
\begin{align*}
 \mathcal I_\eta(H)
 &=\Gamma(n)e^{-nh_1/\eta^2}
     \int_{\mathcal S}e^{-n\sum_{j=2}^n(h_j-h_1)w_j/\eta^2}\dd w\\
 &\ge\Gamma(n)e^{-nh_1/\eta^2}
     \left(\prod_{j=2}^n\lambda_j^{-1}\right)
     \Prob\left\{\sum_{j=2}^n E_j\le1\right\},
\end{align*}
where $E_j$ are independent exponentials of rates $\lambda_j$.
Since $\lambda_j\ge2n$, Markov's inequality bounds the last probability
below by $1-(n-1)/(2n)\ge1/2$. Finally,
\[
 \frac12\prod_{j=2}^n\lambda_j^{-1}
 =\frac{\eta^{2n}}{n^{n-1}\det(H+2\eta^2I)}
       \frac{h_1+2\eta^2}{2\eta^2}
 \ge\frac{\eta^{2n}}{n^{n-1}\det(H+2\eta^2I)},
\]
which proves \eqref{u:eq:witness}.
\end{proof}

For large fixed $L$, set
\[
 \mathcal F_{n,L}=\{v:\norm v=1,\ \sqrt n\norm v_\infty\le L\},
 \quad q_{n,L}=\sigma_n(\mathcal F_{n,L}),\quad
 \nu_{n,L}=q_{n,L}^{-1}\mathbf1_{\mathcal F_{n,L}}\sigma_n.
\]
Let $g\in\F$ have independent $N(0,1/\beta)$ real coordinates.
Normalizing iid copies gives a Haar direction. Conditional on all copies
having modulus at most $L/2$, their squared moduli have mean $>1/4$ for
large $L$. The law of large numbers makes their Euclidean norm at least
$\sqrt n/2$ with probability tending to one, placing the normalized
direction in $\mathcal F_{n,L}$. Hence
\begin{equation}\label{u:eq:flat-mass}
 q_{n,L}\ge e^{-a_Ln+o(n)},\qquad
 a_L=-\log\Prob(|g|\le L/2)\downarrow0.
\end{equation}
Fix $b\in\F$ with $|b|>1$,
$\eta>0$ and sufficiently large $L$, independently of $n$, and define
\[
 Z_{b,\eta,L}(X)
 =\int e^{-\beta n\norm{(X-bI)v}^2/(2\eta^2)}\dd\nu_{n,L}(v).
\]
\begin{lemma}[Annealed flat-sphere asymptotics]\label{u:lem:flat-annealed}
Under \eqref{u:eq:basic} and, in the complex case, \eqref{u:eq:properness},
for fixed $b,\eta,L$ as above one has
\begin{equation}\label{u:eq:annealed-lower}
 \log\E_{\mathbb P_n}Z_{b,\eta,L}(X_n)=n\mathcal A_{\beta,\eta}(b)+O(\sqrt n),\qquad
 \mathcal A_{\beta,\eta}(b)=\frac\beta2
 \left(\log\frac{\eta^2}{1+\eta^2}-\frac{|b|^2}{1+\eta^2}\right).
\end{equation}
\end{lemma}
\begin{proof}
For $v\in\mathcal F_{n,L}$, set $t_i=\sqrt n\,b v_i$. Then
$\sum_j|v_j|^3\le L/\sqrt n$ and $|t_i|\le |b|L$.
Lemma~\ref{u:lem:soft-replacement}, with $a=2\eta^2/\beta$ and $q=1$,
gives uniformly in $v$ and $i$,
\[
 \mathcal Z^{\mathrm{row}}_v(t_i)=\left(\frac{\eta^2}{1+\eta^2}\right)^{\beta/2}
 e^{-\beta|t_i|^2/(2(1+\eta^2))}\bigl(1+O(n^{-1/2})\bigr).
\]
Its additive error becomes relative because the displayed Gaussian
factor is bounded below by a positive constant for $|t_i|\le |b|L$.
Row independence and $\sum_i|t_i|^2=n|b|^2$ therefore yield
\[
 \E_{\mathbb P_n}Z_{b,\eta,L}(X_n)
 =\int\prod_{i=1}^n \mathcal Z^{\mathrm{row}}_v(t_i)\dd\nu_{n,L}(v)
 =\exp\{n\mathcal A_{\beta,\eta}(b)+O(\sqrt n)\},
\]
since $\nu_{n,L}$ is a probability measure and the error is uniform.
\end{proof}

\begin{proof}[Proof of Theorem~\ref{u:thm:lower}]
Fix $r>1$, $b>r$, $0<\eta^2<b/r-1$, and sufficiently large $L$.
All these parameters remain fixed as $n\to\infty$.
Define the tilted matrix law by
\[
 \frac{d\mathbb Q_n^{b,\eta,L}}{d\mathbb P_n}(X)
 =\frac{Z_{b,\eta,L}(X)}{\E_{\mathbb P_n}Z_{b,\eta,L}(X_n)}.
\]
The joint tilted law of the matrix and the auxiliary direction is
\[
 d\widehat{\mathbb Q}_n^{b,\eta,L}(X,v)
 =\frac{e^{-\beta n\norm{(X-bI)v}^2/(2\eta^2)}}
 {\E_{\mathbb P_n}Z_{b,\eta,L}(X_n)}
 \,d\mathbb P_n(X)\,d\nu_{n,L}(v).
\]
Its matrix marginal is $\mathbb Q_n^{b,\eta,L}$. Under the joint
tilt the direction need not retain the distribution $\nu_{n,L}$;
we first condition on $v$ and then integrate against its tilted marginal.
Conditional on $v$, the unnormalized rows $\sqrt n\,X_{i,\cdot}$ are
independent with laws $\mu_{v,t_i}$ from \eqref{u:eq:row-tilt}, where
$t_i=\sqrt n\,b v_i$. Lemmas~\ref{u:lem:coupling}--\ref{u:lem:outlier} imply
\begin{equation}\label{u:eq:tilted-event}
 \mathbb Q_n^{b,\eta,L}\{\rho(X_n)>r\}\longrightarrow1
 \quad\hbox{for }1<r<b/(1+\eta^2).
\end{equation}
In the real case, the spectral-radius event can be replaced by the
event that a positive real eigenvalue exceeds the same threshold.
The uniformity over deterministic $v$ allows integration against its
tilted marginal.

Since $0<Z_{b,\eta,L}(X)\le1$, Lemma~\ref{u:lem:flat-annealed} transfers any $Ce^{-cn^2}$
exception under $\mathbb P_n$ to an $o(1)$ exception under $\mathbb Q_n^{b,\eta,L}$.
Write $H=(X_n-bI)^*(X_n-bI)$. Restricting the spherical integral and
normalizing by $q_{n,L}$ gives
$Z_{b,\eta,L}(X_n)\le q_{n,L}^{-1}\mathcal I_\eta(H)$.
Apply \eqref{u:eq:bulk-lower} with $\tau=\eta^2$ and deviation $\eta^2$,
so no separate regularization or error parameter is needed.
With $\mathbb Q_n^{b,\eta,L}$-probability tending to one,
$F_{\eta^2,n}(b)\ge2\log b-\eta^2$.
On this event, \eqref{u:eq:flat-mass} and \eqref{u:eq:spherical-upper} yield
\begin{align*}
 \frac1n\log Z_{b,\eta,L}(X_n)
 &\le-\frac1n\log q_{n,L}
   +\frac\beta2\left[\log\eta^2+\eta^2-1-F_{\eta^2,n}(b)\right]\\
 &\le a_L+\frac\beta2\left[\log\eta^2-1-2\log b+2\eta^2\right]
   +o(1).
\end{align*}
Subtracting the normalizing constant from \eqref{u:eq:annealed-lower} now gives
\begin{align}
 \frac1n\log\frac{d\mathbb Q_n^{b,\eta,L}}{d\mathbb P_n}
 &=\frac1n\log Z_{b,\eta,L}(X_n)
   -\frac1n\log\E_{\mathbb P_n}Z_{b,\eta,L}(X_n)\notag\\
 &\le a_L+\frac\beta2\left[\log\eta^2-1-2\log b+2\eta^2\right]
   -\mathcal A_{\beta,\eta}(b)+o(1)\notag\\
 &=\frac\beta2\left[\frac{b^2}{1+\eta^2}+\log(1+\eta^2)
                   -2\log b-1+2\eta^2\right]+a_L+o(1).\label{u:eq:cost}
\end{align}
Since $b^2/(1+\eta^2)\le b^2$ and $\log(1+\eta^2)\le\eta^2$,
the last bound is at most $I_\beta(b)+a_L+3\beta\eta^2/2+o(1)$.
On the intersection of the bulk event with $\{\rho(X_n)>r\}$,
use $d\mathbb P_n=(d\mathbb Q_n^{b,\eta,L}/d\mathbb P_n)^{-1}
d\mathbb Q_n^{b,\eta,L}$. This intersection has tilted probability tending
to one by \eqref{u:eq:tilted-event}, so
\[
 \liminf_{n\to\infty}\frac1n\log\mathbb P_n\{\rho(X_n)>r\}
 \ge-I_\beta(b)-a_L-\frac{3\beta}{2}\eta^2.
\]
Let $L\to\infty$, then $\eta\downarrow0$, and finally $b\downarrow r$.
\end{proof}

\section{Matching upper bounds and local deviations}\label{sec:upper-matching}
We now prove the matching upper bounds in Theorem~\ref{u:thm:matching}.
The complex argument uses a full-sphere integral and also gives the
local eigenvalue deviations of Theorem~\ref{u:thm:complex-local}.
The real argument instead compares moments of large matrix powers
with those of real Ginibre matrices.

\subsection{The complex sharp-subgaussian upper bound}\label{u:sec:complex-upper}
Here $\F=\C$ and \eqref{u:eq:sharp} holds. Fix $\eta>0$, suppress its
dependence in $Z_z$, and use the full sphere:
\[
 Z_z(X)=\int e^{-n\norm{(X-zI)v}^2/\eta^2}\dd\sigma_n(v),\qquad
 \mathcal A_\eta(s)=\log\frac{\eta^2}{1+\eta^2}-\frac{s^2}{1+\eta^2},\quad s\ge0.
\]
The next lemma is the full-sphere counterpart of
Lemma~\ref{u:lem:flat-annealed}, with the same Gaussian rate
$\mathcal A_\eta=\mathcal A_{2,\eta}$.
The flat restriction made every coefficient small and every row target
bounded. Here we must also allow eigenvectors with large coordinates:
large coefficients can prevent the replacement error from vanishing, while targets
$\sqrt n\,zv_i$ may grow with $n$, preventing direct conversion of additive
errors to relative ones. Controlling these contributions requires the
sharp Laplace bound and a further spherical average.
\begin{lemma}[Annealed full-sphere bound]\label{u:lem:full-sphere}
Assume $\F=\C$, $\E\xi=0$, $\E|\xi|^2=1$, \eqref{u:eq:properness}, and
\emph{sharp subgaussianity} \eqref{u:eq:sharp}: for every $u\in\C$,
$\E e^{2\operatorname{Re}(\bar u\xi)}\le e^{|u|^2}
=\E e^{2\operatorname{Re}(\bar u g)}$, where $g$ is a standard proper
complex Gaussian with $\E|g|^2=1$.
For every fixed $\eta>0$ and $R<\infty$,
\[
 \limsup_{n\to\infty}\sup_{|z|\le R}
 \left[\frac1n\log\E Z_z(X_n)-\mathcal A_\eta(|z|)\right]\le0.
\]
\end{lemma}
\begin{proof}
Fix $\eta>0$ and $R<\infty$. All errors below are uniform over $|z|\le R$
and unit vectors $v$; the cutoffs are explicit functions of $n$,
so only the limit $n\to\infty$ is needed.

\emph{Row decomposition and sharp Laplace bound.}
For a deterministic unit vector $v\in\C^n$, define
\[
 \mathcal Z^{\mathrm{row}}_v(t)=\E e^{-|\sum_jv_j\xi_j-t|^2/\eta^2},\qquad
 t_i=\sqrt n\,zv_i.
\]
Row independence gives
\begin{equation}\label{u:eq:full-sphere-rows}
 \E Z_z(X_n)=\int\prod_{i=1}^n \mathcal Z^{\mathrm{row}}_v(t_i)\dd\sigma_n(v),
 \qquad \sum_i|t_i|^2=n|z|^2.
\end{equation}
Define
\[
 I(v)=\{j:|v_j|>n^{-1/4}\},\qquad |I(v)|\le\sqrt n,
 \qquad q=q(v)=\sum_{j\notin I(v)}|v_j|^2,
\]
and split the original row sum into two independent parts:
\[
 \begin{aligned}
 S_{\rm large}&=\sum_{j\in I(v)}v_j\xi_j,
 &\E|S_{\rm large}|^2&=1-q,\\
 S_{\rm small}&=\sum_{j\notin I(v)}v_j\xi_j,
 &\E|S_{\rm small}|^2&=q.
 \end{aligned}
\]
Independence and the sharp entry bound \eqref{u:eq:sharp} give, for $u\in\C$,
\begin{equation}\label{u:eq:large-sharp-laplace}
 \begin{aligned}
 \E e^{2\operatorname{Re}(\bar uS_{\rm large})}
 &=\prod_{j\in I(v)}\E e^{2\operatorname{Re}(\bar u v_j\xi_j)}\\
 &\overset{\eqref{u:eq:sharp}}{\le}
 e^{|u|^2\sum_{j\in I(v)}|v_j|^2}=e^{(1-q)|u|^2}.
 \end{aligned}
\end{equation}
Thus its Laplace coefficient is exactly its variance $1-q$, as for a
proper Gaussian of that variance. The same calculation gives coefficient
$1$ for the full row sum.
We use the elementary implication
\begin{equation}\label{u:eq:row-chernoff}
 \E e^{2\operatorname{Re}(\bar u\zeta)}\le e^{s|u|^2}\ (u\in\C)
 \quad\Longrightarrow\quad
 \E e^{-|\zeta-t|^2/a}\le e^{-|t|^2/(a+s)},\qquad a>0,\ s\ge0.
\end{equation}
Indeed, $-|\zeta-t|^2/a\le2\operatorname{Re}(\bar u(\zeta-t))+a|u|^2$;
take expectations and choose $u=t/(a+s)$. For the full row, $s=1$, so
\begin{equation}\label{u:eq:row-all-targets}
 e^{|t|^2/(1+\eta^2)}\mathcal Z^{\mathrm{row}}_v(t)\le1,
 \qquad t\in\C.
\end{equation}
This bound has no additive error and will be used for large targets.
The replacement below improves the prefactor to $h(q)$, at the cost
of an additive error that may dominate for large $|t|$.

\emph{Replace only the small part.}
Conditional on $S_{\rm large}$, replace $S_{\rm small}$ by an independent
proper Gaussian $G_q$ with $\E|G_q|^2=q$. The large part retains its
original law. Covariance matching and \eqref{u:eq:properness} give
$\E G_q^2=\E S_{\rm small}^2
=\sum_{j\notin I(v)}v_j^2\E\xi_j^2=0$ for every $v$.
Thus $G_q\overset d=\sqrt q\,g$ with a standard proper complex Gaussian
$g$ whose law does not depend on $v$; the only dependence is through $q$.
The Gaussian replacement formula \eqref{u:eq:smooth-replacement}
of Lemma~\ref{u:lem:soft-replacement}, proved in Subsection~\ref{u:sec:coupling-details},
uses only moment matching and a finite third moment. Its additive error
is at most $C_{\mu,\eta}n^{-1/4}$, since
$\sum_{j\notin I(v)}|v_j|^3\le n^{-1/4}$.
To integrate out $G_q$, for $q>0$ and $s\in\C$ complete the square:
\[
 \begin{aligned}
 \E e^{-|G_q-s|^2/\eta^2}
 &=\frac1{\pi q}\int_\C
 \exp\left[-\frac{\eta^2+q}{\eta^2q}
 \left|w-\frac{q s}{\eta^2+q}\right|^2
 -\frac{|s|^2}{\eta^2+q}\right]\dd^2w\\
 &=\underbrace{\frac{\eta^2}{\eta^2+q}}_{=:h(q)}
 e^{-|s|^2/(\eta^2+q)}.
 \end{aligned}
\]
Here $\dd^2w$ is planar Lebesgue measure and
$\int_\C e^{-c|w-w_0|^2}\dd^2w=\pi/c$.
For $q=0$, $G_q=0$ and the same formula holds with $h(0)=1$.
This is \eqref{u:eq:gaussian-soft-formula} with $a=\eta^2$ and $\beta=2$.
Set $s=t-S_{\rm large}$ and average over $S_{\rm large}$; then apply
\eqref{u:eq:row-chernoff} with $a=\eta^2+q$ and Laplace coefficient $1-q$
from \eqref{u:eq:large-sharp-laplace}:
\begin{equation}\label{u:eq:row-light-replacement}
 \begin{aligned}
 \mathcal Z^{\mathrm{row}}_v(t)
 &\le h(q)\E e^{-|S_{\rm large}-t|^2/(\eta^2+q)}+C_{\mu,\eta}n^{-1/4}\\
 &\overset{\eqref{u:eq:large-sharp-laplace}}{\le}
 h(q)e^{-|t|^2/(1+\eta^2)}+C_{\mu,\eta}n^{-1/4}.
 \end{aligned}
\end{equation}
Sharpness gives the Gaussian denominator $(\eta^2+q)+(1-q)=1+\eta^2$;
a larger Laplace coefficient would weaken the exponent.

\emph{Multiply the rows and average over the sphere.}
For $|t|^2\le\frac18\log n$, the Gaussian main term in
\eqref{u:eq:row-light-replacement} is at least $h(1)n^{-1/8}$.
Hence its additive error becomes a relative $O_{\mu,\eta}(n^{-1/8})$ error.
By \eqref{u:eq:full-sphere-rows}, at most $8n|z|^2/\log n$ targets violate
this cutoff. Use the sharp bound \eqref{u:eq:row-all-targets} there.
Since $h(q)\ge h(1)$, multiplication gives
\begin{equation}\label{u:eq:row-product}
 \prod_i \mathcal Z^{\mathrm{row}}_v(t_i)\le h(q(v))^n
 \exp\left[-\frac{n|z|^2}{1+\eta^2}
       +O_{\mu,\eta,R}\left(n^{7/8}+\frac n{\log n}\right)\right].
\end{equation}
For a fixed set $I$ of size $m$, write $q_I=\sum_{j\notin I}|v_j|^2$.
Lemma~\ref{u:lem:spherical-polar}, with $A=\eta^2I_n+P_{I^c}$, gives exactly
\begin{equation}\label{u:eq:beta}
 \E_{\sigma_n}h(q_I)^n
 =\frac{\eta^{2n}}{\det(\eta^2I_n+P_{I^c})}
 =h(1)^{n-m}.
\end{equation}
Here $P_{I^c}$ is the coordinate projection onto $I^c$.
Since $|I(v)|\le\sqrt n$, summing over all possible sets gives
\[
 \E_{\sigma_n}h(q(v))^n
 \le\sum_{m=0}^{\lfloor\sqrt n\rfloor}\binom nm h(1)^{n-m}
 =h(1)^n e^{O_\eta(\sqrt n\log n)}.
\]
Combining with \eqref{u:eq:full-sphere-rows} and \eqref{u:eq:row-product},
uniformly for $|z|\le R$,
\[
 \frac1n\log\E Z_z(X_n)-\mathcal A_\eta(|z|)
 \le O_{\mu,\eta,R}\left(n^{-1/8}+\frac1{\log n}
                         +\frac{\log n}{\sqrt n}\right)=o(1).
\]
This proves the lemma.
\end{proof}

\begin{proof}[Complex upper bound in Theorem~\ref{u:thm:matching}]\leavevmode\par
Under \eqref{u:eq:sharp}, independence and Chernoff give
$\Prob\{|u^*X_nv|>s\}\le4e^{-ns^2/2}$ for deterministic unit vectors
$u,v$ (test real and imaginary parts at $s/\sqrt2$).
Two $1/4$-nets, each of size at most $9^{2n}$, bound $\norm{X_n}$ by
twice the largest tested bilinear form; a union bound gives
\begin{equation}\label{u:eq:norm-tail}
 \Prob\{\norm{X_n}>R\}\le4\exp(4n\log9-nR^2/8).
\end{equation}
Any fixed subgaussian constant would suffice for this norm cutoff.
Fix $r>1$, and choose fixed $R$ so that the coefficient of $n$
in the exponent in \eqref{u:eq:norm-tail} is below $-I_2(r)-1$. Let
$K=\{(1+r)/2\le|z|\le R+1\}$ and fix small $\eta$ with
$1+\eta^2<r^2$. Take a deterministic $1/n$-net of $\{r\le|z|\le R\}$,
of cardinality $O_R(n^2)$, contained in $\{r-1/n\le|z|\le R+1/n\}$.
An eigenvalue in the annulus supplies a grid point with
$s_{\min}(X_n-zI)\le1/n$, where $s_{\min}$ denotes the smallest
singular value. Apply \eqref{u:eq:witness} from Lemma~\ref{u:lem:spherical-polar} with
$H=(X_n-zI)^*(X_n-zI)$, whose smallest eigenvalue is
$s_{\min}(X_n-zI)^2$.
On the bulk event \eqref{u:eq:bulk-upper} with $\tau=\sqrt2\eta$,
Stirling and \eqref{u:eq:witness} give at this point
\[
 \frac1n\log Z_z(X_n)\ge
 \log\eta^2-1-2\log|z|-2M_K^2\eta^2-\delta+o(1).
\]
Markov's inequality, Lemma~\ref{u:lem:full-sphere}, and a union bound give
\[
 \limsup_n\frac1n\log\Prob\{\rho(X_n)\ge r,\ \norm{X_n}\le R,
                   \text{bulk event of \eqref{u:eq:bulk-upper}}\}
 \le-I_{2,\eta}(r)+2M_K^2\eta^2+\delta,
\]
where $I_{2,\eta}(b)=b^2/(1+\eta^2)-1-2\log b+\log(1+\eta^2)$ is
increasing for $b>\sqrt{1+\eta^2}$. The grid size is negligible at speed
$n$, and the bulk failure probability decays at speed $n^2$.
Let $\delta,\eta\downarrow0$ after the dimension limit, with $R$ fixed.
Then $I_{2,\eta}(r)\to I_2(r)$ and the desired upper bound follows.
Theorem~\ref{u:thm:lower} supplies the matching lower bound.
\end{proof}

\subsection{Local eigenvalue deviations in the complex exterior}\label{u:sec:complex-local}
We prove Theorem~\ref{u:thm:complex-local}. The lower-bound tilt can place
an actual outlier near any prescribed point outside the unit disk.
All smoothing and flatness parameters remain fixed during the limit
$n\to\infty$.

\begin{proof}[Proof of Theorem~\ref{u:thm:complex-local}]
\emph{A lower bound near a prescribed point.}
We first show that, under \eqref{u:eq:basic}--\eqref{u:eq:properness},
for every $w\in\C$ with $|w|>1$ and every $\delta>0$,
\begin{equation}\label{u:eq:point-local-lower}
 \liminf_{n\to\infty}\frac1n\log\Prob(\mathcal E_n(w,\delta))
 \ge-I_2(|w|).
\end{equation}
Fix $\eta>0$, let $b=(1+\eta^2)w$, and fix sufficiently large $L$.
Use the joint tilt $\widehat{\mathbb Q}_n^{b,\eta,L}$ and its matrix
marginal $\mathbb Q_n^{b,\eta,L}$ from the proof of
Theorem~\ref{u:thm:lower}, now with this complex $b$ and $\beta=2$.
Lemma~\ref{u:lem:flat-annealed} gives
\begin{equation}\label{u:eq:local-normalization}
 \frac1n\log\E_{\mathbb P_n}Z_{b,\eta,L}(X_n)
 =\log\frac{\eta^2}{1+\eta^2}
  -\frac{|b|^2}{1+\eta^2}+o(1).
\end{equation}
Indeed, the row targets $t_i=\sqrt n\,b v_i$ satisfy
$|t_i|\le |b|L$ and $\sum_i|t_i|^2=n|b|^2$ for every
$v\in\mathcal F_{n,L}$; no rotation of the entry law is used.

Conditional on such a deterministic $v$, Lemma~\ref{u:lem:coupling}
represents the tilted matrix as
\[
 X_n=Y_nT+wvv^*+E_n,\qquad
 \norm{T-I}_{\HS}^2\le(2\eta^4)^{-1},\qquad
 \norm{E_n}_{\op}\xrightarrow{\Prob}0.
\]
Lemma~\ref{u:lem:outlier}, with the complex spike $s=w$, gives an
eigenvalue converging to $w$. Its uniformity over deterministic flat
directions allows integration against the tilted marginal of $v$.
Consequently,
\begin{equation}\label{u:eq:local-tilted-event}
 \mathbb Q_n^{b,\eta,L}(\mathcal E_n(w,\delta))\longrightarrow1.
\end{equation}

Let $\mathcal B_n=\{F_{\eta^2,n}(b)\ge2\log|b|-\eta^2\}$.
The complex-point version of \eqref{u:eq:bulk-lower} gives
$\mathbb P_n(\mathcal B_n^c)\le Ce^{-cn^2}$.
Since $Z_{b,\eta,L}\le1$, \eqref{u:eq:local-normalization} implies
$\mathbb Q_n^{b,\eta,L}(\mathcal B_n^c)=o(1)$.
On $\mathcal B_n$, \eqref{u:eq:flat-mass} and
\eqref{u:eq:spherical-upper} yield
\[
 \frac1n\log Z_{b,\eta,L}(X_n)
 \le a_L+\log\eta^2-1-2\log|b|+2\eta^2+o(1).
\]
Subtracting \eqref{u:eq:local-normalization}, we obtain
\begin{equation}\label{u:eq:local-cost}
 \frac1n\log\frac{d\mathbb Q_n^{b,\eta,L}}{d\mathbb P_n}(X_n)
 \le a_L+C_\eta(w)+o(1),
\end{equation}
where
\begin{align*}
 C_\eta(w)&=\frac{|b|^2}{1+\eta^2}+\log(1+\eta^2)
             -2\log|b|-1+2\eta^2\\
 &=I_2(|w|)+\eta^2|w|^2-\log(1+\eta^2)+2\eta^2.
\end{align*}
By \eqref{u:eq:local-tilted-event},
$\mathbb Q_n^{b,\eta,L}(\mathcal E_n(w,\delta)\cap\mathcal B_n)\to1$.
Integrating the reciprocal likelihood ratio on this intersection gives
\[
 \liminf_{n\to\infty}\frac1n\log\mathbb P_n(\mathcal E_n(w,\delta))
 \ge-a_L-C_\eta(w).
\]
Let $L\to\infty$ and then $\eta\downarrow0$ to obtain
\eqref{u:eq:point-local-lower}.

\emph{The rate for a fixed disk.}
Fix $|z|>1$ and $0<\varepsilon<|z|-1$, and set
$a=|z|-\varepsilon>1$. For $0<t<\varepsilon$, set
$w_t=(a+t)z/|z|$. Then
$B(w_t,t/2)\subset B(z,\varepsilon)$, so
\eqref{u:eq:point-local-lower} gives
\[
 \liminf_{n\to\infty}\frac1n\log\Prob(\mathcal E_n(z,\varepsilon))
 \ge-I_2(a+t).
\]
Letting $t\downarrow0$ proves \eqref{u:eq:complex-local-lower}.
Under \eqref{u:eq:sharp}, the reverse bound follows from
\[
 \mathcal E_n(z,\varepsilon)\subset\{\rho(X_n)>a\}
\]
and Theorem~\ref{u:thm:matching}. This proves
\eqref{u:eq:complex-local}; its small-disk limit follows by continuity
of $I_2$.
\end{proof}

\subsection{Real moment domination and the Gaussian power estimate}\label{u:sec:real-upper}
For the real upper bound in Theorem~\ref{u:thm:matching},
$\rho(X_n)^{2k}\le\|X_n^k\|_{\HS}^2$ turns the spectral event into a
moment bound, with $k$ proportional to
$n$. Entrywise moment domination reduces this bound to real Ginibre.
The Schur estimate below controls the additional growth of Gaussian
matrix powers caused by nonnormality: its factor $e^{C_\varepsilon k^{2/3}}$
is negligible at speed $n$. This lets the known Gaussian spectral-radius
rate determine the power-moment bound.

\begin{fact}[Conditional real Schur distribution]\label{u:fact:conditional-schur}
Let $G_n$ have iid real $N(0,1/n)$ entries. Write $G_n=QUQ^\top$
in ordered real Schur form, where $Q$ is orthogonal and
\[
 U=\begin{pmatrix}
 \mathsf S_1&N_{12}&\cdots&N_{1,N_{\rm blk}}\\
 0&\mathsf S_2&\cdots&N_{2,N_{\rm blk}}\\
 \vdots&\ddots&\ddots&\vdots\\
 0&\cdots&0&\mathsf S_{N_{\rm blk}}
 \end{pmatrix}.
\]
Here $N_{\rm blk}$ is the number of diagonal blocks,
$\mathsf S_a\in\R^{d_a\times d_a}$ with $d_a\in\{1,2\}$, and
$N_{ab}\in\R^{d_a\times d_b}$ for $a<b$.
Fix an ordering of the eigenvalues and the corresponding block pattern.
Use the diagonal-block convention in item~1 below, and fix the remaining
sign of each block basis by requiring its first column to have positive
first nonzero coordinate. The following statements hold for almost every
spectrum, in the sense of a regular conditional distribution.
\begin{enumerate}
\item A real eigenvalue $\lambda_a$ gives the fixed scalar block
$\mathsf S_a=[\lambda_a]$. A conjugate pair $x_a\pm i y_a$, $y_a>0$, gives
\[
 \mathsf S_a=\begin{pmatrix}x_a&b_a\\-c_a&x_a\end{pmatrix},\qquad
 b_a\ge c_a>0,\qquad b_ac_a=y_a^2.
\]
Its remaining random coordinate is $\Delta_a=b_a-c_a\ge0$, with
\[
 b_a=\frac{\sqrt{\Delta_a^2+4y_a^2}+\Delta_a}{2},\qquad
 c_a=\frac{\sqrt{\Delta_a^2+4y_a^2}-\Delta_a}{2}.
\]
\item For each $a<b$, the off-diagonal block $N_{ab}$ has density
\[
 B\longmapsto
 \left(\frac n{2\pi}\right)^{d_ad_b/2}
 e^{-n\norm B_{\HS}^2/2},\qquad B\in\R^{d_a\times d_b},
\]
with respect to Lebesgue measure on its $d_ad_b$ scalar entries.
Each scalar entry thus has density
$x\mapsto\sqrt{n/(2\pi)}e^{-nx^2/2}$, $x\in\R$.
\item For each two-dimensional diagonal block $\mathsf S_a$, its
coordinate $\Delta_a=b_a-c_a$ has density
\begin{equation}\label{u:eq:conditional-schur-density}
 p_{n,y_a}(t)=\frac1{Z_{n,y_a}}
 \frac{t}{\sqrt{t^2+4y_a^2}}e^{-nt^2/2}\mathbf1_{t>0},\qquad
 Z_{n,y}=\int_0^\infty\frac{t e^{-nt^2/2}}{\sqrt{t^2+4y^2}}\dd t.
\end{equation}
The numbers $b_a,c_a$ inside $\mathsf S_a$ are functions of $\Delta_a$
as above; they are not entries of an off-diagonal block $N_{ab}$.
\item All scalar entries of all blocks $(N_{ab})_{a<b}$, together with
all coordinates $(\Delta_a)_{d_a=2}$, are mutually independent under
the conditional law given the spectrum.
\item Uniformly in the conditioned spectrum,
\begin{equation}\label{u:eq:conditional-schur-second-moments}
 \E[\Delta_a^2\mid\operatorname{spec}(G_n)]\le\frac2n,
 \qquad
 \E[\norm{N_{ab}}_{\HS}^2\mid\operatorname{spec}(G_n)]
 =\frac{d_ad_b}{n}\le\frac4n.
\end{equation}
\end{enumerate}
\end{fact}
The proof is given in Subsection~\ref{u:sec:schur-distribution}.

\begin{lemma}[Conditional Schur power bound]\label{u:lem:schur}
Let $G_n$ be an $n\times n$ real random matrix with independent
$N(0,1/n)$ entries.
For each $\varepsilon>0$ there is $C_\varepsilon<\infty$ such that,
almost surely in the eigenvalues, for all integers $k\ge1$,
\begin{equation}\label{u:eq:schur}
 \E[\norm{G_n^k}_{\HS}^2\mid\operatorname{spec}(G_n)]
 \le C_\varepsilon n\exp(C_\varepsilon k^{2/3})
 (\max(1,\rho(G_n))+\varepsilon)^{2k}.
\end{equation}
\end{lemma}
\begin{proof}
Use the Schur decomposition and block notation of
Fact~\ref{u:fact:conditional-schur}, so
$\norm{G_n^k}_{\HS}=\norm{U^k}_{\HS}$.
The notation $\mathsf S_a^p$ denotes the ordinary matrix power, with
$\mathsf S_a^0=I_{d_a}$.
Condition throughout on $\operatorname{spec}(G_n)$: all expectations
and independence statements below refer to this conditional law.
By that fact, the diagonal blocks are mutually independent, and all scalar entries
of the $N_{ab}$ are independent $N(0,1/n)$ variables, independent of
all diagonal blocks.

\emph{Diagonal block powers.}
Set $r_0=\max(1,\rho(G_n))$, $r_\varepsilon=r_0+\varepsilon$.
For a complex block, write
$\mathsf S=\left(\begin{smallmatrix}x&b\\-c&x\end{smallmatrix}\right)$,
$b\ge c>0$, and set $y=\sqrt{bc}$ and $\Delta=b-c$.
The same fact gives $\E\Delta^2\le2/n$. The normal matrix
$\mathsf C=\left(\begin{smallmatrix}x&y\\-y&x\end{smallmatrix}\right)$ has the
same eigenvalues as $\mathsf S$ and $\norm{\mathsf S-\mathsf C}\le\Delta$.
Write $\lambda_\pm=x\pm iy$. Both matrices satisfy the same quadratic
equation, so for $p\ge1$, polynomial division by
$(z-\lambda_+)(z-\lambda_-)$ gives
\[
 \mathsf S^p-\mathsf C^p
 =\left(\sum_{j=0}^{p-1}\lambda_+^j\lambda_-^{p-1-j}\right)
       (\mathsf S-\mathsf C).
\]
Bounding each coefficient by $r_0^{p-1}$ yields, also for $p=0$,
\[
 \norm{\mathsf S^p}\le r_0^p+p\Delta r_0^{p-1}
 \le r_\varepsilon^p(1+\Delta/\varepsilon),\qquad p\in\mathbb Z_{\ge0}.
\]
The last inequality uses
$(r_0+\varepsilon)^p\ge p\varepsilon r_0^{p-1}$ for $p\ge1$.
Moreover, $\E(1+\Delta/\varepsilon)^2\le2+4/(n\varepsilon^2)
\le C_\varepsilon$, uniformly in the power $p$.
The same bound holds for scalar blocks, with multiplier $1$.

\emph{Path expansion and cancellation.}
Each factor in $U^k$ either stays in the current diagonal block or jumps
to a block with a larger index. The contribution of a strictly increasing
block path $a_0<\cdots<a_\ell$ is therefore the sum over
nonnegative integer exponents satisfying
$p_0+\cdots+p_\ell=k-\ell$ of
\[
 \mathsf S_{a_0}^{p_0}N_{a_0a_1}\mathsf S_{a_1}^{p_1}\cdots
 N_{a_{\ell-1}a_\ell}\mathsf S_{a_\ell}^{p_\ell}.
\]
There are $\binom{k}{\ell}$ ways to allocate the $k-\ell$ diagonal steps
among the $\ell+1$ visited blocks. Distinct paths with the same
endpoints have an edge in their symmetric difference. That independent
centered off-diagonal block occurs linearly in one path only, so the
cross expectation vanishes. Different endpoints occupy different
blocks. Time allocations along a fixed path need not be orthogonal:
the $L^2$ triangle inequality gives the squared factor $\binom{k}{\ell}^2$.
Independence and $\E\norm{N_{ab}}_{\HS}^2\le4/n$ then bound the path's
expected squared Hilbert--Schmidt norm by
$2C_\varepsilon^{\ell+1}(4/n)^\ell\binom{k}{\ell}^2r_\varepsilon^{2(k-\ell)}$.

\emph{Counting paths.}
There are at most $\binom n{\ell+1}$ paths with $\ell$ jumps.
Since $\binom n{\ell+1}\le n^{\ell+1}/(\ell+1)!$ and
$\binom{k}{\ell}\le k^\ell/\ell!$, the path count contributes one factorial to the denominator,
and the squared number of time allocations contributes two. Using $r_\varepsilon\ge1$ gives
\[
 \E\norm{U^k}_{\HS}^2\le C_\varepsilon n r_\varepsilon^{2k}
 \sum_{\ell\ge0}\frac{(C_\varepsilon k^2)^\ell}{(\ell!)^3}
 \le C_\varepsilon n r_\varepsilon^{2k}e^{C_\varepsilon k^{2/3}}.
\]
The last inequality follows by retaining only equal-index terms in
$(\sum_{\ell\ge0}a^\ell/\ell!)^3=e^{3a}$, with
$a=(C_\varepsilon k^2)^{1/3}$.
\end{proof}

\begin{corollary}[Gaussian power-moment asymptotics]\label{u:cor:gaussian-power-rate}
Let $G_n$ be an $n\times n$ real random matrix with independent
$N(0,1/n)$ entries. For every sequence of positive integers $k=k_n$
with $k_n/n\to\alpha>0$,
\begin{equation}\label{u:eq:power-rate}
 \lim_{n\to\infty}\frac1n\log\E\norm{G_n^{k_n}}_{\HS}^2
 =F_1(\alpha):=\frac{(1+2\alpha)\log(1+2\alpha)-2\alpha}{2}.
\end{equation}
\end{corollary}
\begin{proof}
Write $\rho_n=\rho(G_n)$ and $S_n=\max(1,\rho_n)$.
The Gaussian right-tail bound \cite[Theorem 1.2, $\beta=1$]{XuZeng},
in the same $N(0,1/n)$ normalization, gives
\[
 \limsup_n\frac1n\log\Prob(\rho_n>t)\le-I_1(t),\qquad t\ge1,
\]
where $t=1$ is trivial with $I_1(1)=0$.
For fixed $0<\varepsilon\le1$, Lemma~\ref{u:lem:schur} gives
\begin{equation}\label{u:eq:power-tower}
 \E\norm{G_n^{k_n}}_{\HS}^2
 \le C_\varepsilon n e^{C_\varepsilon k_n^{2/3}}
                \E(S_n+\varepsilon)^{2k_n}.
\end{equation}
The prefactor has logarithm $o(n)$. Tail integration gives the exact identity
\begin{equation}\label{u:eq:power-tail-integral}
 \E(S_n+\varepsilon)^{2k_n}
 =(1+\varepsilon)^{2k_n}
 +\int_1^\infty 2k_n(t+\varepsilon)^{2k_n-1}\Prob(\rho_n>t)\dd t.
\end{equation}

For a partition $1=t_0<\cdots<t_q=M$, the integral over $[1,M]$ is at most
$\sum_{j<q}(t_{j+1}+\varepsilon)^{2k_n}\Prob(\rho_n>t_j)$.
Apply the tail bound at the fixed partition points, then let the mesh
tend to zero. Including the first term in \eqref{u:eq:power-tail-integral}
gives logarithmic limsup at most
\[
 \sup_{1\le t\le M}\{2\alpha\log(t+\varepsilon)-I_1(t)\}.
\]
For $t\ge M$, the same net argument as for \eqref{u:eq:norm-tail}, now
on the real unit sphere, gives
$\Prob(\rho_n>t)\le2e^{n(c_{\mathrm{net}}-t^2/8)}$, $c_{\mathrm{net}}=2\log9$.
Fix $\alpha_+>\alpha$ and choose $M$ so large that
$2\alpha_+\log(t+1)\le t^2/16$ on this range.
Since $k_n\le\alpha_+n$ eventually, the remaining integral is at most
\[
 4\alpha_+n e^{nc_{\mathrm{net}}}\int_M^\infty e^{-nt^2/16}\dd t
 \le\frac{32\alpha_+}{M}e^{n(c_{\mathrm{net}}-M^2/16)}.
\]
The last inequality uses $t/M\ge1$ inside the integral.
Its rate tends to $-\infty$ with $M$, uniformly for $0<\varepsilon\le1$.
Combining the bounds with \eqref{u:eq:power-tower}, then letting
$M\to\infty$ and $\varepsilon\downarrow0$, yields
\[
 \limsup_n\frac1n\log\E\norm{G_n^{k_n}}_{\HS}^2
 \le\sup_{t\ge1}\{2\alpha\log t-I_1(t)\}=F_1(\alpha).
\]
Here the passage $\varepsilon\downarrow0$ uses
$0\le\log(t+\varepsilon)-\log t\le\log(1+\varepsilon)$ for $t\ge1$.
Differentiation places the supremum at $t=\sqrt{1+2\alpha}$.

For the reverse inequality, Theorem~\ref{u:thm:lower} applies to $G_n$.
For each fixed $t>1$, the inequality
$\rho(G_n)^{2k_n}\le\norm{G_n^{k_n}}_{\HS}^2$ gives
\[
 \E\norm{G_n^{k_n}}_{\HS}^2\ge t^{2k_n}\Prob(\rho_n>t),
 \qquad
 \liminf_n\frac1n\log\E\norm{G_n^{k_n}}_{\HS}^2
 \ge2\alpha\log t-I_1(t).
\]
Taking the supremum over $t>1$ proves the matching lower bound.
\end{proof}

\begin{proof}[Real upper bound in Theorem~\ref{u:thm:matching}]
Expand $\E\norm{X_n^k}_{\HS}^2$ as pairs of directed paths with the
same endpoints. All coefficients are positive. Independence and
symmetry make a term vanish if any entry label has odd multiplicity;
otherwise Gaussian even-moment domination bounds its expectation
by the corresponding Gaussian term. Therefore, for every $n,k$,
\[
 \E\norm{X_n^k}_{\HS}^2\le\E\norm{G_n^k}_{\HS}^2.
\]
Since $\rho(X_n)^{2k}\le\norm{X_n^k}_{\HS}^2$, Markov and
\eqref{u:eq:power-rate} yield
\[
 \limsup_n n^{-1}\log\Prob\{\rho(X_n)\ge r\}
 \le\inf_{\alpha>0}\{F_1(\alpha)-2\alpha\log r\}=-I_1(r).
\]
The minimizer is $\alpha=(r^2-1)/2$. Combine with
Theorem~\ref{u:thm:lower}.
\end{proof}

\section{Coupling and Schur estimates}\label{u:sec:technical-proofs}
This section supplies the two technical inputs deferred in the preceding
arguments: the Gaussian replacement and triangular coupling used for
the outlier construction, and the conditional real Schur distribution
used for the Gaussian power estimate.

\subsection{Gaussian replacement and the triangular coupling}\label{u:sec:coupling-details}
This subsection proves the Gaussian replacement estimate of
Lemma~\ref{u:lem:soft-replacement} and then the triangular coupling of
Lemma~\ref{u:lem:coupling}. The latter proof uses the row tilt and parameters
of Subsection~\ref{u:sec:coupling}, with $\eta,L,T_0$ fixed and all estimates
uniform over the permitted row directions and targets.

\begin{proof}[Proof of Lemma~\ref{u:lem:soft-replacement}]
Every derivative of order at most four of $\varphi_a(x)=e^{-|x|^2/a}$
is a polynomial times the same Gaussian, hence is bounded by a constant
depending only on $a$. Replace the summands one at a time, applying
second-order Taylor expansion to $\varphi_a$ and to each component of
$\nabla\varphi_a$. The constant, linear and quadratic terms agree:
the two entry laws have the same mean and real covariance
$\beta^{-1}\mathrm{Id}_\beta$, and multiplication by $v_j$ is a real linear map.
The remainders are bounded by $C_a m_3|v_j|^3$, uniformly in the sum
of the other entries and in $s$. Summation proves
\eqref{u:eq:smooth-replacement}. Bounded derivatives also justify
differentiation under each expectation.

The Gaussian sum has independent real coordinates of variance $q/\beta$.
Completing the square in each coordinate gives
\[
 \E e^{-(s_\ell-Z)^2/a}
 =\sqrt{\frac{a}{a+2q/\beta}}
           e^{-s_\ell^2/(a+2q/\beta)},\qquad Z\sim N(0,q/\beta).
\]
Multiplication over $\beta$ coordinates proves
\eqref{u:eq:gaussian-soft-formula}; for $q=0$ the sum is zero.
Finally Jensen gives $\mathcal Z_\mu(s),\mathcal Z_G(s)\ge
J:=e^{-(K^2+1)/a}$ on $|s|\le K$. Thus
\[
 \left|\frac{\nabla\mathcal Z_\mu}{\mathcal Z_\mu}
              -\frac{\nabla\mathcal Z_G}{\mathcal Z_G}\right|
 \le J^{-1}|\nabla\mathcal Z_\mu-\nabla\mathcal Z_G|
       +J^{-2}|\nabla\mathcal Z_G|\,|\mathcal Z_\mu-\mathcal Z_G|.
\]
The Gaussian derivative is bounded, so \eqref{u:eq:compact-score} follows.
\end{proof}

\subsubsection*{Proof of the triangular coupling}
\begin{proof}[Proof of Lemma~\ref{u:lem:coupling}]
All estimates are uniform over the permitted parameters, with
$\eta,L,T_0$ fixed. Constants may depend on $\mu,\eta,L,T_0$;
we display the $\eta$-dependence in the remainder bounds, suppressing
dependence on the fixed $\mu,L,T_0$. Set $\Delta:=\max_j|v_j|$;
the hypothesis of Lemma~\ref{u:lem:coupling} gives $\Delta\le L/\sqrt n$.
Centeredness and \eqref{u:eq:basic} imply
$\E e^{\langle h,\xi\rangle}\le e^{C|h|^2}$: Taylor expansion gives
$1+O(|h|^2)$ for $|h|\le1$, while Young's inequality absorbs the
square-exponential moment into $e^{C|h|^2}$ for $|h|>1$.
By independence, every original sum
$U=\sum_k a_k\xi_k$ with $\sum_k|a_k|^2\le1$ satisfies
$\E e^{\langle h,U\rangle}\le e^{C|h|^2}$.
Integrating against an independent standard Gaussian $G\in\R^\beta$ gives,
for sufficiently small $c>0$,
\[
 \E e^{c|U|^2}
 =\E_G\E_U e^{\sqrt{2c}\langle G,U\rangle}
 \le\E_G e^{2Cc|G|^2}<\infty.
\]
Since the row density in \eqref{u:eq:row-tilt} is bounded uniformly,
tilted entries and prefix residuals also have uniform square-exponential
moments.

\paragraph{Conditional means.}
Write $\varphi(s)=e^{-\beta|s|^2/(2\eta^2)}$ and, using independent
original entries, define
\[
 U_j=\sum_{k>j}v_k\xi_k,\qquad
 \mathcal Z_j(s)=\E\varphi(s-U_j),\qquad
 A_j(s)=\int\mathcal Z_j(s-v_jx)\dd\mu(x).
\]
For a tilted row with target $t$, set $S_j=t-\sum_{k<j}v_kx_k$.
For a deterministic residual value $s\in\F$, define the kernel and mean
\begin{equation}\label{u:eq:kernel}
 K_{j,s}(\dd x)=\frac{\mathcal Z_j(s-v_jx)}{A_j(s)}\dd\mu(x),
 \qquad m_j(s)=\int x\dd K_{j,s}(x).
\end{equation}
The actual conditional law uses the random residual $s=S_j$:
integrating out the future entries in \eqref{u:eq:row-tilt} gives, under $\mu_{v,t}$,
\[
 \mathcal L(x_j\mid x_1,\ldots,x_{j-1})=K_{j,S_j},\qquad
 \E[x_j\mid x_1,\ldots,x_{j-1}]=m_j(S_j).
\]
For $\xi\sim\mu$ and the real gradient $\gamma=\nabla\mathcal Z_j(s)\in\F$,
the Hessian bound $\sup_s\norm{\nabla^2\mathcal Z_j(s)}_{\op}
\le C\eta^{-2}$ gives
\[
 \mathcal Z_j(s-v_j\xi)=\mathcal Z_j(s)
 -\langle \gamma,v_j\xi\rangle+O(\eta^{-2}|v_j|^2|\xi|^2).
\]
Averaging and centeredness give $A_j(s)=\mathcal Z_j(s)+O(\eta^{-2}|v_j|^2)$.
In the complex case, \eqref{u:eq:properness} and $\E|\xi|^2=1$ give
\[
 \E[\xi\langle \gamma,v_j\xi\rangle]
 =\tfrac12\bigl(\bar \gamma v_j\E\xi^2
                   +\gamma\bar v_j\E|\xi|^2\bigr)
 =\tfrac12\bar v_j \gamma.
\]
In the real case this expectation is $v_j\gamma$. Thus centeredness and
the finite third moment yield, in either field,
\begin{align*}
 A_j(s)m_j(s)&=\E[\xi\mathcal Z_j(s-v_j\xi)]\\
 &=-\E[\xi\langle \gamma,v_j\xi\rangle]
       +O(\eta^{-2}|v_j|^2\E|\xi|^3)\\
 &=-\frac{\bar v_j}{\beta}\nabla\mathcal Z_j(s)+O(\eta^{-2}|v_j|^2).
\end{align*}
Recall the tail denominators from Lemma~\ref{u:lem:coupling}:
\[
 d_j=\eta^2+\sum_{k=j}^n|v_k|^2\quad(1\le j\le n+1),
 \qquad d_{n+1}=\eta^2.
\]
Choose an arbitrary deterministic cutoff radius $K$, independent of $n$:
\[
 K\ge1,\qquad |s|\le K\ \text{(compact part)},\qquad
 |s|>K\ \text{(tail)}.
\]
On the compact part, Jensen gives the explicit lower bound
\[
 \min\{A_j(s),\mathcal Z_j(s)\}
 \ge \exp\!\left[-\frac{\beta(K^2+1)}{2\eta^2}\right],\qquad |s|\le K.
\]
The derivative bounds $\norm{\nabla^\ell\varphi}_\infty\le C_\ell\eta^{-\ell}$
for $1\le\ell\le4$ and this lower bound allow the constants in the
conditional-mean estimates below, including those in $O_{\eta,K}$,
to satisfy the coarse bounds
\[
 C_{\eta,K}\le C(1+\eta^{-8})
      \exp\!\left[\frac{C(1+K^2)}{\eta^2}\right],\qquad
 C_\eta\le C(1+\eta^{-8})e^{C/\eta^2}.
\]
Here $C$ depends only on $\mu,L,T_0$; the exponential factors account
for division by the normalizers.
Dividing and applying \eqref{u:eq:compact-score} with
$a=2\eta^2/\beta$ to the future coefficients $k>j$, whose cubic sum
is at most $\Delta$ and whose variance sum is $d_{j+1}-\eta^2$, yields
$m_j(s)=\bar v_j s/d_{j+1}+O_{\eta,K}(\Delta|v_j|)$.
We use $\bar v_j s/d_j$ as the regression term. Its difference from
the leading term just obtained is controlled explicitly: since
$d_j-d_{j+1}=|v_j|^2$ and both denominators are at least $\eta^2$,
\[
 \left|\frac{\bar v_j s}{d_{j+1}}-\frac{\bar v_j s}{d_j}\right|
 =\frac{|s||v_j|^3}{d_jd_{j+1}}
 \le K\eta^{-4}|v_j|^3,\qquad |s|\le K.
\]
Since $|v_j|\le\Delta\le1$, the additional $K\eta^{-4}|v_j|^3$
error is absorbed into $C_{\eta,K}\Delta|v_j|$. Thus
\begin{equation}\label{u:eq:regression}
 r_j(s):=m_j(s)-\frac{\bar v_j s}{d_j},\qquad
 \sup_{|s|\le K}|r_j(s)|\le C_{\eta,K}\Delta|v_j|.
\end{equation}
To control the complement of this compact set, take independent
$\xi,\xi'\sim\mu$, independent of the same $U_j$, and define
$F_x=\varphi(s-v_jx-U_j)$. Centeredness gives
\[
 A_j(s)m_j(s)=\tfrac12\E[(\xi-\xi')(F_\xi-F_{\xi'})].
\]
The function $\sqrt\varphi$ is $C/\eta$-Lipschitz, so Cauchy--Schwarz and
$\E(\sqrt{F_\xi}+\sqrt{F_{\xi'}})^2\le4A_j(s)$ imply
\begin{equation}\label{u:eq:weighted-mean}
 A_j(s)|m_j(s)|^2\le C\eta^{-2}|v_j|^2\E|\xi-\xi'|^4
 \le C\eta^{-2}|v_j|^2.
\end{equation}
The tilted prefix density relative to the original product prefix law
is $A_j(S_j)/\mathcal Z^{\mathrm{row}}_v(t)$. Thus, using $A_j\le1$ and $d_j\ge\eta^2$,
\[
 \E_{\mu_{v,t}}[|r_j(S_j)|^2\mathbf1_{\{|S_j|>K\}}]
 \le C_\eta|v_j|^2\E_0[(1+|S_j|^2)\mathbf1_{\{|S_j|>K\}}]
 \le C_\eta|v_j|^2e^{-cK^2},
\]
where $\E_0$ evaluates the same residual function on original independent
prefix entries. Summing the tail bound over $j=1,\ldots,n$
and combining it with \eqref{u:eq:regression} yields
\begin{equation}\label{u:eq:finite-row-remainder}
 \E_{\mu_{v,t}}\sum_j|r_j(S_j)|^2
 \le C_{\eta,K}\Delta^2+C_\eta e^{-cK^2}.
\end{equation}
Consequently, for $s_{ij}=t_i-\sum_{k<j}v_kx_{ik}$ and
$\mathcal R_n=n^{-1/2}(r_j(s_{ij}))$,
$\E\norm{\mathcal R_n}_{\HS}^2\to0$, by first letting $n\to\infty$
at fixed $\eta,K$, then $K\to\infty$ with $\eta$ still fixed.

\paragraph{Coupling errors.}
We construct the coupling one entry at a time. For a fixed residual $s$,
write
\[
 f_{j,s}(x)=\frac{dK_{j,s}}{d\mu}(x)
          =\frac{\mathcal Z_j(s-v_jx)}{A_j(s)},\qquad
 \alpha_{j,s}=\int\min\{f_{j,s},1\}\dd\mu.
\]
Define a joint law on $\F\times\F$ by
\[
 \begin{aligned}
 \pi_{j,s}(\dd x,\dd y)
 &=\min\{f_{j,s}(x),1\}\mu(\dd x)\delta_x(\dd y)\\
 &\quad+\frac{(f_{j,s}(x)-1)_+(1-f_{j,s}(y))_+}{1-\alpha_{j,s}}
                   \mu(\dd x)\mu(\dd y).
 \end{aligned}
\]
Here $\delta_x$ is the point mass at $x$, and the second term is omitted
when $\alpha_{j,s}=1$. The first term, of mass $\alpha_{j,s}$, makes
the two entries equal by sampling their common part. With the remaining
probability, the entries are sampled independently from the two residual
laws. Integrating either coordinate shows that the marginals of
$\pi_{j,s}$ are $K_{j,s}$ and $\mu$. All these kernels are Borel in $s$.

Now generate pairs $(x_{ij},\zeta_{ij})$ in row-major order
$\ell=(i-1)n+j$, and let $\mathcal G_\ell$ be the sigma-algebra generated
by the first $\ell$ pairs, with $\mathcal G_0$ trivial. Given the past,
the residual $s_{ij}=t_i-\sum_{k<j}v_kx_{ik}$ is known; sample the next pair
from $\pi_{j,s_{ij}}$. Thus
\[
 \mathcal L(x_{ij}\mid\mathcal G_{\ell-1})=K_{j,s_{ij}},\qquad
 \mathcal L(\zeta_{ij}\mid\mathcal G_{\ell-1})=\mu.
\]
The first identity reproduces the independent tilted rows, while the
second makes the entire comparator array iid with law $\mu$, by iterated
conditioning. The two arrays may depend on one another. Since $\mu$ is
centered, the centered coupling errors satisfy
\begin{equation}\label{u:eq:coupling-error}
 e_{ij}:=x_{ij}-\zeta_{ij}-m_j(s_{ij}),\qquad
 \E(e_{ij}\mid\mathcal G_{\ell-1})=0.
\end{equation}
To bound their conditional variance, observe that the common part of
$\pi_{j,s}$ contributes zero to $|x-y|^2$. On its residual part use
$|x-y|^2\le2|x|^2+2|y|^2$; conditional centering then gives
\[
 \E(|e_{ij}|^2\mid\mathcal G_{\ell-1})
 \le\E(|x_{ij}-\zeta_{ij}|^2\mid\mathcal G_{\ell-1})
 \le2\int|x|^2|f_{j,s_{ij}}(x)-1|\dd\mu(x).
\]

It remains to estimate this weighted difference of densities.
For integrable $W\ge0$ and $c>0$, $\operatorname{Cov}(W,e^{-cW})\le0$:
average $(W-W')(e^{-cW}-e^{-cW'})\le0$ over independent copies.
Apply this with $W=|s-U_j|^2$. Differentiation and Cauchy--Schwarz give
\[
 |\nabla\log\mathcal Z_j(s)|
 \le\frac\beta{\eta^2}
 \left(\frac{\E[|s-U_j|^2\varphi(s-U_j)]}{\mathcal Z_j(s)}\right)^{1/2}
 \le\frac\beta{\eta^2}\sqrt{\E|s-U_j|^2}
 \le C\eta^{-2}(1+|s|).
\]
Hence, with $u=|v_j|(1+|s|)$,
\[
 D(x):=\log\frac{\mathcal Z_j(s-v_jx)}{\mathcal Z_j(s)},\qquad
 |D(x)|\le C\eta^{-2}(u|x|+|v_j|^2|x|^2).
\]
For a sufficiently small $u_0=u_0(\mu,\eta)>0$, the square-exponential moment and
$|e^D-1|\le|D|e^{|D|}$ imply
$\int(1+|x|^2)|e^{D(x)}-1|\dd\mu(x)\le Cu$ whenever $u\le u_0$.
Normalizing by $\int e^D\dd\mu\ge1-Cu\ge1/2$ gives
\begin{equation}\label{u:eq:weighted-tv}
 \int(1+|x|^2)\left|\frac{dK_{j,s}}{d\mu}(x)-1\right|\dd\mu(x)
 \le C|v_j|(1+|s|),\qquad u\le u_0.
\end{equation}
Combining this with the preceding coupling bound gives
\[
 \E(|e_{ij}|^2\mid\mathcal G_{\ell-1})
 \le C|v_j|(1+|s_{ij}|),\qquad |v_j|(1+|s_{ij}|)\le u_0.
\]
Row-density domination, conditional Jensen applied to
$m_j(s_{ij})=\E(x_{ij}\mid\mathcal G_{\ell-1})$, and H\"older give
$\E e^{c|e_{ij}|^2}\le C$ after decreasing $c$.

To pass from these scalar errors to their matrix operator norm, truncate
and apply matrix Freedman. Choose a sufficiently large fixed $C_0$, set
$\Lambda_n=C_0\sqrt{\log n}$, and define
\[
 h_{ij}=e_{ij}\mathbf1_{\{|e_{ij}|\le\Lambda_n\}},
 \qquad \widetilde e_{ij}=h_{ij}
                  -\E(h_{ij}\mid\mathcal G_{\ell-1}).
\]
Conditional centering decreases variance, and
$e_{ij}-\widetilde e_{ij}$ is the conditionally centered tail.
Thus $|\widetilde e_{ij}|\le2\Lambda_n$ and
\[
 \E\norm{n^{-1/2}(e_{ij}-\widetilde e_{ij})}_{\HS}^2
 \le Cn e^{-c\Lambda_n^2}=o(1).
\]
The flatness assumption in Lemma~\ref{u:lem:coupling} and the chosen
truncation level give
\[
 \Delta(1+\Lambda_n)
 \le\frac{L(1+C_0\sqrt{\log n})}{\sqrt n}\longrightarrow0.
\]
Consequently, on $\{\max_{i,j}|s_{ij}|\le\Lambda_n\}$, the condition
$|v_j|(1+|s_{ij}|)\le u_0$ holds for every entry when $n$ is large,
so the preceding conditional variance bound applies throughout.

Now fill the matrix $n^{-1/2}(\widetilde e_{ij})$ in row-major order,
keeping unrevealed entries zero. These partial matrices form a martingale
with respect to $\mathcal G_\ell$. Each increment changes just one entry,
so its operator norm is at most $b_n=2\Lambda_n/\sqrt n$.
The two predictable quadratic variation matrices are diagonal: their
diagonal entries are the row sums and column sums of the conditional
entry variances. Thus, on the same event, their operator norms are bounded by
\[
 \max\left\{
 \max_i\frac1n\sum_j\E(|\widetilde e_{ij}|^2\mid\mathcal G_{\ell-1}),
 \max_j\frac1n\sum_i\E(|\widetilde e_{ij}|^2\mid\mathcal G_{\ell-1})
 \right\}
 \le V_n:=C(1+\Lambda_n)\Delta,
\]
where $\ell=(i-1)n+j$ in each summand; we used $|v_j|\le\Delta$.
In particular $V_n\le C\sqrt{\log n/n}$.
The complementary event has probability at most $Cn^2e^{-c\Lambda_n^2}$.
Matrix Freedman in its joint-event form \cite[Corollary 1.3]{Tropp}
therefore gives, for every fixed $u>0$,
\[
 \Prob\{\norm{n^{-1/2}(\widetilde e_{ij})}_{\op}>u\}
 \le Cn^2e^{-c\Lambda_n^2}
       +2n\exp\left[-\frac{u^2}{2(V_n+b_nu/3)}\right]\longrightarrow0.
\]
Removing truncation proves
$\norm{\mathcal M_n}_{\op}\xrightarrow{\Prob}0$ for
$\mathcal M_n=n^{-1/2}(e_{ij})$.

\paragraph{Triangular system.}
It remains to solve the coordinate regressions and assemble the
decomposition in Lemma~\ref{u:lem:coupling}: an iid array multiplied by $T$,
a rank-one drift, and a vanishing error. Combining the centered-error
definition \eqref{u:eq:coupling-error} with the regression decomposition
\eqref{u:eq:regression}, and substituting
$s_{ij}=t_i-\sum_{k<j}v_kx_{ik}$, gives
\begin{align*}
 x_{ij}&=\zeta_{ij}+m_j(s_{ij})+e_{ij}\\
 &=\zeta_{ij}+\frac{\bar v_j}{d_j}
       \left(t_i-\sum_{k<j}v_kx_{ik}\right)+e_{ij}+r_j(s_{ij}).
\end{align*}
Moving all terms involving $\bm x_i$ to the left gives a triangular system:
the $j$th equation involves only $x_{i1},\ldots,x_{ij}$.
Let $\mathsf L_{jj}=1$, $\mathsf L_{kj}=v_k\bar v_j/d_j$ for $k<j$,
and zero below the diagonal; define the row vector
$\bm a:=(\bar v_1/d_1,\ldots,\bar v_n/d_n)\in\F^{1\times n}$.
The bold symbols $\bm x_i,\bm\zeta_i,\bm e_i,\bm r_i\in\F^{1\times n}$
denote row vectors with components
\[
 (\bm x_i)_j=x_{ij},\quad (\bm\zeta_i)_j=\zeta_{ij},\quad
 (\bm e_i)_j=e_{ij},\quad (\bm r_i)_j=r_{ij}:=r_j(s_{ij}).
\]
Then
\[
 \bm x_i\mathsf L=\bm\zeta_i+t_i\bm a+\bm e_i+\bm r_i.
\]
The matrix $\mathsf L$ is invertible because it is upper triangular with
diagonal entries one. We identify its inverse to recover the specific
$T$ and drift in \eqref{u:eq:coupling}--\eqref{u:eq:T}.
For the candidate $T$ in \eqref{u:eq:T}, the identity
$|v_\ell|^2/(d_\ell d_{\ell+1})=d_{\ell+1}^{-1}-d_\ell^{-1}$ gives
\[
 (\mathsf LT)_{kj}
 =v_k\bar v_j\left(\frac1{d_j}-\frac1{d_{k+1}}
       -\sum_{k<\ell<j}\frac{|v_\ell|^2}{d_\ell d_{\ell+1}}\right)
 =0,\qquad k<j.
\]
The diagonal entries are one and those below the diagonal are zero,
so $\mathsf LT=I$. The same telescoping sum identifies the drift:
\[
 (\bm aT)_j=\bar v_j\left(\frac1{d_j}
             -\sum_{k<j}\frac{|v_k|^2}{d_kd_{k+1}}\right)
 =\frac{\bar v_j}{d_1}=\frac{\bar v_j}{1+\eta^2}.
\]
Thus solving the system gives
\[
 \bm x_i=\bm\zeta_iT+\frac{t_i v^*}{1+\eta^2}+(\bm e_i+\bm r_i)T.
\]
To finish, we check that multiplication by $T$ preserves the smallness
of the two errors. Its explicit entries give
\[
 \norm{T-I}_{\HS}^2
 =\sum_{k<j}\frac{|v_k|^2|v_j|^2}{d_{k+1}^2}
 \le\eta^{-4}\frac{1-\sum_j|v_j|^4}{2}\le(2\eta^4)^{-1},
\]
which is the bound in \eqref{u:eq:T}; hence
$\norm T_{\op}\le1+1/(\sqrt2\eta^2)$ for fixed $\eta$.
Stacking the solved rows and dividing by $\sqrt n$ gives
\eqref{u:eq:coupling} with
\[
 E_n=\frac1{\sqrt n}(e_{ij}+r_{ij})_{i,j=1}^n T
    =(\mathcal M_n+\mathcal R_n)T.
\]
Indeed,
the preceding two parts proved that $\norm{\mathcal M_n}_{\op}$ and
$\norm{\mathcal R_n}_{\HS}$ tend to zero in probability, so
\[
 \norm{E_n}_{\op}
 \le\bigl(\norm{\mathcal M_n}_{\op}+\norm{\mathcal R_n}_{\HS}\bigr)
       \left(1+\frac1{\sqrt2\eta^2}\right)
 \xrightarrow{\Prob}0.
\]
\end{proof}

\subsection{Proof of the conditional real Schur distribution}\label{u:sec:schur-distribution}
\begin{proof}[Proof of Fact~\ref{u:fact:conditional-schur}]
We reproduce the Schur Jacobian calculation of
\cite[Theorem 5.1 and Lemma 5.2]{Edelman} for completeness.
The real Schur theorem gives diagonal blocks of size one or two.
In a two-dimensional block, choosing axes at $45^\circ$ to eigenvectors
of its symmetric part makes the diagonal entries equal to $x_a$.
The determinant then gives $b_ac_a=y_a^2>0$; a reflection and a
quarter-turn arrange $b_a\ge c_a>0$. For a fixed ordering and simple
spectrum, this fixes the block bases up to signs whenever $b_a>c_a$.
Repeated eigenvalues and the lower-dimensional sets $b_a=c_a$ have
Lebesgue measure zero and may be discarded.

\emph{The Jacobian between blocks.}
Here $d$ denotes the formal first-order differential in the real Schur
coordinates: $dQ$, $dU$, and $dG_n$ are matrices of differential
one-forms, and the usual product rule applies. The Jacobian is the
determinant of the resulting linear map between coordinate differentials.
Set $\Omega=Q^\top dQ$, so $\Omega^\top=-\Omega$. Differentiation gives
\[
 D:=Q^\top(dG_n)Q=dU+\Omega U-U\Omega.
\]
Orthogonal conjugation has Jacobian of absolute value one on matrix
space, so it suffices to compute the Jacobian of this differential.
For a lower block $a>b$, upper triangularity of $U$ gives
\[
 D_{ab}=\Omega_{ab}\mathsf S_b-\mathsf S_a\Omega_{ab}
       +\sum_{c<b}\Omega_{ac}N_{cb}
       -\sum_{c>a}N_{ac}\Omega_{cb}.
\]
The rotation blocks in both sums lie farther below the diagonal than
$(a,b)$. Ordering the pairs by decreasing $a-b$ therefore makes this
linear system block triangular. Its diagonal maps, acting on
$H\in\R^{d_a\times d_b}$, satisfy
\[
 \mathcal L_{ab}(H)=H\mathsf S_b-\mathsf S_aH,
 \qquad
 |\det\mathcal L_{ab}|=
 \prod_{\lambda\in\operatorname{spec}(\mathsf S_a)}
 \prod_{\lambda'\in\operatorname{spec}(\mathsf S_b)}|\lambda'-\lambda|.
\]
Indeed, diagonalizing the two blocks over $\C$ makes
$\mathcal L_{ab}$ diagonal on matrix units, with eigenvalues
$\lambda'-\lambda$. Thus all these factors depend only on the spectrum,
not on $N_{ab}$ or $\Delta_a$.

\emph{The Jacobian inside a block.}
After eliminating the rotation variables just treated, a scalar
diagonal block contributes $d\lambda_a$. For a two-dimensional block,
write $\Omega_{aa}=\left(\begin{smallmatrix}0&d\theta_a\\
-d\theta_a&0\end{smallmatrix}\right)$. Its remaining differential is
\[
 d\mathsf S_a+[\Omega_{aa},\mathsf S_a]
 =\begin{pmatrix}
 dx_a+\Delta_a d\theta_a&db_a\\
 -dc_a&dx_a-\Delta_a d\theta_a
 \end{pmatrix},
\]
whose Jacobian with respect to $(x_a,\theta_a,b_a,c_a)$ has absolute
value $2\Delta_a$. Finally, each upper block has differential
$dN_{ab}$ plus already treated rotation terms, and hence contributes
Jacobian one. Relative to Haar volume for $Q$ and Lebesgue measure for
$N_{ab}$, the real $\lambda_a$, and the triples $(x_a,b_a,c_a)$, the
full Jacobian is therefore, up to a constant depending only on the
block pattern,
\[
 J=\prod_{a>b}|\det\mathcal L_{ab}|
       \prod_{a:d_a=2}2\Delta_a.
\]
The sign convention restricts $Q$ to a fixed domain independent of
the other coordinates; its volume is another constant.

\emph{Conditional density and moments.}
For each complex block, the change from $(b,c)$ to
$(y,\Delta)=(\sqrt{bc},b-c)$ satisfies
\[
 \left|\frac{\partial(b,c)}{\partial(y,\Delta)}\right|
 =\frac{2y}{\sqrt{\Delta^2+4y^2}}.
\]
This follows by differentiating $bc=y^2$ and $b-c=\Delta$, using
$b+c=\sqrt{\Delta^2+4y^2}$. Moreover,
\[
 \norm{G_n}_{\HS}^2=\norm U_{\HS}^2
 =\sum_{a:d_a=1}\lambda_a^2
  +\sum_{a:d_a=2}(2x_a^2+2y_a^2+\Delta_a^2)
  +\sum_{a<b}\norm{N_{ab}}_{\HS}^2.
\]
Multiply the Jacobian by the Gaussian density
$\exp(-n\norm{G_n}_{\HS}^2/2)$ and integrate out $Q$; this is the
density factorization in \cite[Section 6.1, equation (10)]{Edelman},
with entries scaled to variance $1/n$.
After fixing the spectrum, all factors depending only on eigenvalues
enter the normalizing constant. The remaining density is proportional to
\[
 \prod_{a<b}e^{-n\norm{N_{ab}}_{\HS}^2/2}\dd N_{ab}
 \prod_{a:d_a=2}
 \frac{\Delta_a e^{-n\Delta_a^2/2}}
      {\sqrt{\Delta_a^2+4y_a^2}}\dd\Delta_a,
 \qquad \Delta_a>0.
\]
The domain is a Cartesian product of Euclidean spaces and positive
half-lines. Normalizing each factor proves the stated
densities and all conditional independence claims.

For the second moments, all expectations below are conditional on the
spectrum. Integrate the derivative of
$t^2(t^2+4y_a^2)^{-1/2}e^{-nt^2/2}$ over $(0,\infty)$.
The boundary terms vanish; dividing by $Z_{n,y_a}$ gives
\[
 n\E\Delta_a^2
 =2-\E\frac{\Delta_a^2}{\Delta_a^2+4y_a^2}\le2.
\]
Finally, $\E\norm{N_{ab}}_{\HS}^2=d_ad_b/n$ by summing the
variances of its scalar Gaussian entries. The argument applies to every
block pattern, including odd dimensions.
\end{proof}

\part{Lower tails}\label{part:lower}
\section{Lower-tail reduction and the border estimate}\label{sec:lower-reduction}
Two estimates prove the real lower-tail theorem. When the core has a
sufficiently large weighted Arnoldi link product,
Proposition~\ref{l:thm:MB} bounds the conditional probability of the
spectral event over the independent border.
Proposition~\ref{l:thm:ML} controls the complementary event; its proof is
in Section~\ref{sec:comparison-links}, while the border bound is proved
below.

\subsection{Lower-tail reduction and probability inputs}\label{l:sec:real_main}

In this part, \(\mathsf X_n\) denotes the unnormalized entry matrix from
\eqref{eq:common-normalization}; its normalized version remains
\(X_n=\mathsf X_n/\sqrt n\). Thus the event
\(\rho(\mathsf X_n)\le r\sqrt n\) is exactly \(\rho(X_n)\le r\).
We first treat real entries and then give the complex extension in
Section~\ref{l:sec:complex}.

Let \(\mu(dx)=f(x)\,dx\) be a probability law on \(\mathbb R\), where
\(f\) is Borel measurable and, for some finite \(K_f\) and some
\(t_f>0\), \begin{equation}
0\le f\le K_f,\qquad
 \mathbb E\xi=0,\qquad \mathbb E\xi^2=1,\qquad
 \mathbb E e^{t_f\xi^2}<\infty,\qquad \xi\sim\mu.
 \label{l:eq:M.1}
\end{equation} Fix a bounded representative of the density, changing it on a null
set if necessary. No lower bound or smoothness is assumed.

We prove the real case of Theorem~\ref{l:thm:MT} in
Subsection~\ref{l:sec:mainproof}, using the two probability estimates below.
The complex case is completed in Section~\ref{l:sec:complex}.

Write \(\operatorname{Ran}M\) for column space, \(P_E\) for orthogonal
projection onto \(E\), and \(M^T\) for transpose. Matrix restrictions
\(M_{I,:}\), \(M_{:,J}\), \(M_{I,J}\) specify rows, columns, or both; a
colon retains all indices, in their original order. Set
\([s]=\{1,\ldots,s\}\), \([0]=\varnothing\), and use \([a,b]\) for
integer index sets in restrictions. For an index set \(J\),
\(\mathbb R^J\) is the space of vectors indexed by \(J\), and
\(\mathrm{Id}_J\) is its identity. Complements are relative to the
ambient index set; \(\mathrm{Id}_s\) is the identity in dimension \(s\).
The random deletion output is \(I_{\mathrm{del}}\); its
deterministic candidates are \(I_0\), with complement \(J_0\).

\subsubsection{The core, the starter, and the
border}\label{l:the-core-the-starter-and-the-border}

Here \(n\) denotes the size of the full matrix and \(d:=n-1\) the size
of its core. We use \(d\) for core dimensions and retain \(\sqrt n\) for
the spectral normalization. Split the \emph{unnormalized} matrix as \begin{equation}
\mathsf X_n=\begin{pmatrix}\xi_{11}&\mathbf y^T\\\mathbf x&A\end{pmatrix},
 \qquad A\in\mathbb R^{d\times d}.
 \label{l:eq:M.3}
\end{equation} The core \(A\), border \(\mathbf b=(\xi_{11},\mathbf x,\mathbf y)\), and
starter \(\mathbf s\sim\mu^{\otimes d}\) are independent. The starter is
an additional random vector, not a row or column of \(\mathsf X_n\); in
particular it is neither \(\mathbf x\) nor \(\mathbf y\). It chooses the
Arnoldi basis and does not enter the spectral event. Thus
\(\mathbb P_{A,\mathbf s}\) samples the \(d^2\) core entries and \(d\)
additional starter entries, whereas \(\mathbb P_{\mathbf b}\) samples
the \(2d+1\) border entries. We also write
\(\mathbb P_\mu=\mathbb P_{A,\mathbf s}\) for the joint original law of
the core and starter; \(\mathbb P_{\mathbf b}\) is the separate border law.

\Needspace{6\baselineskip}
\textbf{Only the core is queried.} Algorithm \ref{l:thm:ARN} acts on the
bottom-right matrix \(A\), with the extra independent vector
\(\mathbf s\) used only to initialize it. All its directions, responses
and residual lengths are functions of \((A,\mathbf s)\). We reserve the
border for a separate step: after fixing \((A,\mathbf s)\), its \(2d+1\)
entries still have their independent original laws, so Proposition \ref{l:thm:MB}
can estimate the spectral event over the border alone. Querying the full
matrix \(\mathsf X_n\) would generally make the query record and basis depend on
that border, preventing this direct conditional use of independence. The
spectral event throughout concerns \(\mathsf X_n\); the core supplies the
link-product condition for this border estimate, not a substitute
spectral-radius event.

Choose a small constant \(\alpha=\alpha(\mu,r)>0\), fixed later, and set
\[
 k=2\lfloor\alpha d/2\rfloor,\qquad m=k/2.
\] Thus \(k\) is even, \(k/d\to\alpha>0\), and \(2\le k<d\) for large
\(d\). All algorithms use this deterministic query length.

We use the classical Arnoldi orthogonalization \cite{Arnoldi1951},
with an explicit completion convention when a residual vanishes.

\begin{algorithms}[Arnoldi queries of the core]\label{l:thm:ARN}

Given the core \(A\in\mathbb R^{d\times d}\), the independent starter
\(\mathbf s\in\mathbb R^d\), and \(k<d\), set
\(\mathbf v_1=\mathbf s/\|\mathbf s\|\) when \(\mathbf s\ne0\), and \(\mathbf v_1=e_1\)
otherwise. For \(j=1,\ldots,k\), observe \(A\mathbf v_j\) and set \[
 P_j=\sum_{i=1}^j\mathbf v_i\mathbf v_i^T,\qquad
 \mathbf r_j=(\mathrm{Id}_d-P_j)A\mathbf v_j,\qquad R_j=\|\mathbf r_j\|.
\] Set \[
 \mathbf v_{j+1}=\begin{cases}
 \displaystyle\frac{\mathbf r_j}{R_j},&R_j>0,\\[6pt]
 \displaystyle\frac{(\mathrm{Id}_d-P_j)e_{i_j}}
                    {\|(\mathrm{Id}_d-P_j)e_{i_j}\|},&R_j=0,
 \end{cases}
 \qquad
 i_j=\min\{i\in[d]:(\mathrm{Id}_d-P_j)e_i\ne0\}.
\] The \(R_j=0\) branch is the deterministic
\emph{orthogonal completion} rule. The index \(i_j\) exists since
\(j<d\). This branch adds one orthogonal unit direction while retaining
all previous directions and responses; \(R_j\) remains zero and no new
starter is sampled. Continue with the next step. Record \[
 V=[\mathbf v_1,\ldots,\mathbf v_k],\qquad V^TV=\mathrm{Id}_k,
\] the responses \(A\mathbf v_1,\ldots,A\mathbf v_k\), and the residual norms. We use
this deterministic orthogonal completion rule below as well.
\end{algorithms}

The construction of the full orthogonal basis and its Hessenberg form
for the border calculation is given in Section \ref{l:sec:real_border}.

\subsubsection{A link estimate under the original
law}\label{l:a-link-estimate-under-the-original-law}

For \(2\le k<d\), define \begin{equation}
w_j=\min\{j,k-j\}\ (1\le j<k),\qquad
 M_k=\sum_{j=1}^{k-1}w_j=\lfloor k^2/4\rfloor,\qquad
 W_k(A,\mathbf s)=\prod_{j=1}^{k-1}(R_j/\sqrt n)^{w_j}.
 \label{l:eq:M.6}
\end{equation} We normalize by the full matrix scale \(\sqrt n=\sqrt{d+1}\).

\begin{proposition}[Original-law weighted links]\label{l:thm:ML}

For each \(0<\varepsilon\le1\), there exist
\(\beta_{\mu,\varepsilon}>0\), \(c_{\mathrm{link}}(\mu,\varepsilon)>0\)
and \(d_0(\mu,\varepsilon)\) such that, for every \(d\ge d_0\), one can
choose an integer \(m\ge1\) satisfying \begin{equation}
2m<d,\qquad \beta_{\mu,\varepsilon}d\le 2m,
 \qquad 256m\le\varepsilon d,
 \label{l:eq:M.7}
\end{equation} for which \begin{equation}
\mathbb P_{A,\mathbf s}\{W_{2m}(A,\mathbf s)\le(1-\varepsilon/4)^{m^2}\}
 \le e^{-c_{\mathrm{link}}(\mu,\varepsilon)d^2}.
 \label{l:eq:M.8}
\end{equation} Here \(A\) and \(\mathbf s\) have their actual independent iid laws
\(\mu\). The main proof uses only the even query length \(2m\), for
which \(M_{2m}=m^2\). The general parameter \(k\) in \eqref{l:eq:M.6} is retained
only for the reusable link and stopping estimates below; their
application here is \(k=2m\). The estimate is unconditional in the core
and starter; the border has not been sampled in its proof. \end{proposition}

Section \ref{l:sec:real_stopping} proves Proposition \ref{l:thm:ML} using the
comparison theorem stated in Section \ref{l:sec:real_comparison}.

\subsubsection{A border estimate with the core held
fixed}\label{l:a-border-estimate-with-the-core-held-fixed}

The next estimate holds with a finite law-dependent constant
\(C_{\mathrm{border}}(f)>0\); Section \ref{l:sec:real_border} proves it by
Gaussian border comparison.

\begin{proposition}[Good-core border probability]\label{l:thm:MB}

Fix \(0\le r<R\le1\), \(a>0\), and an integer \(m\ge1\) with \(2m<d\)
and \((r/R)^m\le1/2\). Fix a deterministic core
\(A\in\mathbb R^{d\times d}\) and starter \(\mathbf s\in\mathbb R^d\)
with \(W_{2m}(A,\mathbf s)\ge a^{m^2}\), where \(d=n-1\). The
probability \(\mathbb P_{\mathbf b}\) below is over precisely the
\(2d+1=2n-1\) independent scalar entries \[
 \mathbf b=(\xi_{11},x_1,\ldots,x_d,y_1,\ldots,y_d)\sim\mu^{\otimes(2d+1)}.
\] Then \begin{equation}
\mathbb P_{\mathbf b}\left\{
  \rho\begin{pmatrix}\xi_{11}&\mathbf y^T\\\mathbf x&A\end{pmatrix}\le r\sqrt n
 \right\}
 \le\exp\left\{
    C_{\mathrm{border}}(f)n+\frac{m^2}{4}\log(R/a)
   \right\}.
 \label{l:eq:M.10}
\end{equation} Explicitly, for a border event \(E\), \[
 \mathbb P_{\mathbf b}(E)=
 \int_{\mathbb R^{2d+1}}\mathbf1_E(\xi_{11},\mathbf x,\mathbf y)\,
 f(\xi_{11})\prod_{i=1}^d f(x_i)f(y_i)\,d\xi_{11}\,d\mathbf x\,d\mathbf y.
\] There is no averaging over \(A\) or \(\mathbf s\), no random change
of basis, and no Gaussian border in this statement. A Gaussian border is
introduced only as a comparison device in its proof. \end{proposition}

\begin{example}[A nilpotent core with strong links]\label{l:thm:NILCORE}

Consider the normalized companion matrix \[
 C_n=\begin{pmatrix}
 \xi_1&\xi_2&\cdots&\xi_n\\
 1&0&\cdots&0\\
 &\ddots&\ddots&\vdots\\
 0&&1&0
 \end{pmatrix},\qquad \xi_j\ \text{independent uniform on }[-1,1].
\] Its core is the nilpotent shift \(S_de_j=e_{j+1}\) for \(1\le j<d\), with \(S_de_d=0\). In
\eqref{l:eq:M.3} this corresponds to \(A=\sqrt nS_d\); with starter \(e_1\),
Arnoldi gives \(R_j=\sqrt n\) and \(W_{2m}=1\), although \(\rho(A)=0\).
Since \(\det(z\mathrm{Id}_n-C_n)=z^n-\sum_{j=1}^n\xi_jz^{n-j}\), Vieta's
formula gives \(|\xi_j|\le\binom njr^j\) whenever all roots have modulus
at most \(r\). Independence therefore yields \[
 \mathbb P\{\rho(C_n)\le r\}
 \le\prod_{j=1}^n\min\{1,\tbinom njr^j\}
 \le2^{n^2}r^{n(n+1)/2}\le2^{-n^2/2}\quad(r=1/8).
\] Thus \(n\) variables can incur an \(n^2\)-scale cost through
accumulated precision. This illustrates the mechanism; Proposition \ref{l:thm:MB}
uses a random border on both sides and a different normalization. For a
general good core, two border integrations replace this
independent-coefficient argument. \end{example}

The Gaussian border calculation gives an exponent of \(m^2/2\);
comparison with the original border law halves it to \(m^2/4\).
The positive term in the resulting exponent is only linear in \(n\).

\subsubsection{Proof of the real lower-tail theorem}\label{l:sec:mainproof}

\begin{proof}[\proofname\ of Theorem \ref{l:thm:MT}, real case]Fix \(r<1\), set
\(\varepsilon=1-r\), \(a=1-\varepsilon/4\), and \(R=(r+a)/2\). Then
\(r<R<a<1\). Proposition \ref{l:thm:ML} supplies \(m\ge\beta_{\mu,\varepsilon}d/2\),
so \((r/R)^m\le1/2\) for large \(d\). On \(W_{2m}>a^{m^2}\), Proposition
\ref{l:thm:MB} gives \[
 \mathbb P_{\mathbf b}\{\rho(\mathsf X_n)\le r\sqrt n\}
 \le e^{C_{\mathrm{border}}(f)n-\delta_rm^2}
 \le e^{-c_{\mathrm{good}}d^2},\qquad
 \delta_r=\tfrac14\log(a/R)>0,
\] for example with
\(c_{\mathrm{good}}=\delta_r\beta_{\mu,\varepsilon}^2/8\) and a
sufficiently large dimension threshold. The starter does not enter the
spectral event. Averaging over \((A,\mathbf s)\), using the trivial
bound one on the bad-core event, gives \[
 \mathbb P\{\rho(\mathsf X_n)\le r\sqrt n\}
 \le e^{-c_{\mathrm{link}}d^2}+e^{-c_{\mathrm{good}}d^2}.
\] Since \(d=n-1\), this proves \eqref{l:eq:M.2} after adjusting the rate and
threshold. \end{proof}

\subsubsection{Reading guide and external
inputs}\label{l:reading-guide-and-external-inputs}

The proof in Subsection \ref{l:sec:mainproof} is the final reduction
assuming Propositions \ref{l:thm:ML} and \ref{l:thm:MB}. Section \ref{l:sec:real_border} first
establishes \ref{l:thm:MB}: it explains why a strong link product is useful for the
spectral event. Section \ref{l:sec:real_comparison} states Theorem \ref{l:thm:AC},
develops exact reconstruction and its flat-coordinate consequence, and
proves the mixed-reference estimate in Subsection
\ref{l:sec:real_reference}.

Section \ref{l:sec:real_stopping} supplies the geometric estimates, proves
\ref{l:thm:SS}, and completes \ref{l:thm:ML} in Subsection \ref{l:sec:real_link_completion}. Only
the analytic proof of Theorem \ref{l:thm:AC} remains deferred; Section
\ref{l:sec:real_comparison_analysis} gives it using the likelihood,
Fourier and sequential coupling estimates, independently of \ref{l:thm:SS} and \ref{l:thm:ML}.
Section \ref{l:sec:complex} proves the complex extension.

The substantive external inequalities are geometric Brascamp--Lieb
\cite[Proposition 2.8]{BCCT2008} and one-dimensional Gaussian
transport--entropy \cite[Theorem 2.9 and Proposition 5.9]{BBLZ2017}. We
derive the projection, singular-value, Gaussian link and border bounds
here. The remaining tools are finite-dimensional integration, Fourier
inversion, Doob's maximal inequality and nonnegativity of relative
entropy.

\subsection{The border calculation}\label{l:sec:real_border}

With the core and starter fixed, we first express the core in Hessenberg
form and compare only the border with a Gaussian border.
Lemma~\ref{l:lem:gaussian-border} states the Gaussian estimate and gives
a short proof of Proposition~\ref{l:thm:MB}. The remaining subsections
prove that estimate by an evaluation determinant and two Gaussian
integrations.

Fix the core \(A\) and starter \(\mathbf s\) as in Proposition \ref{l:thm:MB}.
Continue the recursion of Algorithm \ref{l:thm:ARN} through step \(d-1\), including
its orthogonal completion rule at zero residuals, to obtain
\(Q=[\mathbf v_1,\ldots,\mathbf v_d]\).

\begin{fact}[Hessenberg form of the core]\label{l:thm:HESS}

The matrix \(Q\) is orthogonal, and the transformed core \begin{equation}
B=(Q^TAQ)^T/\sqrt n
 \label{l:eq:M.4}
\end{equation} is lower Hessenberg, with \[
 B_{ij}=0\quad(j>i+1),\qquad
 B_{j,j+1}=R_j/\sqrt n\quad(1\le j<d).
\] \end{fact}

\begin{proof}

The construction makes \(\mathbf v_1,\ldots,\mathbf v_d\) orthonormal. For \(j<d\), \[
 A\mathbf v_j=\sum_{i=1}^j\langle \mathbf v_i,A\mathbf v_j\rangle \mathbf v_i+R_j\mathbf v_{j+1}.
\] This identity also holds when \(R_j=0\), since then \(A\mathbf v_j=P_jA\mathbf v_j\).
Thus \(Q^TAQ\) is upper Hessenberg and its \((j+1,j)\)-entry is \(R_j\).
Transposing and dividing by \(\sqrt n\) proves both claims. \end{proof}

Conjugating the whole matrix and then transposing gives the exact
identity \begin{equation}
\left[
  \begin{pmatrix}1&0\\0&Q^T\end{pmatrix}
  \frac{\mathsf X_n}{\sqrt n}
  \begin{pmatrix}1&0\\0&Q\end{pmatrix}
 \right]^T
 =\begin{pmatrix}
   \xi_{11}/\sqrt n&(Q^T\mathbf x)^T/\sqrt n\\
   Q^T\mathbf y/\sqrt n&B
  \end{pmatrix}.
 \label{l:eq:M.5}
\end{equation} Both operations preserve the characteristic polynomial. After fixing
\((A,\mathbf s)\), the \emph{original} border coordinates are still
independent; the rotated non-Gaussian coordinates on the right need not
be.

\subsubsection{Reduction to a Gaussian border}\label{l:gaussian-comparison-and-inverse-moments}

Choose \(\kappa>0\) with \(\mathbb E e^{\xi^2/(2\kappa)}<\infty\), for
example \(\kappa=(2t_f)^{-1}\). If \(\phi_\kappa\) is the
\(N(0,\kappa)\) density, then \begin{equation}
\int_{\mathbb R}\frac{f^2}{\phi_\kappa}
 \le K_f\sqrt{2\pi\kappa}\,\mathbb E e^{\xi^2/(2\kappa)}=:J_f<\infty.
 \label{l:eq:D.1}
\end{equation} Throughout the comparison, keep the revealed \((A,\mathbf s)\), and
hence \(Q,B\), fixed. A border event \(E\) is a Borel event depending
only on the \(2n-1\) border entries \((\xi_{11},\mathbf x,\mathbf y)\), with
these revealed data held fixed. Here \(\mathbb P_{\mathbf b}\) and
\(\mathbb P_\kappa\) denote the product laws \(\mu^{\otimes(2n-1)}\) and
\(N(0,\kappa)^{\otimes(2n-1)}\) on those entries; the core \(B\) is the
same in both. The product likelihood and Cauchy--Schwarz give \begin{equation}
\mathbb P_{\mathbf b}(E)\le J_f^{(2n-1)/2}\mathbb P_\kappa(E)^{1/2}
                  \le e^{J_fn}\mathbb P_\kappa(E)^{1/2}.
 \label{l:eq:D.2}
\end{equation} Here \(J_f\ge1\) by Cauchy--Schwarz. After this comparison in the
original coordinates, the fixed orthogonal changes in \eqref{l:eq:M.5} preserve the
Gaussian product law. We may thus work with \begin{equation}
T=\begin{pmatrix}g_{\mathrm b}&\mathbf x_{\mathrm b}^T\\\mathbf u_{\mathrm b}&B\end{pmatrix},
 \qquad g_{\mathrm b},(\mathbf x_{\mathrm b})_j,(\mathbf u_{\mathrm b})_j
       \ \text{independent }N(0,\kappa/n),\quad 1\le j\le d=n-1,
 \label{l:eq:D.3}
\end{equation} where \(B=(Q^TAQ)^T/\sqrt n\) is the fixed lower-Hessenberg core.
The notation \(T\) denotes the transformed full matrix in
\eqref{l:eq:M.5} under either border law; the independent Gaussian
distributions in \eqref{l:eq:D.3} describe its border under
\(\mathbb P_\kappa\). We write \(\mathbb E_\kappa\) for expectation
under this Gaussian border law, with the core held fixed.

In the next lemma, \(m\) is a freely chosen integer: the proof integrates
\(m\) row coordinates and \(m\) column coordinates, using \(m\)
polynomial evaluation points. The same Hessenberg cofactor factorization
\eqref{l:eq:D.6} works for every such choice. The radius \(R\) is an
auxiliary evaluation scale, not a bound on the spectrum of \(B\).
The estimate holds for every admissible \(m,R\), with the same constant,
but need not be smaller than one: its usefulness depends on the
corresponding link product \(W_{2m}\), which itself changes with \(m\).

\begin{lemma}[Gaussian border bound in terms of Hessenberg links]\label{l:lem:gaussian-border}
Let \(B\in\mathbb R^{d\times d}\) be a fixed lower-Hessenberg matrix,
\(n=d+1\), and let \(T\) have the independent Gaussian border in
\eqref{l:eq:D.3}, with \(\kappa>0\). For every integer \(m\ge1\) with
\(2m<d\) and every pair \(0\le r<R\le1\) with \((r/R)^m\le1/2\), set
\[
 W_{2m}=\prod_{j=1}^{2m-1}|B_{j,j+1}|^{\min(j,2m-j)}.
\]
If \(W_{2m}>0\), then, for a constant \(C_{\mathrm G}(\kappa)\)
independent of \(B,n,m,r,R\),
\begin{equation}\label{l:eq:gaussian-border}
 \mathbb P_\kappa\{\rho(T)\le r\}
 \le e^{C_{\mathrm G}(\kappa)n}
       \left(\frac{R^{m^2}}{W_{2m}}\right)^{1/2}.
\end{equation}
\end{lemma}

\begin{proof}[\proofname\ of Proposition \ref{l:thm:MB}]
The Hessenberg core in \eqref{l:eq:M.4} has
\(B_{j,j+1}=R_j/\sqrt n\), so its \(W_{2m}\) is exactly the Arnoldi
link product in the hypothesis. Substituting \(W_{2m}\ge a^{m^2}\)
in Lemma~\ref{l:lem:gaussian-border} gives
\begin{equation}\label{l:eq:D.15}
 \mathbb P_\kappa\{\rho(T)\le r\}
 \le\exp\{C_{\mathrm G}(\kappa)n+(m^2/2)\log(R/a)\}.
\end{equation}
The transformation \eqref{l:eq:M.5} preserves the normalized spectrum.
Applying the border comparison \eqref{l:eq:D.2} therefore yields
\[
\begin{aligned}
 \mathbb P_{\mathbf b}\{\rho(\mathsf X_n)\le r\sqrt n\}
 &\le e^{J_fn}\bigl[\mathbb P_\kappa\{\rho(T)\le r\}\bigr]^{1/2}\\
 &\le\exp\left\{\left(J_f+\tfrac12C_{\mathrm G}(\kappa)\right)n
                   +\frac{m^2}{4}\log(R/a)\right\}.
\end{aligned}
\]
This proves \eqref{l:eq:M.10}, with
\(C_{\mathrm{border}}(f)=J_f+C_{\mathrm G}(\kappa)/2\).
\end{proof}

\subsubsection{Gaussian inverse moments}\label{l:gaussian-inverse-moments}

The first border integration uses a negative moment of a product of
Gaussian evaluations. The second uses a negative moment of a shifted
determinant. The lemma gives these two estimates in that order.

\begin{lemma}[Gaussian inverse moments for conjugate-paired values]\label{l:thm:DN}
Let \(s,p\) be nonnegative integers with \(m=s+2p\ge1\), and let
\[
 \mathbf Z=(X_1,\ldots,X_s,A_1,B_1,\ldots,A_p,B_p)\in\mathbb R^m
\]
be a real Gaussian vector with arbitrary mean and positive definite
covariance \(\Gamma\). Its coordinates may be correlated. Form the
\(m\) real or conjugate-paired values
\[
 \mathbf Y=(X_1,\ldots,X_s,
 A_1+iB_1,A_1-iB_1,\ldots,A_p+iB_p,A_p-iB_p).
\]
There is a universal \(C_*\ge\sqrt2\) such that:

\smallskip\noindent
\textnormal{(i) Product estimate (used for the row integration).}
\begin{equation}\label{l:eq:D.5}
 \mathbb E\prod_{h=1}^m|Y_h|^{-1/2}
 =\mathbb E\left[
   \prod_{j=1}^s|X_j|^{-1/2}
   \prod_{j=1}^p(A_j^2+B_j^2)^{-1/2}\right]
 \le C_*^m(\det\Gamma)^{-1/4}.
\end{equation}

\smallskip\noindent
\textnormal{(ii) Shifted determinant estimate (used for the column integration).}
For \(\mathbf Y\) above, define the real block-diagonal matrix
\[
 \mathcal D(\mathbf Z)=\operatorname{diag}
 \left(X_1,\ldots,X_s,
 \begin{pmatrix}A_1&-B_1\\B_1&A_1\end{pmatrix},\ldots,
 \begin{pmatrix}A_p&-B_p\\B_p&A_p\end{pmatrix}\right).
\]
Denote the displayed two-dimensional blocks by \(\mathcal M(A_j,B_j)\).
Each represents multiplication by \(A_j+iB_j\)
in real coordinates:
\[
 (A_j+iB_j)(u+iv)
 =(A_ju-B_jv)+i(B_ju+A_jv).
\]
Equivalently, \(\mathcal D(\mathbf Z)\) is the real-coordinate
representation of \(\operatorname{diag}(\mathbf Y)\) acting on
vectors with the same real/conjugate-pair pattern. For every fixed
\(D\in\mathbb R^{m\times m}\),
\begin{equation}\label{l:eq:D.4}
 \mathbb E|\det(\mathcal D(\mathbf Z)+D)|^{-1/2}
 \le C_*^m(\det\Gamma)^{-1/4}.
\end{equation}
The constant is uniform in the Gaussian mean and in \(D\).
Taking \(D=0\) recovers \textnormal{(i)}, since
\(\det\mathcal D(\mathbf Z)=\prod_{j=1}^sX_j
\prod_{j=1}^p(A_j^2+B_j^2)\).
\end{lemma}

\begin{proof}

We prove \textnormal{(ii)} first. Let \(G\sim N(\mu_G,v)\), with
\(v>0\). For every fixed shift \(d_{\mathrm{sc}}\in\mathbb R\)
and every \(t>0\), the Gaussian density bound gives
\[
 \mathbb P\{|G+d_{\mathrm{sc}}|<t\}
 \le\min\left\{1,\frac{2t}{\sqrt{2\pi v}}\right\},
\]
uniformly in \(\mu_G\) and \(d_{\mathrm{sc}}\), since the event is
an interval of length \(2t\) and the density is at most
\((2\pi v)^{-1/2}\). For a single pair and a fixed real
\(2\times2\) shift \(D\), completing squares gives
\[
 \det(\mathcal M(x,y)+D)
 =(x+a_D)^2+(y+b_D)^2-c_D,\qquad c_D\ge0,
\]
where \(a_D,b_D,c_D\) depend only on \(D\). For \(t>0\), the set
\[
 \mathcal A_t=\{(x,y)\in\mathbb R^2:
                  |\det(\mathcal M(x,y)+D)|<t\}
\]
is an annulus centered at \((-a_D,-b_D)\), or a disk if the inner
radius vanishes. Its squared radii are \(c_D+t\) and
\(\max\{0,c_D-t\}\), so
\[
 \operatorname{area}(\mathcal A_t)
 =\pi\bigl(c_D+t-\max\{0,c_D-t\}\bigr)\le2\pi t.
\]
If \((G_1,G_2)\) is Gaussian with covariance \(S>0\), its density
is at most \((2\pi\sqrt{\det S})^{-1}\). Thus, uniformly in its mean
and in \(D\),
\[
 \mathbb P\{|\det(\mathcal M(G_1,G_2)+D)|<t\}
 \le\min\left\{1,\frac{t}{\sqrt{\det S}}\right\}.
\]
Integrating these bounds by \(\mathbb E X^{-1/2}=\tfrac12\int_0^\infty
 t^{-3/2}\mathbb P\{X<t\}\,dt\) gives the one-block estimates, with
factors \(v^{-1/4}\) and \((\det S)^{-1/4}\).

For more than one block, write
\(\mathbf Z=(\mathbf Z_1,\mathbf Z_{\mathrm{rem}})\), where
\(\mathbf Z_1\in\mathbb R^q\), \(q\in\{1,2\}\), is the first
scalar or pair block. Partition the fixed shift accordingly:
\[
 \mathcal D(\mathbf Z)+D
 =\begin{pmatrix}
   \mathcal D_1(\mathbf Z_1)+D_{11}&D_{12}\\
   D_{21}&\mathsf C(\mathbf Z_{\mathrm{rem}})
  \end{pmatrix},\qquad
 \mathsf C(\mathbf Z_{\mathrm{rem}})
 =\mathcal D_{\mathrm{rem}}(\mathbf Z_{\mathrm{rem}})+D_{22}.
\]
Here \(\mathcal D_1\) is the first block and
\(\mathcal D_{\mathrm{rem}}\) contains the remaining blocks.
The polynomial \(\det\mathsf C\) is not identically zero: its
highest-degree part is \(\det\mathcal D_{\mathrm{rem}}\), the product
of the remaining scalar coordinates and the pair factors
\(A_j^2+B_j^2\). Since \(\mathbf Z_{\mathrm{rem}}\) has a
nondegenerate Gaussian density, \(\mathsf C\) is invertible almost
surely.

Fix \(\mathbf Z_{\mathrm{rem}}\) where \(\mathsf C\) is invertible.
The Schur complement identity \cite[Section 0.8.5]{HJ2013} is
\[
 \det(\mathcal D(\mathbf Z)+D)
 =\det\mathsf C\,
   \det\bigl(\mathcal D_1(\mathbf Z_1)+\Delta\bigr),
 \qquad \Delta=D_{11}-D_{12}\mathsf C^{-1}D_{21}.
\]
Thus \(\Delta\) is a fixed real shift under this conditioning.
Writing the Gaussian covariance in the same partition,
\[
 \Gamma=\begin{pmatrix}
  \Gamma_{11}&\Gamma_{1,\mathrm{rem}}\\
  \Gamma_{\mathrm{rem},1}&\Gamma_{\mathrm{rem}}
 \end{pmatrix},\qquad
 S=\Gamma_{11}-\Gamma_{1,\mathrm{rem}}
       \Gamma_{\mathrm{rem}}^{-1}\Gamma_{\mathrm{rem},1}>0,
\]
the conditional law of \(\mathbf Z_1\) is Gaussian with covariance
\(S\). Its mean depends on \(\mathbf Z_{\mathrm{rem}}\), but the
one-block estimate is uniform in that mean and in \(\Delta\). Therefore
\[
 \mathbb E\bigl[|\det(\mathcal D(\mathbf Z)+D)|^{-1/2}
                  \mid\mathbf Z_{\mathrm{rem}}\bigr]
 \le C_*^q(\det S)^{-1/4}|\det\mathsf C|^{-1/2}.
\]
Now average over \(\mathbf Z_{\mathrm{rem}}\) and apply induction to
its marginal Gaussian law, whose covariance is \(\Gamma_{\mathrm{rem}}\),
and to the fixed shift \(D_{22}\). This gives
\[
 \begin{aligned}
 \mathbb E|\det(\mathcal D(\mathbf Z)+D)|^{-1/2}
 &\le C_*^q(\det S)^{-1/4}\,
             \mathbb E|\det\mathsf C|^{-1/2}\\
 &\le C_*^m(\det S)^{-1/4}
             (\det\Gamma_{\mathrm{rem}})^{-1/4}
  = C_*^m(\det\Gamma)^{-1/4},
 \end{aligned}
\]
where the last equality uses
\(\det\Gamma=\det S\det\Gamma_{\mathrm{rem}}\).
The one-block bounds supply the base of this induction, proving
\eqref{l:eq:D.4}.
Setting \(D=0\) then gives \eqref{l:eq:D.5} by the determinant identity
in the statement. \end{proof}

\subsubsection{The evaluation determinant and choice of
nodes}\label{l:the-evaluation-determinant-and-choice-of-nodes}

Characteristic-polynomial approaches to spectral-radius bounds also
appear in \cite{BCG2022,MR2025}. Here we use a finite evaluation
determinant to bound the conditional probability of the spectral
event, with randomness only in the independent border.

Recall that \([j]=\{1,\ldots,j\}\), with \([0]=\varnothing\),
and \([j,k]=\{\ell\in\mathbb Z:j\le\ell\le k\}\) is an integer
index set. For row and column sets \(I,J\), \(B_{I,J}\) keeps those
rows and columns in their original order. In particular,
\(B_{[j],[j]}\) is the leading \(j\times j\) principal block, and
\(B_{[j,d],[j,d]}\) is the principal block on indices \(j,\ldots,d\).

Write \(a_j=B_{j,j+1}\), \({d_{\mathrm{tail}}}=d-2m\ge1\), and
\(s_j=\prod_{i=j}^{2m-1}a_i\), with \(s_{2m}=1\). The assumption \(W_{2m}>0\) in Lemma~\ref{l:lem:gaussian-border}
makes all these links nonzero. Define the monic prefix and
suffix polynomials \[
 \theta_j(z)=\det(z\mathrm{Id}_j-B_{[j],[j]}),\qquad
 \varphi_j(z)=\det(z\mathrm{Id}_{d-j+1}-B_{[j,d],[j,d]}),
\] with \(\theta_0=\varphi_{d+1}=1\).

\begin{lemma}[Hessenberg cofactor factorization]\label{l:lem:hessenberg-cofactor}
For every \(z\in\mathbb C\) and \(1\le j\le\ell\le d\),
\begin{equation}
\operatorname{adj}(z\mathrm{Id}_d-B)_{j\ell}
 =\left(\prod_{i=j}^{\ell-1}a_i\right)
       \theta_{j-1}(z)\varphi_{\ell+1}(z),\qquad j\le\ell.
 \label{l:eq:D.6}
\end{equation}
\end{lemma}

\begin{proof}
Let \(\mathsf C_{j,\ell}(z)\) be the matrix obtained from
\(z\mathrm{Id}_d-B\) by deleting row \(\ell\) and column \(j\), retaining
the original order. Partition its rows and columns respectively as
\[
 \begin{aligned}
 \text{rows:}\quad &[j-1],\quad [j,\ell-1],\quad [\ell+1,d],\\
 \text{columns:}\quad &[j-1],\quad [j+1,\ell],\quad [\ell+1,d].
 \end{aligned}
\]
Because \((z\mathrm{Id}_d-B)_{ab}=0\) when \(b>a+1\), the blocks
above the diagonal vanish. Thus
\[
 \mathsf C_{j,\ell}(z)=
 \begin{pmatrix}
  \mathsf C_{\mathrm{pre}}(z)&0&0\\
  *&\mathsf C_{\mathrm{mid}}(z)&0\\
  *&*&\mathsf C_{\mathrm{suf}}(z)
 \end{pmatrix},\qquad
 \mathsf C_{\mathrm{mid}}(z)
   =(z\mathrm{Id}_d-B)_{[j,\ell-1],[j+1,\ell]}.
\]
The prefix and suffix blocks are the principal blocks on \([j-1]\)
and \([\ell+1,d]\), with determinants \(\theta_{j-1}(z)\) and
\(\varphi_{\ell+1}(z)\). The middle block is lower triangular:
its entry in relative row \(p\), column \(q\) comes from original
row \(j+p-1\), column \(j+q\), so it vanishes for \(q>p\).
Its diagonal is \(-a_j,\ldots,-a_{\ell-1}\), hence
\[
 \det\mathsf C_{\mathrm{mid}}(z)
   =(-1)^{\ell-j}\prod_{i=j}^{\ell-1}a_i.
\]
The adjugate uses the transposed cofactor, so block triangularity gives
\[
 \begin{aligned}
 \operatorname{adj}(z\mathrm{Id}_d-B)_{j\ell}
  &=(-1)^{j+\ell}\det\mathsf C_{j,\ell}(z)\\
  &=(-1)^{j+\ell+\ell-j}
      \theta_{j-1}(z)\left(\prod_{i=j}^{\ell-1}a_i\right)
      \varphi_{\ell+1}(z)\\
  &=\left(\prod_{i=j}^{\ell-1}a_i\right)
      \theta_{j-1}(z)\varphi_{\ell+1}(z).
 \end{aligned}
\]
Here the sign is \((-1)^{2\ell}=1\). Empty blocks have determinant
one, and empty products equal one, so the calculation also covers
\(j=\ell\), \(j=1\) and \(\ell=d\).
\end{proof}

For the two conditional integrations, we select only \(2m\) of the
\(2d+1\) border coordinates:
\[
 \mathbf x_*=(\mathbf x_{\mathrm b})_{[m]},\qquad
 \mathbf u_*=(\mathbf u_{\mathrm b})_{[m+1,2m]}.
\]
The full \(d\times d\) core \(B\) is retained, and all remaining
border variables are held fixed during these integrations.
The following evaluation matrices correspond to these two coordinate sets.

\begin{definition}[Prefix and suffix evaluation matrices]\label{l:thm:EVALMAT}
Let \(w>0\), and let \(z_1,\ldots,z_m\) be all the distinct roots of
\(z^m=w\), ordered with the real roots first and each nonreal root
adjacent to its complex conjugate. The two \(m\times m\) evaluation
matrices are \begin{equation}
(E_{\mathrm{pre}})_{hj}=s_j\theta_{j-1}(z_h)\quad(1\le j\le m),
 \qquad (E_{\mathrm{suf}})_{h\ell}=s_\ell^{-1}\varphi_{\ell+1}(z_h)
                    \quad(m+1\le\ell\le2m).
 \label{l:eq:D.7}
\end{equation} The columns of \(E_{\mathrm{suf}}\) retain their increasing original labels.
\end{definition}

Equation \eqref{l:eq:D.6} separates the selected cofactors as
\((E_{\mathrm{pre}})_{hj}(E_{\mathrm{suf}})_{h\ell}\).
The next lemma identifies the remaining polynomial factor and chooses
nodes at which both evaluation matrices are invertible.

\begin{lemma}[Evaluation determinant and choice of nodes]\label{l:lem:evaluation-determinant}
Fix the real lower-Hessenberg core \(B\) and an integer \(m\) with
\(2m<d\) and \(W_{2m}>0\), and write \(d_{\mathrm{tail}}=d-2m\).

\textup{(i)} There is a unique real monic polynomial \(Q_B\) of degree
\(d_{\mathrm{tail}}\), depending only on \(B\) and \(m\), such that,
for every \(w>0\) and its \(m\)-th roots \(z_1,\ldots,z_m\),
\[
 \frac{\det[\varphi_{2m-j+1}(z_h)]_{\substack{1\le h\le m\\0\le j<m}}}
      {\det[z_h^j]_{\substack{1\le h\le m\\0\le j<m}}}
   =(-1)^{d_{\mathrm{tail}}(m-1)}Q_B(w).
\]
For the matrices in Definition~\ref{l:thm:EVALMAT}, this gives
\begin{equation}
 |\det E_{\mathrm{pre}}\det E_{\mathrm{suf}}|
   =m^m w^{m-1}W_{2m}|Q_B(w)|.
 \label{l:eq:D.8}
\end{equation}

\textup{(ii)} For every \(0<R\le1\), there exists
\(w\in[R^m/2,R^m]\) such that
\begin{equation}
 |Q_B(w)|\ge\left(\frac{R^m}{4\mathrm e}\right)^{d_{\mathrm{tail}}}.
 \label{l:eq:D.9}
\end{equation}
The point \(w\) can be chosen to depend only on \(B,m,R\).
\end{lemma}

\begin{proof}
For \textup{(i)}, first compute the prefix determinant. Write
\[
 \mathcal V=\det[z_h^{j-1}]_{h,j=1}^m.
\]
Since \(\theta_{j-1}\) is monic of degree \(j-1\), write
\[
 \theta_{j-1}(z)=z^{j-1}+\sum_{q=0}^{j-2}c_{q,j}z^q,
 \qquad 1\le j\le m,
\]
where the sum is empty for \(j=1\). In the matrix
\([\theta_{j-1}(z_h)]_{h,j=1}^m\), the first column is \(1\).
Proceeding through columns \(j=2,\ldots,m\), subtract the linear
combination \(\sum_{q=0}^{j-2}c_{q,j}(z_h^q)_{h=1}^m\) of the
already reduced columns \(1,\ldots,j-1\) from column \(j\).
This leaves column \((z_h^{j-1})_{h=1}^m\) and does not change the
determinant. Factoring out the column weights in
Definition~\ref{l:thm:EVALMAT} therefore gives
\[
 \det[\theta_{j-1}(z_h)]_{h,j=1}^m=\mathcal V,
 \qquad
 \det E_{\mathrm{pre}}
   =\left(\prod_{j=1}^m s_j\right)\mathcal V.
\]
Thus the prefix contributes precisely one Vandermonde factor.

Next compute the suffix determinant. Reverse the suffix family, writing
\(p_j=\varphi_{2m-j+1}\), monic of degree
\(d_{\mathrm{tail}}+j\) for \(0\le j<m\). The quotient
\[
 S(z_1,\ldots,z_m)
   =\frac{\det[p_j(z_h)]_{\substack{1\le h\le m\\0\le j<m}}}
          {\det[z_h^j]_{\substack{1\le h\le m\\0\le j<m}}}
\]
is a real polynomial, as follows. If \(z_h=z_k\) for \(h<k\),
the numerator has two equal rows and vanishes, so \(z_k-z_h\) divides
it. These are distinct irreducible linear factors in
\(\mathbb R[z_1,\ldots,z_m]\), hence their product divides the numerator.
That product is precisely the denominator:
\[
 \det[z_h^j]_{\substack{1\le h\le m\\0\le j<m}}
    =\prod_{1\le h<k\le m}(z_k-z_h).
\]
Thus the quotient \(S\) is a polynomial. Exchanging any two variables
changes the sign of both determinants, so \(S\) is symmetric.
Its total degree is at most \(md_{\mathrm{tail}}\), and its highest
homogeneous part is \(\prod_{h=1}^m z_h^{d_{\mathrm{tail}}}\).
To show that evaluating \(S\) at the roots of \(z^m-w\) produces
an ordinary polynomial in \(w\), we use the fundamental theorem of
symmetric polynomials. It gives a real polynomial \(P\) such that
\(S=P(e_1,\ldots,e_m)\), where
\[
 e_j(z_1,\ldots,z_m)
   =\sum_{1\le h_1<\cdots<h_j\le m}z_{h_1}\cdots z_{h_j},
 \qquad 1\le j\le m.
\]
Comparing coefficients in \(\prod_{h=1}^m(t-z_h)=t^m-w\) gives
\[
 e_j(z_1,\ldots,z_m)=0\quad(1\le j<m),\qquad
 e_m(z_1,\ldots,z_m)=\prod_{h=1}^m z_h=(-1)^{m-1}w.
\]
Consequently,
\[
 S(z_1,\ldots,z_m)=P\bigl(0,\ldots,0,(-1)^{m-1}w\bigr),
\]
which is a real polynomial in \(w\). Only powers of \(e_m\) survive;
since \(e_m\) has total degree \(m\), the degree in \(w\) is at most
\(d_{\mathrm{tail}}\). The highest homogeneous part of \(S\) gives
the leading coefficient \((-1)^{d_{\mathrm{tail}}(m-1)}\).
This defines the monic \(Q_B\); its coefficients depend only on \(B,m\),
and the displayed identity for every \(w>0\) proves uniqueness.

Factoring out the suffix column weights and using that column
reversal changes only the sign, we obtain
\[
 |\det E_{\mathrm{suf}}|
   =\left(\prod_{\ell=m+1}^{2m}|s_\ell|^{-1}\right)
       |\mathcal V|\,|Q_B(w)|.
\]
Combining this with the prefix calculation gives
\[
 |\det E_{\mathrm{pre}}\det E_{\mathrm{suf}}|
   =\frac{\prod_{j=1}^m|s_j|}{\prod_{\ell=m+1}^{2m}|s_\ell|}
       |\mathcal V|^2|Q_B(w)|.
\]
Finally,
\[
 |\mathcal V|^2=m^mw^{m-1},\qquad
 \frac{\prod_{j=1}^m|s_j|}{\prod_{\ell=m+1}^{2m}|s_\ell|}
   =\prod_{i=1}^{2m-1}|a_i|^{\min(i,2m-i)}=W_{2m},
\]
which proves \eqref{l:eq:D.8}.

For \textup{(ii)}, define the grid
\[
 \omega_j=\frac{R^m}{2}+\frac{jR^m}{2d_{\mathrm{tail}}},
 \qquad j\in\{0,\ldots,d_{\mathrm{tail}}\},
\]
and choose
\[
 \begin{aligned}
 j_*&=\min\left\{j\in\{0,\ldots,d_{\mathrm{tail}}\}:
       |Q_B(\omega_j)|=\max_{0\le\ell\le d_{\mathrm{tail}}}
                         |Q_B(\omega_\ell)|\right\},\\
 w&=\omega_{j_*}\in[R^m/2,R^m].
 \end{aligned}
\]
This choice depends only on \(B,m,R\). The grid spacing is
\(\Delta=R^m/(2d_{\mathrm{tail}})\). Taking the
\(d_{\mathrm{tail}}\)-th forward difference of the monic polynomial
\(Q_B\), with step \(\Delta\), gives
\[
 d_{\mathrm{tail}}!\Delta^{d_{\mathrm{tail}}}
 =\left|\sum_{j=0}^{d_{\mathrm{tail}}}
       (-1)^{d_{\mathrm{tail}}-j}\binom{d_{\mathrm{tail}}}{j}
       Q_B(\omega_j)\right|
 \le 2^{d_{\mathrm{tail}}}|Q_B(w)|.
\]
Using \(d_{\mathrm{tail}}!\ge(d_{\mathrm{tail}}/\mathrm e)^{d_{\mathrm{tail}}}\)
proves \eqref{l:eq:D.9}. Since \(w>0\), \(W_{2m}>0\) and
\(Q_B(w)\ne0\), \eqref{l:eq:D.8} makes both matrices invertible.
\end{proof}

From now on, use the \(w\) and nodes supplied by
Lemma~\ref{l:lem:evaluation-determinant}\textup{(ii)}. They remain fixed
in every border integration.

\subsubsection{The two border integrations}\label{l:the-two-border-integrations}

Let \(P_T(z)=\det(z\mathrm{Id}_n-T)\). The block determinant formula gives
\[
 P_T(z)=(z-g_{\mathrm b})\det(z\mathrm{Id}_d-B)
       -\mathbf x_{\mathrm b}^T\operatorname{adj}(z\mathrm{Id}_d-B)\mathbf u_{\mathrm b}.
\]
For invertible \(z\mathrm{Id}_d-B\), this is the Schur complement
identity; both sides are polynomials, so it holds for every \(z\).
At the nodes \(z_h^m=w\) chosen above, define
\(Y=\prod_{h=1}^m|P_T(z_h)|\). If \(\lambda_1,\ldots,\lambda_n\)
are the eigenvalues of \(T\), then
\[
 Y=\prod_{j=1}^n|w-\lambda_j^m|\le(w+r^m)^n
 \quad\text{on }\{\rho(T)\le r\}.
\]
Consequently,
\(\mathbb P_\kappa\{\rho(T)\le r\}\le(w+r^m)^{n/2}\mathbb E_\kappa Y^{-1/2}\).
The following estimate performs both border integrations and leaves only
the deterministic evaluation matrices.

\begin{lemma}[Two-stage border negative moment]\label{l:lem:border-negative-moment}
For the fixed core, nodes and evaluation matrices above,
\begin{equation}\label{l:eq:border-negative-moment}
 \mathbb E_\kappa\prod_{h=1}^m|P_T(z_h)|^{-1/2}
 \le C^m(n/\kappa)^{m/2}
       |\det E_{\mathrm{pre}}\det E_{\mathrm{suf}}|^{-1/2}.
\end{equation}
\end{lemma}

\begin{proof}
Recall the scalar quantity
\[
 Y=\prod_{h=1}^m|P_T(z_h)|,\qquad
 Y^{-1/2}=\prod_{h=1}^m|P_T(z_h)|^{-1/2},
\]
whose negative half-power is the integrand in
\eqref{l:eq:border-negative-moment}.
Throughout these two integrations, replace each conjugate pair of
values by its real and imaginary parts, as in Lemma~\ref{l:thm:DN}.
For a matrix \(K\) with real rows and conjugate-paired rows, write
\(K_{\mathbb R}\) for the real matrix obtained by making the same
replacement in its rows. Each pair replacement has determinant of
absolute value \(1/2\), so, with \(b_{\mathrm{pair}}\le m/2\) pairs,
\[
 |\det K_{\mathbb R}|=2^{-b_{\mathrm{pair}}}|\det K|.
\]
In particular, if \(\mathbf v\sim N(0,(\kappa/n)\mathrm{Id}_m)\),
then \(K_{\mathbb R}\mathbf v\) has covariance determinant
\[
 2^{-2b_{\mathrm{pair}}}(\kappa/n)^m|\det K|^2.
\]
Below \(C\) is universal and may increase between lines.

\emph{Integrate the row.} Fix \(g_{\mathrm b}\), the entire border
column \(\mathbf u_{\mathrm b}\), and the row coordinates
\((\mathbf x_{\mathrm b})_j\) for \(m<j\le d\). Integrate only over
\(\mathbf x_*=(\mathbf x_{\mathrm b})_{[m]}\), the first \(m\)
coordinates of the border row. The evaluation vector is
\((P_T(z_h))_{h=1}^m=\mathbf b_{\mathrm{shift}}-H\mathbf x_*\), where
\[
 H_{hj}=\sum_{\ell=1}^d
 \operatorname{adj}(z_h\mathrm{Id}_d-B)_{j\ell}(\mathbf u_{\mathrm b})_\ell,
 \qquad 1\le h,j\le m,
\]
and the shift vector is explicitly
\[
 \begin{aligned}
 (\mathbf b_{\mathrm{shift}})_h
  &=(z_h-g_{\mathrm b})\det(z_h\mathrm{Id}_d-B)\\
  &\quad-\sum_{j=m+1}^d(\mathbf x_{\mathrm b})_j
       \sum_{\ell=1}^d
       \operatorname{adj}(z_h\mathrm{Id}_d-B)_{j\ell}
       (\mathbf u_{\mathrm b})_\ell,
       \qquad 1\le h\le m.
 \end{aligned}
\]
It collects all terms independent of \(\mathbf x_*\) and is fixed
during the row integration. Since \(\mathbf x_*\) is centered, it is
also the conditional mean of the evaluation vector.
The real coordinate vector of \((P_T(z_h))_{h=1}^m\) has covariance
determinant \(2^{-2b_{\mathrm{pair}}}(\kappa/n)^m|\det H|^2\).
Applying the product bound \eqref{l:eq:D.5} to the conjugate-paired
values \((P_T(z_h))_{h=1}^m\), whose absolute product is \(Y\), and
absorbing the factor
\(2^{b_{\mathrm{pair}}/2}\) into \(C^m\), gives, when \(H\) is invertible,
\begin{equation}\label{l:eq:D.10}
 \mathbb E_{\mathbf x_*}Y^{-1/2}
 \le C^m(n/\kappa)^{m/4}|\det H|^{-1/2}.
\end{equation}
The product bound is uniform in \(\mathbf b_{\mathrm{shift}}\), so its
dependence on \(\mathbf u_*\) does not enter the right-hand side of
\eqref{l:eq:D.10}. We next integrate the remaining factor
\(|\det H|^{-1/2}\) over those column coordinates.

\emph{Integrate the column.} Fix the coordinates
\((\mathbf u_{\mathrm b})_\ell\) of the border column vector
\(\mathbf u_{\mathrm b}\) for
\(\ell\in[d]\setminus[m+1,2m]\). Integrate only over
\(\mathbf u_*=(\mathbf u_{\mathrm b})_{[m+1,2m]}\), the \(m\) border
column coordinates with indices \(m+1,\ldots,2m\).
The cofactor factorization \eqref{l:eq:D.6} gives
\begin{equation}\label{l:eq:D.11}
 H=\operatorname{diag}(E_{\mathrm{suf}}\mathbf u_*)E_{\mathrm{pre}}+H_0,
\end{equation}
with \(H_0\) fixed. Let \(\mathbf Z_u\) consist of the real values
and the real and imaginary parts of the conjugate-paired values of
\(E_{\mathrm{suf}}\mathbf u_*\). Complex multiplication in these
coordinates gives
\[
 H_{\mathbb R}(E_{\mathrm{pre}})_{\mathbb R}^{-1}
 =\mathcal D(\mathbf Z_u)
    +(H_0)_{\mathbb R}(E_{\mathrm{pre}})_{\mathbb R}^{-1}.
\]
The determinant on the left equals \(\det H/\det E_{\mathrm{pre}}\),
and the second term on the right is a fixed real matrix. This is
exactly the shifted determinant in Lemma~\ref{l:thm:DN}\textnormal{(ii)}.
The real coordinates of \(E_{\mathrm{suf}}\mathbf u_*\) in
Lemma~\ref{l:thm:DN} have covariance determinant
\(2^{-2b_{\mathrm{pair}}}(\kappa/n)^m|\det E_{\mathrm{suf}}|^2\).
Applying \eqref{l:eq:D.4} and again absorbing the factor
\(2^{b_{\mathrm{pair}}/2}\) into \(C^m\) yields
\begin{equation}\label{l:eq:D.12}
 \mathbb E_{\mathbf u_*}|\det H|^{-1/2}
 \le C^m(n/\kappa)^{m/4}
       |\det E_{\mathrm{pre}}\det E_{\mathrm{suf}}|^{-1/2}.
\end{equation}
Its finiteness proves that \(H\) is invertible almost surely, as required
in the row integration. Combining \eqref{l:eq:D.10} and \eqref{l:eq:D.12}
proves \eqref{l:eq:border-negative-moment}. Both bounds are uniform in
all fixed outside coordinates, so Tonelli also permits averaging over
the rest of the border.
\end{proof}

\subsubsection{Completion of the Gaussian estimate}\label{l:the-quadratic-exponent}

\begin{proof}[\proofname\ of Lemma~\ref{l:lem:gaussian-border}]
Apply Lemma~\ref{l:lem:border-negative-moment} to the spectral-event bound
above, then insert the evaluation determinant \eqref{l:eq:D.8}:
\begin{equation}\label{l:eq:D.13}
 \mathbb P_\kappa\{\rho(T)\le r\}
 \le C^m\left[
 \frac{(n/\kappa)^m(w+r^m)^n}
      {m^mw^{m-1}W_{2m}|Q_B(w)|}\right]^{1/2}.
\end{equation}
Here \(r^m\le R^m/2\le w\le R^m\). Thus \(w+r^m\le2w\), and
\eqref{l:eq:D.9}, with \(d_{\mathrm{tail}}=n-1-2m\), gives
\begin{equation}\label{l:eq:D.14}
 \mathbb P_\kappa\{\rho(T)\le r\}
 \le C^m\left[
 2^n(4\mathrm e)^{d_{\mathrm{tail}}}
 \left(\frac{n}{\kappa m}\right)^m
 \frac{R^{m^2+2m}}{W_{2m}}\right]^{1/2}.
\end{equation}
The radial power is
\(m(n-m+1-d_{\mathrm{tail}})=m^2+2m\). Since \(R\le1\), we may
replace \(R^{m^2+2m}\) by \(R^{m^2}\). All remaining factors are
bounded by \(e^{C_{\mathrm G}(\kappa)n}\), using
\(m,d_{\mathrm{tail}}\le n\) and
\(m\log(n/(\kappa m))\le n/\kappa\).
This proves \eqref{l:eq:gaussian-border}.
\end{proof}

\Needspace{14\baselineskip}
\section{Observation comparison and the link estimate}\label{sec:comparison-links}
\label{l:proof-strategy-and-contributions}
We prove the original-law link estimate in Proposition~\ref{l:thm:ML}.
Arnoldi queries are adaptive under both entry laws. For Gaussian entries,
orthogonality gives fresh Gaussian responses and hence independent chi
laws for the links. To transfer
probabilities of order \(e^{-cn^2}\), we compare the entire observation
record by a multiplicative bound that preserves this scale.
Theorem~\ref{l:thm:AC} makes the cost explicit in terms of the query
frame's coordinate weights: it is small when these weights are small.
Subsection~\ref{l:sec:real_comparison} states this comparison,
reconstructs the observations after separating a fixed column set, and
provides the reference-link estimate.

The remaining difficulty is that flatness of the frame is itself
random. After completing the queries, we delete concentrated coordinates;
many deletions force a large Gram-determinant loss. For each fixed
candidate deleted set, an independent-column reference controls that
loss, while a coordinate-weight moment controls the comparison cost.
Subsection~\ref{l:sec:real_stopping} uses these estimates to control
failed or late flattening and completes the link proof. The analytic
comparison proof is in Subsection~\ref{l:sec:real_comparison_analysis}.
The coarse comparison controls failure to reach flatness within the
deletion budget; the weighted comparison controls outputs that become
flat only after many deletions.

\subsection{Observation comparison, reconstruction, and reference
links}\label{l:sec:real_comparison}

For the link estimates, fix a deterministic candidate set
\(I_0\subseteq[d]\), write \(J_0=[d]\setminus I_0\), and assume
\(k\le h:=|J_0|\). The extra starter \(\mathbf s\sim\mu^{\otimes d}\)
is independent of \(A\). We supply the starter and the separated
columns at the outset, and make orthogonal observations of only the
remaining matrix. Thus the source and auxiliary input in our application
are
\[
 X=A_{:,J_0}\in\mathbb R^{d\times h},\qquad
 \mathsf C=(\mathbf s,A_{:,I_0}).
\]
Lemma~\ref{l:thm:RECONPROP} constructs these observations so that they,
together with \(\mathsf C\), recover the original \(k\)-step Arnoldi run.
The original link and deletion events are therefore functions of this
record.

Under \(\mathsf Q_{I_0}\), the auxiliary input keeps its original law,
while \(X\) has independent standard Gaussian entries, independent of
\(\mathsf C\) (Definition~\ref{l:thm:FIXREF}). The candidate \(I_0\)
is fixed throughout this comparison; we do not condition on the random
deletion output. We first state the comparison in a general observation
model, then verify the reconstruction and estimate the reference links.

\subsubsection{The observation comparison
theorem}\label{l:the-observation-comparison-theorem}

Think of an algorithm that accesses a matrix through \(k\) matrix-vector
products. It chooses a unit direction, observes the matrix's response,
and uses the information obtained so far to choose the next direction,
orthogonal to all previous ones. The comparison below applies to the
resulting record of observations. Adaptive transcript likelihoods and
truncation on low-information events are used, for example, in Gaussian
query lower bounds \cite[Lemma 3.4 and Section 4]{SER2020}. Universality
for specific message-passing iterations is studied in \cite{BLM2015}.
The result here compares general predictable orthonormal observations
under iid and Gaussian entry laws, with an explicit cost determined by
the coordinate weights of the realized frame.

Here \(d,h\) are arbitrary source dimensions; in the core application,
\(d=n-1\) and \(h=d-|I_0|\). The source matrix below is unnormalized,
so its Gaussian reference has entry variance one.

Let \(X\in\mathbb R^{d\times h}\) have iid entries of law \(\mu\), and
let \(\mathsf C\) be an independent auxiliary input in a standard Borel space
\(\mathcal C\). The input may include a starting vector or an
independent random seed. Fix a Borel update rule before the experiment:
\[
 \mathbf u_a=\mathbf u_a(\mathsf C,\mathbf y_1,\ldots,\mathbf y_{a-1})\in\mathbb R^h,\qquad
 \mathbf y_a=X\mathbf u_a,\qquad U=[\mathbf u_1,\ldots,\mathbf u_k],\qquad Y=[\mathbf y_1,\ldots,\mathbf y_k].
\] At step \(a\), the rule may use only \(\mathsf C\) and the preceding
responses; this is the predictability requirement. It must produce a
unit vector orthogonal to \(\mathbf u_1,\ldots,\mathbf u_{a-1}\) on every history. Thus
\(U\in\mathbb R^{h\times k}\), \(U^TU=\mathrm{Id}_k\), and
\(\operatorname{Ran}U\) is \(k\)-dimensional, with \(k\le h\). The rule
is fixed, while its output frame \(U\) is random because the responses
are random. For an input \(c\in\mathcal C\), consider a supplied
response matrix
\(\mathbf z=[\mathbf z_1,\ldots,\mathbf z_k]\in\mathbb R^{d\times k}\),
with columns \(\mathbf z_a\in\mathbb R^d\). Here \(\mathbf z\) denotes a
value of the observations \(Y\). Define \[
 U(c,\mathbf z)=\big[\mathbf u_a(c,\mathbf z_1,\ldots,\mathbf z_{a-1})\big]_{a=1}^k
          \in\mathbb R^{h\times k}.
\] This runs the fixed rule on the supplied history: its \(a\)-th query
uses only the first \(a-1\) response columns. Set \begin{equation}
\ell_j(U)=\sum_{a=1}^kU_{ja}^2,\qquad
 \Theta(U)=\sum_{j=1}^h\ell_j(U)^{3/2},\qquad
 0\le\ell_j\le1,\quad\sum_{j=1}^h\ell_j=k.
 \label{l:eq:A.1}
\end{equation} We call \(\ell_j(U)\) the coordinate weight of the query subspace:
\(\ell_j(U)=\|P_{\operatorname{Ran}U}e_j\|^2\), equivalently the largest
possible value of \(|(\mathbf v)_j|^2\) over unit vectors
\(\mathbf v\in\operatorname{Ran}U\). Small coordinate weights therefore mean
that every vector in the query subspace is spread across the input
coordinates.

\begin{definition}[Original and reference observation laws]\label{l:thm:OBSLAW}

On the observation space \(\mathcal C\times\mathbb R^{d\times k}\),
define \[
 P=\mathcal L(\mathsf C,Y),\qquad
 G=\mathcal L(\mathsf C)\otimes G_0,\qquad G_0=N(0,1)^{\otimes dk}.
\] Equivalently, \(G\) is obtained by replacing \(X\) with a standard
Gaussian matrix and running the same rule. Predictability and
orthogonality make its responses fresh Gaussian columns, as proved in
Lemma \ref{l:thm:LIKE} in Section \ref{l:sec:real_comparison_analysis}. The auxiliary
input keeps its original law in both experiments. \end{definition}

An observation event is a Borel set
\(F\subseteq\mathcal C\times\mathbb R^{d\times k}\): it occurs when
\((\mathsf C,Y)\in F\), so it may depend on both \(\mathsf C\) and \(Y\). For a fixed
input \(c\), write \[
 F_c=\{\mathbf z:(c,\mathbf z)\in F\},\qquad
 P_c=\text{law of the responses when the input is fixed at }c.
\] Thus \(F_c\) is determined by the response record alone; in
particular, its query frame is the known function \(U(c,\mathbf z)\).
The comparison concerns events observable from this record, such as the
link and deletion events in Lemma \ref{l:thm:RECONPROP}.

Both experiments run the same update rule; their realized query frames
may differ because their responses differ. Any predictable Borel rule
satisfying orthonormality is allowed, including directions fixed in
advance. No Arnoldi or Krylov recurrence is assumed.

Ordinary Arnoldi is one example, with \(h=d\), \(\mathsf C=\mathbf s\),
and the orthogonal completion rule of Algorithm~\ref{l:thm:ARN} at a
zero residual. For the separated-column experiment above,
Lemma~\ref{l:thm:RECONPROP} will construct a frame \(\widehat V\) with
\[
 U=\widehat V,\qquad
 Y=X\widehat V=[\mathbf y_1^{\mathrm{rem}},\ldots,\mathbf y_k^{\mathrm{rem}}].
\]
Under this identification, the observation laws are
\[
 P=\mathcal L_{\mathbb P_\mu}(\mathsf C,Y),\qquad
 G=\mathcal L_{\mathsf Q_{I_0}}(\mathsf C,Y).
\]
Thus \(G\) is the joint law of the auxiliary input and the
remaining-matrix responses under \(\mathsf Q_{I_0}\), while
\(\mathsf Q_{I_0}\) itself is a law of the matrix and starter.

\Needspace{12\baselineskip}
\begin{theorem}[Coarse and weighted observation comparison]\label{l:thm:AC}

There are law-dependent positive constants
\(C_{\mathrm{coarse}},C_{\mathrm{weight}},\delta_\mu\), with
\(\delta_\mu\le1/4\), and
\(\theta_{\mathrm{coarse}},\theta_\mu\in(0,1)\), such that, when
\(h\le d\), every Borel observation event \(F\) satisfies \begin{equation}
P(F)\le e^{C_{\mathrm{coarse}}dk}G(F)^{\theta_{\mathrm{coarse}}}.
 \label{l:eq:A.2}
\end{equation} For \(F\subseteq\{\max_j\ell_j(U)\le\delta_\mu\}\), one has \begin{equation}
P(F)\le\left[\mathbb E_G\mathbf1_F
                    e^{C_{\mathrm{weight}}d\Theta(U)}\right]^{\theta_\mu}.
 \label{l:eq:A.3}
\end{equation} The constants are independent of dimensions, the query rule and the
law of \(\mathsf C\). \end{theorem}

The coarse assertion applies without a coordinate-weight cap. On
\(\max_j\ell_j\le\sigma\), the identity \(\sum_j\ell_j=k\) gives
\(\Theta(U)\le k\sqrt\sigma\). Thus the weight has logarithmic cost at
most \(C_{\mathrm{weight}}\sqrt\sigma\,dk\), which is small relative to
\(dk\) when the cap is small. The stopping argument will
also use the finer random cost \(\Theta(U)\) itself, which is why
\eqref{l:eq:A.3} retains it inside the Gaussian expectation. The following
reconstruction expresses our matrix products in this setting.

\subsubsection{Reconstruction and comparison for fixed column
sets}\label{l:sec:real_links}

We now verify that the separated-column observations recover the
Arnoldi run on the full core \(A\). The later deletion rule
(Algorithm~\ref{l:thm:DEL}) is applied to its completed frame; separating
columns for comparison leaves the original spectral event unchanged.

\begin{definition}[The reference for a fixed candidate]\label{l:thm:FIXREF}

Fix \(I_0\subseteq[d]\), set \(J_0=[d]\setminus I_0\), and assume
\(k\le |J_0|\). The variables \(\mathbf s\) and \(A_{:,I_0}\) are
supplied to the query rule at the start. Under the original law they
have their original joint law and are independent of the remaining
matrix \(A_{:,J_0}\), which has iid entries of law \(\mu\). Under the
reference law \(\mathsf Q_{I_0}\), only the entries of \(A_{:,J_0}\) are
replaced by independent standard Gaussians, independent of
\((\mathbf s,A_{:,I_0})\). The latter pair retains its original joint
law and is averaged over, not fixed to a typical value. \end{definition}

\Needspace{6\baselineskip}
\textbf{What is reconstructed.} Given the retained auxiliary input
\((\mathbf s,A_{:,I_0})\), we rerun the same Arnoldi recursion by
querying only \(A_{:,J_0}\). At each step the full response is recovered
from \[
 A\mathbf v_j=A_{:,I_0}(\mathbf v_j)_{I_0}+A_{:,J_0}(\mathbf v_j)_{J_0}.
\] The first term is directly available from the supplied columns; the
second comes from queries to the remaining matrix. Recovering this sum
recovers the original next direction, so the entire response record,
link product and deletion output are unchanged. Because the restricted directions \((\mathbf v_j)_{J_0}\) need not be
orthogonal, the lemma first orthogonalizes them and expresses their
responses through the resulting queries. The reconstruction identity
is the sum above written in this orthogonal basis.

\Needspace{8\baselineskip}

Here \(V=[\mathbf v_1,\ldots,\mathbf v_k]\) is the completed Arnoldi
frame for \(A\). For a fixed retained set \(J_0\), whenever
\(V_{J_0,:}\) has rank \(k\),
write \[
 \ell_i(J_0):=(P_{\operatorname{Ran}V_{J_0,:}})_{ii},\qquad i\in J_0.
\] For a cap \(\sigma>0\), call the retained frame flat if it has rank
\(k\) and \(\max_{i\in J_0}\ell_i(J_0)\le\sigma\). This is the criterion
used by Algorithm \ref{l:thm:DEL} with cap \(\sigma\).

\begin{lemma}[Exact reconstruction and the comparison inputs]\label{l:thm:RECONPROP}

Fix \(I_0\subseteq[d]\), write \(J_0=[d]\setminus I_0\), and assume
\(k\le |J_0|\). Starting from \(\mathbf s\) and the columns
\(A_{:,I_0}\), one can successively construct orthonormal vectors
\(\widehat{\mathbf v}_1,\ldots,\widehat{\mathbf v}_k\in\mathbb R^{J_0}\) and form \[
 \mathbf y_j^{\mathrm{rem}}:=A_{:,J_0}\widehat{\mathbf v}_j\in\mathbb R^d,
 \qquad j=1,\ldots,k.
\] Here \(\mathrm{rem}\) refers to the remaining columns indexed by
\(J_0\). Each \(\widehat{\mathbf v}_j\) is a Borel function only of
\(\mathbf s,A_{:,I_0},\mathbf y_1^{\mathrm{rem}},\ldots,\mathbf y_{j-1}^{\mathrm{rem}}\).
The construction has the following properties.

\begin{enumerate}
\def\labelenumi{\arabic{enumi}.}
\tightlist
\item
  For every realization of \((A,\mathbf s)\), the data
  \(\mathbf s,A_{:,I_0},\mathbf y_1^{\mathrm{rem}},\ldots,\mathbf y_j^{\mathrm{rem}}\)
  determine exactly the original products \(A\mathbf v_1,\ldots,A\mathbf v_j\) and hence
  the next Arnoldi vector \(\mathbf v_{j+1}\). Thus the resulting objects are
  the orthonormal vectors \(\widehat{\mathbf v}_j\), their products with the
  remaining matrix, and the reconstructed original Arnoldi vectors and
  residual lengths. Consequently the original link and deletion events
  are measurable with respect to
  \(\sigma(\mathbf s,A_{:,I_0},\mathbf y_1^{\mathrm{rem}},\ldots,\mathbf y_k^{\mathrm{rem}})\).
\item
  If \(V_{J_0,:}\) has rank \(k\), then, with
  \(\widehat V=[\widehat{\mathbf v}_1,\ldots,\widehat{\mathbf v}_k]\), \[
   \operatorname{Ran}\widehat V=\operatorname{Ran}V_{J_0,:},\qquad
   \sum_{a=1}^k (\widehat{\mathbf v}_a)_i^2=\ell_i(J_0)\quad(i\in J_0).
  \]
\item
  Under \(\mathsf Q_{I_0}\), the vectors
  \(\mathbf y_1^{\mathrm{rem}},\ldots,\mathbf y_k^{\mathrm{rem}}\) are independent
  \(N(0,\mathrm{Id}_d)\) vectors, jointly independent of
  \((\mathbf s,A_{:,I_0})\).
\end{enumerate}

\end{lemma}

\begin{proof}

Starting with the original \(\mathbf v_1\), suppose the previous products have
been recovered, so \(\mathbf v_j\) is known. Orthogonalize its restriction: \[
 b_{ja}=\widehat{\mathbf v}_a^T(\mathbf v_j)_{J_0}\quad(a<j),\qquad
 \widehat{\mathbf r}_j=(\mathbf v_j)_{J_0}-\sum_{a=1}^{j-1}b_{ja}\widehat{\mathbf v}_a,\qquad \widehat R_j=\|\widehat{\mathbf r}_j\|.
\] Choose \(\widehat{\mathbf v}_j=\widehat{\mathbf r}_j/\widehat R_j\) when \(\widehat R_j>0\). If \(\widehat R_j=0\), apply
the same orthogonal completion rule as in Algorithm \ref{l:thm:ARN}, now in
\(\mathbb R^{J_0}\), with \[
 \widehat P_{j-1}=\sum_{a=1}^{j-1}\widehat{\mathbf v}_a\widehat{\mathbf v}_a^T:
 \qquad
 \widehat{\mathbf v}_j=\frac{(\mathrm{Id}_{J_0}-\widehat P_{j-1})e_i}
 {\|(\mathrm{Id}_{J_0}-\widehat P_{j-1})e_i\|},
\] where \(i\) is the smallest index in \(J_0\) for which the
denominator is nonzero. This supplements the reconstructed query frame
when a restricted direction is linearly dependent, while preserving the
original Arnoldi directions and responses. Such an index exists since
\(j\le k\le|J_0|\). Form the product \(\mathbf y_j^{\mathrm{rem}}\) defined
above. Linearity gives \begin{equation}
A\mathbf v_j=A_{:,I_0}(\mathbf v_j)_{I_0}
          +\sum_{a=1}^{j-1}b_{ja}\mathbf y_a^{\mathrm{rem}}+\widehat R_j\mathbf y_j^{\mathrm{rem}}.
 \label{l:eq:S.3}
\end{equation} The identity holds also when \(\widehat R_j=0\). It advances the original
Arnoldi recursion and proves (1) by induction. Each construction step is
Borel measurable and uses only the previously formed products. At full
restricted rank every \(\widehat R_j>0\), so \(\widehat V\) is obtained by
ordinary Gram--Schmidt and is a basis of
\(\operatorname{Ran}V_{J_0,:}\); its projection diagonals give (2).
Finally, the Gaussian assertion of Lemma \ref{l:thm:LIKE} gives (3). That lemma is
proved by orthogonal changes of variables independently of the present
link estimates. \end{proof}

On the event \(\{\text{flat output},\,I_{\mathrm{del}}=I_0\}\) of
Algorithm \ref{l:thm:DEL}, the retained frame meets the flatness criterion above.
Thus Lemma \ref{l:thm:RECONPROP} gives
\(\max_{i\in J_0}\sum_{a=1}^k(\widehat{\mathbf v}_a)_i^2\le\sigma\).

With the identification above, Theorem~\ref{l:thm:AC} applies. If
\(F\) is determined by these products and the supplied data, and
\(F\subseteq\{\max_{i\in J_0}\sum_{a=1}^k(\widehat{\mathbf v}_a)_i^2\le\sigma\}\)
with \(0<\sigma\le\delta_\mu\), then
\(\Theta(\widehat V)\le k\sqrt\sigma\) on \(F\), so \[
 \mathbb P_\mu(F)\le
 e^{\theta_\mu C_{\mathrm{weight}}dk\sqrt\sigma}
 \mathsf Q_{I_0}(F)^{\theta_\mu}.
\] Both probabilities average over \(\mathbf s\), the separated columns
and the remaining columns, with their laws as specified in Definition
\ref{l:thm:FIXREF}. This direct consequence will be used in Subsection
\ref{l:sec:real_link_completion}.

\subsubsection{Gaussian and mixed-column reference
links}\label{l:sec:real_reference}

Use the weights and normalization in \eqref{l:eq:M.6}, and write
\(\mathcal B_\varepsilon=\{W_k\le(1-\varepsilon/4)^{M_k}\}\).
We first estimate this event for an entirely Gaussian matrix, then
restore the fixed columns indexed by \(I_0\) to law \(\mu\), obtaining
\(\mathsf Q_{I_0}\). We denote the joint law of the all-Gaussian
matrix and its independent starter by \(\mathbb P_{\mathrm G}\),
with expectation \(\mathbb E_{\mathrm G}\). The starter has the same law in both experiments and is not replaced.
More generally, the estimate allows any independent starter law.

\begin{lemma}[Mixed reference links]\label{l:thm:LM}

For every fixed \(\varepsilon\in(0,1]\), there are
\(\eta_{\mu,\varepsilon}\in(0,1]\) and
\(c_{\mathrm{mix}}(\mu,\varepsilon)>0\) such that \begin{equation}
\mathsf Q_{I_0}(\mathcal B_\varepsilon)
 \le2e^{-c_{\mathrm{mix}}(\mu,\varepsilon)dk}
 \label{l:eq:L.3}
\end{equation} whenever \(k\ge2\), \(128k\le\varepsilon d\), and
\(|I_0|\le\eta_{\mu,\varepsilon}k\). The bound is uniform over the
choice of \(I_0\) and over every independent starter law. \end{lemma}

\begin{proof}

We first bound the all-Gaussian probability, uniformly in the starter.
Lemma \ref{l:thm:LIKE} makes the Arnoldi responses fresh standard Gaussian vectors.
The predictable residual projection at step \(j\) has rank \(d-j\), so
\(R_j^2\), \(1\le j<k\), are independent \(\chi^2_{d-j}\) variables,
where \(\chi^2_\nu\) is the chi-square law with \(\nu\) degrees of freedom.
This is the usual Gaussian Hessenberg structure; see \cite{KV2014} for
the Ginibre representation and its characteristic-polynomial recursion.
For \(Z\sim\chi^2_\nu\) and integers \(s\ge0\) with \(2s<\nu\), the
gamma integral gives \begin{equation}
\mathbb E Z^{-s}=\prod_{\ell=1}^s(\nu-2\ell)^{-1}.
 \label{l:eq:L.2}
\end{equation} Set \(p=\lfloor\varepsilon d/(64k)\rfloor\), \(D=k+1+2pk\). Since
\(128k\le\varepsilon d\), \[
 \frac{\varepsilon d}{128k}\le p\le\frac{\varepsilon d}{64k},
 \qquad D\le\varepsilon(d+1)/16.
\] Every denominator in \eqref{l:eq:L.2} with \(s=pw_j\) is at least \(d+1-D\).
Independence and \(-\log(1-u)\le2u\), \(0\le u\le1/2\), imply
\(\log\mathbb E_{\mathrm G}W_k^{-2p}\le\varepsilon pM_k/8\). Markov,
\(\log(1-\varepsilon/4)\le-\varepsilon/4\), and \(M_k\ge k^2/8\) now
yield \begin{equation}
\mathbb P_{\mathrm G}(\mathcal B_\varepsilon)
 \le e^{-3\varepsilon pM_k/8}
 \le e^{-c_{\mathrm G}(\varepsilon)dk},
 \qquad c_{\mathrm G}(\varepsilon)=\varepsilon^2/4096.
 \label{l:eq:L.1}
\end{equation}

It remains to restore the columns in \(I_0\). Choose
\(0<q\le\min\{1,2t_f\}\) and let \(\phi_1(x)=(2\pi)^{-1/2}e^{-x^2/2}\).
Boundedness of the density and the square-exponential moment give \[
 D_q:=\int_{\mathbb R}f(x)^{1+q}\phi_1(x)^{-q}\,dx
 \le(K_f\sqrt{2\pi})^q\mathbb E e^{q\xi^2/2}<\infty.
\] Also \(D_q\ge1\), by Jensen's inequality under the standard Gaussian
law. Write \(s=|I_0|\). The likelihood ratio of the mixed-column law
relative to the all-Gaussian law is \[
 L_{I_0}(A,\mathbf s)=\prod_{i=1}^d\prod_{j\in I_0}
                           \frac{f(A_{ij})}{\phi_1(A_{ij})},\qquad
 \mathbb E_{\mathrm G}L_{I_0}^{1+q}=D_q^{ds}.
\] The starter contributes no likelihood factor. H\"older and \eqref{l:eq:L.1} give
\begin{equation}
\begin{aligned}
 \mathsf Q_{I_0}(\mathcal B_\varepsilon)
 &\le D_q^{ds/(1+q)}
              \mathbb P_{\mathrm G}(\mathcal B_\varepsilon)^{q/(1+q)}\\
 &\le\exp\!\left\{\frac{ds\log D_q-qc_{\mathrm G}(\varepsilon)dk}{1+q}\right\}.
 \end{aligned}
 \label{l:eq:L.4}
\end{equation} Take \[
 \eta_{\mu,\varepsilon}
 =\min\left\{1,\frac{qc_{\mathrm G}(\varepsilon)}{2(1+\log D_q)}\right\},
 \qquad
 c_{\mathrm{mix}}=\frac{qc_{\mathrm G}(\varepsilon)}{2(1+q)}.
\] For \(s\le\eta_{\mu,\varepsilon}k\), \eqref{l:eq:L.4} is at most
\(e^{-c_{\mathrm{mix}}dk}\), which implies \eqref{l:eq:L.3}. \end{proof}

The moment order \(1+q\) may be arbitrarily close to one, so no
assumption that \(\int f^2/\phi_1<\infty\) is needed.
\subsection{Stopping and completion of the link
estimate}\label{l:sec:real_stopping}

To prove Proposition~\ref{l:thm:ML}, for each \(\varepsilon\in(0,1]\)
we choose an even query count \(k\) proportional to \(d\) so that, for
sufficiently large \(d\),
\[
 \mathbb P_\mu\{W_k(A,\mathbf s)\le(1-\varepsilon/4)^{M_k}\}
 \le e^{-c_{\mathrm{link}}d^2}.
\]
The weighted comparison of Theorem~\ref{l:thm:AC} requires small
coordinate weights. We therefore delete concentrated coordinates from
the completed Arnoldi frame. All \(k\) queries have already been made:
deletion changes only the retained frame, leaving \(A\), its responses
and \(W_k\) unchanged. The separated columns remain available through
the reconstruction of Subsection~\ref{l:sec:real_links}.

Proposition~\ref{l:thm:SS} shows that a flat frame is reached after few
deletions, outside an event of probability \(e^{-cd^2}\). On this good
event, comparison transfers the mixed-reference link bound of
Lemma~\ref{l:thm:LM}; Subsection~\ref{l:sec:real_link_completion} then
completes the proof. We first give the deletion rule and its volume
certificate, followed by the geometric inputs and the stopping proof.
The geometric proofs are collected in Subsection~\ref{l:sec:real_geometry}.

\subsubsection{Coordinate deletion as a tool for the link
estimate}\label{l:coordinate-deletion-as-a-tool-for-the-link-estimate}

For \(J\subseteq[d]\), the restricted matrix
\(V_{J,:}=[(\mathbf v_1)_J,\ldots,(\mathbf v_k)_J]\) need not have orthonormal columns.

\begin{algorithms}[Delete concentrated coordinates after the queries]\label{l:thm:DEL}

Fix the frame \(V=[\mathbf v_1,\ldots,\mathbf v_k]\) obtained after the \(k\) queries of
Algorithm \ref{l:thm:ARN}, a cap \(\sigma\in(0,1/4]\), and an integer deletion
budget \(B_{\mathrm{del}}\ge0\) with \(B_{\mathrm{del}}+k\le d\),
\(B_{\mathrm{del}}<d\). Start with \(J=[d]\), and at each step form \begin{equation}
\mathsf G(J)=V_{J,:}^TV_{J,:},\qquad
 \ell_j(J)=V_{j,:}\mathsf G(J)^{-1}V_{j,:}^T\quad(j\in J),
 \label{l:eq:S.1}
\end{equation} using the coordinate weights only if \(\mathsf G(J)\) is nonsingular. For the
retained subspace \(E_J=\operatorname{Ran}V_{J,:}\subseteq\mathbb R^J\),
\[
 \ell_j(J)=(P_{E_J})_{jj}=\|P_{E_J}e_j\|^2
 =\sup_{\substack{\mathbf v\in E_J\\\|\mathbf v\|=1}}|(\mathbf v)_j|^2.
\] Thus \(\ell_j(J)\) is the largest squared \(j\)-th coordinate of a
unit vector in the retained subspace. Calling the frame flat means that
none of these coordinate weights exceeds \(\sigma\).

\Needspace{7\baselineskip}

\begin{enumerate}
\def\labelenumi{\arabic{enumi}.}
\tightlist
\item
  If \(\mathsf G(J)>0\) and every \(\ell_j(J)\le\sigma\), stop with
  \textbf{flat} output.
\item
  Otherwise, if \(d-|J|=B_{\mathrm{del}}\), stop with \textbf{active}
  output.
\item
  Otherwise delete the smallest \(j\in J\) with \(\ell_j(J)>\sigma\); if
  \(\mathsf G(J)\) is singular, delete the smallest \(j\in J\). Repeat from step
  1.
\end{enumerate}

Set \(I_{\mathrm{del}}=[d]\setminus J\). Only rows of the fixed frame
are deleted; no further queries are made. \end{algorithms}

All choices are deterministic and Borel for fixed \(V\), including at
singular Grams. An active output has used the whole budget; a flat
output first reached on deletion \(B_{\mathrm{del}}\) is still
classified as flat.

\begin{proposition}[Few deletions suffice to flatten the queries]\label{l:thm:SS}

Given \(\sigma_{\max}>0\), one can choose a cap
\(0<\sigma\le\min\{\sigma_{\max},1/4\}\) and an integer
\(L_{\mathrm{stop}}\ge2\), depending only on \(\mu,\sigma_{\max}\). For
every \(\nu,\alpha_{\max}\in(0,1]\), there then exist
\(\beta_{\mathrm{stop}}>0\), \(c_{\mathrm{stop}}>0\) and a dimension
threshold \(d_0\), depending only on
\(\mu,\sigma_{\max},\nu,\alpha_{\max}\), such that there exists a choice
of even integers \(k=k(d)\), one for each \(d\ge d_0\), satisfying \begin{equation}
\beta_{\mathrm{stop}}d\le k\le\alpha_{\max}d,
 \qquad B_{\mathrm{del}}=L_{\mathrm{stop}}k,
 \qquad B_{\mathrm{del}}+k\le d,\quad B_{\mathrm{del}}<d,
 \label{l:eq:S.5}
\end{equation} and Algorithm \ref{l:thm:DEL}, run with cap \(\sigma\), satisfies \begin{equation}
\mathbb P_{A,\mathbf s}\{\text{active output, or flat output with }
                |I_{\mathrm{del}}|\ge\lceil\nu k\rceil\}
 \le e^{-c_{\mathrm{stop}}d^2}.
 \label{l:eq:S.6}
\end{equation} The probability is taken under \[
 \mathcal L(A,\mathbf s)=\mu^{\otimes d^2}\otimes\mu^{\otimes d}.
\] Here \(A\) is the bottom-right core of \(\mathsf X_n\), and \(\mathbf s\) is
the additional independent starter for the Arnoldi iteration in
Algorithm \ref{l:thm:ARN}. In particular, \(\sigma\) and \(L_{\mathrm{stop}}\) are
independent of \(\nu,\alpha_{\max}\) and \(d\). \end{proposition}

The budget \(B_{\mathrm{del}}=L_{\mathrm{stop}}k\) bounds the actual
deletion count \(|I_{\mathrm{del}}|\). With
\(k=2\lfloor\alpha(n-1)/2\rfloor\), this gives the explicit expression
\[
 B_{\mathrm{del}}=L_{\mathrm{stop}}k
 =2L_{\mathrm{stop}}\left\lfloor\frac{\alpha(n-1)}2\right\rfloor.
\] The multiplier is fixed in \eqref{l:eq:P.4}. We call
\[
 \mathcal H:=\{\text{flat output},\ |I_{\mathrm{del}}|<\lceil\nu k\rceil\}
\]
the good deletion event. Since the count is integer, its condition is
equivalent to \(|I_{\mathrm{del}}|<\nu k\). The target \(\nu k\) is
used only to define success: the algorithm continues until the frame is
flat or the budget is exhausted. Proposition~\ref{l:thm:SS} bounds
\(\mathbb P_\mu(\mathcal H^c)\).

\subsubsection{The volume certificate}\label{l:how-the-deletion-rule-works}
\label{l:volume-loss-under-coordinate-deletion}

At a retained set \(J\), the matrix \(V_{J,:}\mathsf G(J)^{-1/2}\) is an
orthonormal frame whenever \(\mathsf G(J)>0\). Its squared row norms
are the coordinate weights \(\ell_j(J)\). Removing a row changes this
normalization and may increase the remaining weights, so the rule
recomputes them after each deletion. There is no assertion that removing
only the initially concentrated rows suffices.

Each deletion is a rank-one update. When \(\mathsf G(J)>0\), the matrix
determinant lemma, also used in reverse iterative volume sampling
\cite[Theorem 2]{DW2018}, gives
\[
 \det\mathsf G(J\setminus\{j\})
 =\det(\mathsf G(J)-V_{j,:}^TV_{j,:})
 =\det\mathsf G(J)(1-\ell_j(J)).
\]
At a nonsingular step the deleted weight exceeds \(\sigma\), so the
determinant falls by a factor at most \(1-\sigma\). Once a Gram is
singular it remains singular under further deletion. Starting from
\(\mathsf G([d])=\mathrm{Id}_k\), we obtain \begin{equation}
\det \mathsf G(J)\le(1-\sigma)^{|I_{\mathrm{del}}|}
             \le e^{-\sigma|I_{\mathrm{del}}|}.
 \label{l:eq:S.2}
\end{equation} An output with
\(s\) deletions therefore has logarithmic volume loss at least
\(\sigma s\). On a flat output, \(0<\mathsf G(J)\preceq\mathrm{Id}_k\)
and, for \(j\in J\),
\(\sum_{a=1}^k(\mathbf v_a)_j^2\le\ell_j(J)\le\sigma\).
The full-rank argument below shows that singular Grams have probability
zero in our application; the algorithm nevertheless specifies their
outcome.

\subsubsection{Two geometric estimates and their
roles}\label{l:two-geometric-estimates-and-their-roles}

In both lemmas below, the matrix has independent entries, each column
consists of iid entries of law \(\mu\) or iid standard Gaussian
entries, and the column types are fixed in advance. Constants are uniform over these
types. Here \(m_{\mathrm{col}}\) counts columns; \(m=k/2\) remains
reserved for the border and final link estimates. The proofs are in
Subsection \ref{l:sec:real_geometry}.

\begin{definition}[Coordinate-volume loss]\label{l:thm:VLOSS}

For a linear subspace \(E\subseteq\mathbb R^d\) and \(I_0\subseteq[d]\)
of size \(s\), define \[
 \Delta_{I_0}(E)=-\log\det(\mathrm{Id}_s-(P_E)_{I_0,I_0}),
 \qquad-\log0=+\infty.
\] For any orthonormal frame \(Q_E\) of \(E\), Sylvester's identity
gives
\[
 \Delta_{I_0}(E)=-\log\det((Q_E)_{I_0^c,:}^T(Q_E)_{I_0^c,:}),
\]
with empty determinant one. Thus this is the loss measured in \eqref{l:eq:S.2}.
\end{definition}

Since \(E\subseteq E'\) implies \(P_E\preceq P_{E'}\), determinant
monotonicity on the coordinate block gives, including singular cases, \begin{equation}
E\subseteq E'\quad\Longrightarrow\quad\Delta_{I_0}(E)\le\Delta_{I_0}(E').
 \label{l:eq:G.8}
\end{equation}

\Needspace{12\baselineskip}
\begin{lemma}[Volume loss for independent columns]\label{l:thm:GV}

Fix \(a>0\). There are \(\kappa_{\mathrm{vol}}(\mu,a)>0\) and
\(c_{\mathrm{vol}}(\mu,a)>0\) such that the following holds. Let
\(\mathsf M\in\mathbb R^{d\times {m_{\mathrm{col}}}}\) have independent
columns with \[
 \mathsf M_{:,j}\sim\nu_j^{\otimes d},\qquad
 \nu_j\in\{\mu,N(0,1)\},\qquad 1\le j\le {m_{\mathrm{col}}},
\] where the laws \(\nu_1,\ldots,\nu_{m_{\mathrm{col}}}\) are fixed in
advance. Uniformly over these choices, for every fixed deletion set
\(I_0\subseteq[d]\) of size \(s\), \begin{equation}
1\le s<{m_{\mathrm{col}}},\quad {m_{\mathrm{col}}}^2\le\kappa_{\mathrm{vol}}ds
 \quad\Longrightarrow\quad
 \mathbb P\{\Delta_{I_0}(\operatorname{Ran}\mathsf M)\ge as\}
 \le e^{-c_{\mathrm{vol}}ds^2/{m_{\mathrm{col}}}}.
 \label{l:eq:G.9}
\end{equation} \end{lemma}

\begin{lemma}[Capped coordinate-weight exponential moment]\label{l:thm:GE}

For each fixed \(\theta>0\) there are
\(\delta_\theta,\alpha_\theta,C_\theta>0\), depending only on
\(\mu,\theta\), such that, for the independent-column matrix
\(\mathsf M\in\mathbb R^{h\times {m_{\mathrm{col}}}}\) and the
coordinate weights of its range
\(\ell_i=(P_{\operatorname{Ran}\mathsf M})_{ii}\),
\(1\le {m_{\mathrm{col}}}\le h\), \(0<\delta\le\delta_\theta\), and
\({m_{\mathrm{col}}}/h\le\alpha_\theta\), \begin{equation}
\log\mathbb E e^{\theta h\sum_{i=1}^h\min(\ell_i,\delta)^{3/2}}
 \le C_\theta h^2({m_{\mathrm{col}}}/h)^{3/2}.
 \label{l:eq:G.6}
\end{equation} \end{lemma}

Lemma \ref{l:thm:GV} makes the volume-loss certificate unlikely under the reference
law. Lemma \ref{l:thm:GE} controls the exponential weight in the comparison for flat
outputs. The cap is essential here: the larger reference space need not
itself be flat, even when the query space it contains is flat.

\subsubsection{The volume certificate under a fixed reference
law}\label{l:the-volume-certificate-under-a-fixed-reference-law}

Fix \(I_0\) with \(s=|I_0|\le d-k\), set \(J_0=I_0^c\), and work under
\(\mathsf Q_{I_0}\). Lemma \ref{l:thm:RECONPROP} shows that \begin{equation}
\mathsf M^{\mathrm{ref}(I_0)}
 =[\mathbf s,A_{:,I_0},\mathbf y_1^{\mathrm{rem}},\ldots,\mathbf y_k^{\mathrm{rem}}]
 \in\mathbb R^{d\times(s+k+1)}
 \label{l:eq:P.1}
\end{equation} has independent entries: the first \(s+1\) columns have law \(\mu\)
and the last \(k\) are Gaussian.

Almost surely, every restriction of
\([\mathbf s,A\mathbf s,\ldots,A^{k-1}\mathbf s]\) to at least \(k\)
rows has rank \(k\). Indeed, a selected \(k\)-row minor is a nonzero
polynomial: diagonal \(A\) with distinct entries and
\(\mathbf s=(1,\ldots,1)^T\) give a Vandermonde witness. Both the
original and mixed physical laws have joint densities. Thus
\(\mathbf s\ne0\) and all residuals used to form the first \(k\) Arnoldi
directions are nonzero: \(R_j>0\) for \(1\le j<k\). Now
\(\mathbf v_1\in\operatorname{span}\{\mathbf s\}\), and \eqref{l:eq:S.3} expresses each
\(A\mathbf v_j\) as a linear combination of the columns of
\(\mathsf M^{\mathrm{ref}(I_0)}\). Since these \(R_j\) are positive,
Arnoldi orthogonalization keeps every \(\mathbf v_j\) in this column space, by
induction. Restricting the resulting inclusion to \(J_0\) gives \begin{equation}
\operatorname{Ran}V\subseteq\operatorname{Ran}\mathsf M^{\mathrm{ref}(I_0)},
 \qquad
 \operatorname{Ran}V_{J_0,:}\subseteq
       \operatorname{Ran}\mathsf M^{\mathrm{ref}(I_0)}_{J_0,:}.
 \label{l:eq:P.2}
\end{equation} On \(\{I_{\mathrm{del}}=I_0\}\), \eqref{l:eq:S.2} and \eqref{l:eq:G.8} imply
\(\Delta_{I_0}(\operatorname{Ran}\mathsf M^{\mathrm{ref}(I_0)})\ge\sigma s\).
Lemma \ref{l:thm:GV}, with \(a=\sigma\) and \({m_{\mathrm{col}}}=s+k+1\), therefore
gives \begin{equation}
\mathsf Q_{I_0}\{I_{\mathrm{del}}=I_0\}
 \le\exp\{-c_{\mathrm{vol}}ds^2/(s+k+1)\}
 \label{l:eq:P.3}
\end{equation} whenever \({m_{\mathrm{col}}}^2\le\kappa_{\mathrm{vol}}ds\). Here and
below the volume constants are evaluated at \((\mu,\sigma)\).

\subsubsection{Proof of the stopping
estimate}\label{l:sec:real_stopping_proof}

\begin{proof}[\proofname\ of Proposition \ref{l:thm:SS}]
We estimate the complement of the good deletion event \(\mathcal H\)
by fixing candidate deletion sets before applying comparison.

\Needspace{6\baselineskip}
\textbf{Split the failure event.} For each fixed candidate deletion set
\(I_0\subseteq[d]\), write \[
 s:=|I_0|,\qquad J_0:=[d]\setminus I_0,\qquad h:=|J_0|=d-s.
\] Thus \(s\) is a deterministic candidate deletion count; on
\(\{I_{\mathrm{del}}=I_0\}\), it equals the actual number deleted.
Define \[
 \begin{aligned}
 F_{I_0}^{\mathrm{act}}
   &:=\{\text{active output},\ I_{\mathrm{del}}=I_0\},\\
 F_{I_0}^{\mathrm{flat}}
   &:=\{\text{flat output},\ I_{\mathrm{del}}=I_0\}.
 \end{aligned}
\] An active output always has \(|I_{\mathrm{del}}|=B_{\mathrm{del}}\).
Thus the event in \eqref{l:eq:S.6} is exactly \[
 \mathcal E:=\mathcal H^c=
 \mathop{\dot\bigcup}_{|I_0|=B_{\mathrm{del}}}F_{I_0}^{\mathrm{act}}
 \,\dot\cup\,
 \mathop{\dot\bigcup}_{\lceil\nu k\rceil\le |I_0|\le B_{\mathrm{del}}}
 F_{I_0}^{\mathrm{flat}}.
\] Lemma \ref{l:thm:RECONPROP} makes each fixed-set event observable. We estimate it
under its own reference law \(\mathsf Q_{I_0}\), then compare back to
\(\mathbb P_\mu\) before summing.

\Needspace{6\baselineskip}
\textbf{A volume bound for both branches.} Choose the cap \[
 0<\sigma\le\min\{\sigma_{\max},1/4,\delta_\mu,
                                      \delta_{4C_{\mathrm{weight}}}\}.
\] This depends only on \(\mu,\sigma_{\max}\) and permits both weighted
comparison and Lemma \ref{l:thm:GE}. The constants \(c_{\mathrm{vol}}\) and
\(\kappa_{\mathrm{vol}}\) below are those of Lemma \ref{l:thm:GV} at \(a=\sigma\).
We will choose \(L_{\mathrm{stop}}\ge2\) and a sufficiently small
\(\alpha>0\), and set \(k=2\lfloor\alpha d/2\rfloor\). For now, assume
the geometric hypotheses hold for all
\(\lceil\nu k\rceil\le s\le L_{\mathrm{stop}}k\); we verify this
simultaneously at the end. For a fixed \(I_0\) of size \(s\), write
\({m_{\mathrm{col}}}=s+k+1\). The volume certificate gives
\eqref{l:eq:P.3}:
\[
 \mathsf Q_{I_0}\{I_{\mathrm{del}}=I_0\}
 \le e^{-c_{\mathrm{vol}}ds^2/(s+k+1)}.
\]
\Needspace{6\baselineskip}
\textbf{Active outputs: choose the budget large enough.} Fix a
candidate \(I_0\) of size
\(s=|I_0|=B_{\mathrm{del}}=L_{\mathrm{stop}}k\): an active output has
used the full deletion budget. Since
\(F_{I_0}^{\mathrm{act}}\subseteq\{I_{\mathrm{del}}=I_0\}\) and
\(s+k+1\le3s\) for \(k\ge2\), coarse comparison \eqref{l:eq:A.2} gives \[
 \begin{aligned}
 \mathbb P_\mu(F_{I_0}^{\mathrm{act}})
 &\le e^{C_{\mathrm{coarse}}dk}
       \mathsf Q_{I_0}(F_{I_0}^{\mathrm{act}})^{\theta_{\mathrm{coarse}}}\\
 &\le\exp\{(C_{\mathrm{coarse}}
          -\theta_{\mathrm{coarse}}c_{\mathrm{vol}}L_{\mathrm{stop}}/3)dk\}.
 \end{aligned}
\] The comparison costs a fixed multiple of \(dk\), whereas the volume
bound supplies a multiple of \(L_{\mathrm{stop}}dk\). We can therefore
fix \begin{equation}
L_{\mathrm{stop}}:=\max\left\{2,
 \left\lceil\frac{6(C_{\mathrm{coarse}}+1)}
 {\theta_{\mathrm{coarse}}c_{\mathrm{vol}}}\right\rceil\right\}.
 \label{l:eq:P.4}
\end{equation} Then \(\theta_{\mathrm{coarse}}c_{\mathrm{vol}}L_{\mathrm{stop}}/6
\ge C_{\mathrm{coarse}}+1\), and in particular
\(\mathbb P_\mu(F_{I_0}^{\mathrm{act}})\le e^{-dk}\).

\Needspace{6\baselineskip}
\textbf{Flat but late outputs: make the weighted cost small.} Fix a
candidate \(I_0\) whose deletion count \(s=|I_0|\) satisfies \[
 \lceil\nu k\rceil\le s\le B_{\mathrm{del}}=L_{\mathrm{stop}}k.
\] Since \(F_{I_0}^{\mathrm{flat}}\subseteq\{I_{\mathrm{del}}=I_0\}\),
the common reference estimate, \(s\ge\nu k\), and
\({m_{\mathrm{col}}}=s+k+1\le(L_{\mathrm{stop}}+2)k\) give \[
 \mathsf Q_{I_0}(F_{I_0}^{\mathrm{flat}})
 \le e^{-c_{\mathrm{vol}}\nu^2dk/(L_{\mathrm{stop}}+2)}.
\] When \(\nu\) is small, this decay need not absorb the fixed
\(C_{\mathrm{coarse}}dk\) cost. We use flatness to apply
\eqref{l:eq:A.3}, keeping its random comparison cost inside the
expectation. On the retained coordinates \(J_0\), set
\[
 \Psi_{I_0}=\sum_{j\in J_0}
 \min\{(P_{\operatorname{Ran}\mathsf M^{\mathrm{ref}(I_0)}_{J_0,:}})_{jj},
       \sigma\}^{3/2}.
\] On \(F_{I_0}^{\mathrm{flat}}\), the reconstructed frame
\(\widehat V\) spans \(\operatorname{Ran}V_{J_0,:}\). Flatness and
\eqref{l:eq:P.2} therefore give
\[
 \ell_j(\widehat V)\le
 \min\{(P_{\operatorname{Ran}\mathsf M^{\mathrm{ref}(I_0)}_{J_0,:}})_{jj},
       \sigma\},\qquad j\in J_0.
\]
Thus \(\Theta(\widehat V)\le\Psi_{I_0}\) on this event.
Applying \eqref{l:eq:A.3} and then
Cauchy--Schwarz yields \begin{equation}
\begin{aligned}
 \mathbb P_\mu(F_{I_0}^{\mathrm{flat}})
 &\le\left[\mathbb E_{\mathsf Q_{I_0}}
       \mathbf1_{F_{I_0}^{\mathrm{flat}}}
       e^{C_{\mathrm{weight}}d\Psi_{I_0}}\right]^{\theta_\mu}\\
 &\le\left[\mathsf Q_{I_0}(F_{I_0}^{\mathrm{flat}})
       \mathbb E_{\mathsf Q_{I_0}}e^{2C_{\mathrm{weight}}d\Psi_{I_0}}
       \right]^{\theta_\mu/2}.
 \end{aligned}
 \label{l:eq:P.6}
\end{equation}
The first factor is the late-output probability controlled by the
volume bound above; the second controls the cost of comparison.
For this second factor, \(h\ge d/2\) implies
\(2d/h\le4\), so Lemma \ref{l:thm:GE}
applies at parameter \(4C_{\mathrm{weight}}\). Set
\(C_{\mathrm{cap}}=C_{4C_{\mathrm{weight}}}(L_{\mathrm{stop}}+2)^{3/2}\).
Since \(h\le d\), \({m_{\mathrm{col}}}\le(L_{\mathrm{stop}}+2)k\), and
\(k/d\le\alpha\), it gives \begin{equation}
\log\mathbb E_{\mathsf Q_{I_0}}e^{2C_{\mathrm{weight}}d\Psi_{I_0}}
 \le C_{4C_{\mathrm{weight}}}\sqrt h\,{m_{\mathrm{col}}}^{3/2}
 \le C_{\mathrm{cap}}\sqrt\alpha\,dk.
 \label{l:eq:P.7}
\end{equation} Substitution in \eqref{l:eq:P.6} makes the required balance explicit: \[
 \mathbb P_\mu(F_{I_0}^{\mathrm{flat}})
 \le\exp\left\{-\frac{\theta_\mu}{2}
 \left(\frac{c_{\mathrm{vol}}\nu^2}{L_{\mathrm{stop}}+2}
               -C_{\mathrm{cap}}\sqrt\alpha\right)dk\right\}.
\] Consequently,
\(\mathbb P_\mu(F_{I_0}^{\mathrm{flat}})\le e^{-c_{\mathrm{flat}}dk}\)
provided that
\[
 C_{\mathrm{cap}}\sqrt\alpha\le
 \frac{c_{\mathrm{vol}}\nu^2}{2(L_{\mathrm{stop}}+2)}
\]
where
\[
 c_{\mathrm{flat}}:=\frac{\theta_\mu c_{\mathrm{vol}}\nu^2}
                            {4(L_{\mathrm{stop}}+2)}>0.
\]

\Needspace{6\baselineskip}
\textbf{Verify the parameter conditions.} The cap \(\sigma\) and
multiplier \(L_{\mathrm{stop}}\) are now fixed. Choose \(\alpha>0\)
small enough that \begin{equation}
\begin{gathered}
 \alpha\le\alpha_{\max},\quad (L_{\mathrm{stop}}+2)\alpha\le1/4,\quad
 (L_{\mathrm{stop}}+2)^2\alpha\le\kappa_{\mathrm{vol}}\nu,\quad
 2(L_{\mathrm{stop}}+2)\alpha\le\alpha_{4C_{\mathrm{weight}}},\\
 C_{\mathrm{cap}}\sqrt\alpha
       \le c_{\mathrm{vol}}\nu^2/[2(L_{\mathrm{stop}}+2)].
 \end{gathered}
 \label{l:eq:P.5}
\end{equation} All right-hand bounds are positive and independent of \(d\), so these
conditions are compatible. The first gives the prescribed ceiling on
\(k/d\); the last is exactly the cost bound just obtained. To check the
geometric hypotheses, take any candidate deletion count \(s=|I_0|\) in
\([\lceil\nu k\rceil,L_{\mathrm{stop}}k]\). Then \[
 \begin{gathered}
 {m_{\mathrm{col}}}=s+k+1\le(L_{\mathrm{stop}}+2)k,\qquad h=d-s\ge d/2,\\
 {m_{\mathrm{col}}}^2\le(L_{\mathrm{stop}}+2)^2k^2\le\kappa_{\mathrm{vol}}\nu dk
                            \le\kappa_{\mathrm{vol}}ds,\\
 \frac {m_{\mathrm{col}}}h\le2(L_{\mathrm{stop}}+2)\frac kd
                   \le\alpha_{4C_{\mathrm{weight}}}.
 \end{gathered}
\] These are precisely the volume and capped-moment hypotheses used
above. For sufficiently large \(d\), the rounding of \(k\) also gives
\(k\ge\max\{2,\alpha d/2\}\). With
\(B_{\mathrm{del}}=L_{\mathrm{stop}}k\), we also have
\(B_{\mathrm{del}}+k\le d\) and \(B_{\mathrm{del}}<d\). Thus the
algorithm and both estimates apply.

\Needspace{6\baselineskip}
\textbf{Sum the fixed-set bounds.} Using the event decomposition, with
at most \(2^d\) possible sets in each branch, we obtain \[
 \begin{aligned}
 \mathbb P_\mu(\mathcal E)
 &=\sum_{|I_0|=B_{\mathrm{del}}}\mathbb P_\mu(F_{I_0}^{\mathrm{act}})
   +\sum_{\lceil\nu k\rceil\le|I_0|\le B_{\mathrm{del}}}
                         \mathbb P_\mu(F_{I_0}^{\mathrm{flat}})\\
 &\le2^d(e^{-dk}+e^{-c_{\mathrm{flat}}dk})\\
 &\le2\exp\{d\log2-\tfrac\alpha2
                              \min(1,c_{\mathrm{flat}})d^2\}
 \le e^{-c_{\mathrm{stop}}d^2}
 \end{aligned}
\] for sufficiently large \(d\), with, for example,
\(c_{\mathrm{stop}}=\alpha\min(1,c_{\mathrm{flat}})/4\). This proves
\eqref{l:eq:S.6} and gives \eqref{l:eq:S.5} with \(\beta_{\mathrm{stop}}=\alpha/2\). Every
summand was estimated under the original law; no conditioning on the
random deletion set is used. \end{proof}

\subsubsection{Completion of the original-law link
estimate}\label{l:sec:real_link_completion}

\begin{proof}[\proofname\ of Proposition \ref{l:thm:ML}]
Fix \(\varepsilon\in(0,1]\). Take \(\eta_{\mu,\varepsilon}\) and
\(c_{\mathrm{mix}}\) from Lemma \ref{l:thm:LM}. Choose
\(0<\tau\le\min\{\delta_\mu,1/4\}\) so small that \[
 C_{\mathrm{weight}}\sqrt\tau\le c_{\mathrm{mix}}/4.
\] Apply Proposition \ref{l:thm:SS} with \(\sigma_{\max}=\tau\),
\(\nu=\eta_{\mu,\varepsilon}\) and \(\alpha_{\max}=\varepsilon/128\). It
supplies a cap \(\sigma\le\tau\), an even \(k\) and a deletion budget
\(B_{\mathrm{del}}\), with \(k\ge\beta_{\mathrm{stop}}d\). Set
\(m=k/2\), so \(M_k=m^2\).
The choice of \(\nu\) allows the mixed-reference link estimate on every
good deletion output; the choice of \(\tau\) pays the comparison cost.

\Needspace{6\baselineskip}
\textbf{Discard unsuccessful deletion.} On the original probability space of
\((A,\mathbf s)\), define \[
 \mathcal B_\varepsilon
   :=\{W_k(A,\mathbf s)\le(1-\varepsilon/4)^{M_k}\}.
\] For the good deletion event \(\mathcal H\) defined above,
\eqref{l:eq:S.6} gives
\(\mathbb P_\mu(\mathcal H^c)\le e^{-c_{\mathrm{stop}}d^2}\).
Define the deterministic family of candidate sets and their events by
\[
 \begin{aligned}
 \mathcal I&:=\{I_0\subseteq[d]:|I_0|<\lceil\nu k\rceil,
                              \ |I_0|\le B_{\mathrm{del}}\},\\
 F_{I_0}&:=\mathcal B_\varepsilon\cap\mathcal H
                                  \cap\{I_{\mathrm{del}}=I_0\}.
 \end{aligned}
\] The deletion count never exceeds its budget, so we have the disjoint
union \[
 \mathcal B_\varepsilon
 = (\mathcal B_\varepsilon\cap\mathcal H^c)
       \,\dot\cup\,\mathop{\dot\bigcup}_{I_0\in\mathcal I}F_{I_0}.
\] In particular, \[
 \mathbb P_\mu(\mathcal B_\varepsilon)
 \le e^{-c_{\mathrm{stop}}d^2}
       +\sum_{I_0\in\mathcal I}\mathbb P_\mu(F_{I_0}).
\]

\Needspace{6\baselineskip}
\textbf{Compare bad links on each good output.} Fix \(I_0\in\mathcal I\) and write
\(J_0=[d]\setminus I_0\). Since \(|I_0|\le B_{\mathrm{del}}\) and
\(B_{\mathrm{del}}+k\le d\), we have \(k\le|J_0|\). Lemma \ref{l:thm:RECONPROP}
makes \(F_{I_0}\) measurable in the reconstructed products and auxiliary
data. On this event the retained frame has full rank and
\(\max_{i\in J_0}\ell_i(\widehat V)\le\sigma\); since the weights sum to
\(k\), we have \(\Theta(\widehat V)\le k\sqrt\sigma\).
The integer inequality
\(|I_0|<\lceil\nu k\rceil\) implies
\(|I_0|<\nu k\le\eta_{\mu,\varepsilon}k\). Since
\(F_{I_0}\subseteq\mathcal B_\varepsilon\), Lemma \ref{l:thm:LM} gives \[
 \mathsf Q_{I_0}(F_{I_0})
 \le\mathsf Q_{I_0}(\mathcal B_\varepsilon)
 \le2e^{-c_{\mathrm{mix}}dk}.
\] Since \(\sigma\le\tau\), we have
\(C_{\mathrm{weight}}\sqrt\sigma\le c_{\mathrm{mix}}/4\).
The flat-coordinate comparison of Theorem~\ref{l:thm:AC} gives
\begin{equation}
\begin{aligned}
 \mathbb P_\mu(F_{I_0})
 &\le e^{\theta_\mu C_{\mathrm{weight}}\sqrt\sigma\,dk}
          \mathsf Q_{I_0}(F_{I_0})^{\theta_\mu}\\
 &\le2^{\theta_\mu}
    e^{-\theta_\mu(c_{\mathrm{mix}}-C_{\mathrm{weight}}\sqrt\sigma)dk}
 \le e^{-c_{\mathrm{cand}}dk},\qquad c_{\mathrm{cand}}:=\theta_\mu c_{\mathrm{mix}}/2,
 \end{aligned}
 \label{l:eq:S.7}
\end{equation} for sufficiently large \(d\). These are probabilities of the full
events \(F_{I_0}\), not probabilities conditioned on
\(I_{\mathrm{del}}=I_0\).

\Needspace{6\baselineskip}
\textbf{Summing the candidates.} Since \(|\mathcal I|\le2^d\) and
\(k\ge\beta_{\mathrm{stop}}d\), \[
 \begin{aligned}
 \sum_{I_0\in\mathcal I}\mathbb P_\mu(F_{I_0})
 &\le 2^d e^{-c_{\mathrm{cand}}dk}
 \le e^{d\log2-c_{\mathrm{cand}}\beta_{\mathrm{stop}}d^2}\\
 &\le e^{-c_{\mathrm{few}}d^2},\qquad
 c_{\mathrm{few}}:=c_{\mathrm{cand}}\beta_{\mathrm{stop}}/2,
 \end{aligned}
\] for sufficiently large \(d\). Substituting into the event
decomposition yields \[
 \mathbb P_\mu(\mathcal B_\varepsilon)
 \le e^{-c_{\mathrm{stop}}d^2}+e^{-c_{\mathrm{few}}d^2}
 \le e^{-c_{\mathrm{link}}d^2},\qquad
 c_{\mathrm{link}}:=\tfrac12\min\{c_{\mathrm{stop}},c_{\mathrm{few}}\},
\] after increasing the dimension threshold. This is \eqref{l:eq:M.8}. Finally,
\(2m=k\ge\beta_{\mathrm{stop}}d\) and \(256m=128k\le\varepsilon d\), so
\eqref{l:eq:M.7} holds with \(\beta_{\mu,\varepsilon}=\beta_{\mathrm{stop}}\); also
\(m\ge1\) for large \(d\). \end{proof}

\subsubsection{Proofs of the two geometric
estimates}\label{l:sec:real_geometry}

We now prove Lemmas \ref{l:thm:GV} and \ref{l:thm:GE} for the independent-entry matrices with
fixed column types specified above. Lemma \ref{l:thm:SVAL} supplies the
singular-value estimates used in the proof of Lemma \ref{l:thm:GE}.

\paragraph{Projection densities and inverse Gram
moments}\label{l:projection-densities-and-inverse-gram-moments}

Projection-density bounds for independent bounded-density coordinates
are established in \cite[Theorem 1.1]{RV2015}. We give the fiberwise
version below, since later substitutions require the specified density
representative on every fiber.

\begin{lemma}[Canonical projection bound]\label{l:thm:GP}

If the independent coordinates of \(x\in\mathbb R^h\) have everywhere
bounded densities \(f_j\le K\), then for every coisometry
\(L:\mathbb R^h\to\mathbb R^q\), \(LL^T=\mathrm{Id}_q\), the canonical
density satisfies \begin{equation}
\int_{\ker L}\prod_j f_j((L^Ty)_j+w_j)\,dw
 \le(\sqrt{\mathrm e}K)^q
 \quad\text{for every }y.
 \label{l:eq:G.1}
\end{equation} Consequently, for a fixed subspace \(F\) of codimension \(q\ge1\), \begin{equation}
\mathbb P\{\operatorname{dist}(x,F)\le u\sqrt q\}\le\min\{1,(C_Ku)^q\},
 \qquad
 \|\operatorname{dist}(x,F)^{-2}\|_{L^p}\le C'_K/q
 \quad(1\le p\le q/4).
 \label{l:eq:G.2}
\end{equation} \end{lemma}

\begin{proof}

Write \(\ell_j^{\ker}=(P_{\ker L})_{jj}\). For \(\ell_j^{\ker}>0\), its
normalized coordinate maps \(B_j:\ker L\to\mathbb R\) satisfy
\(B_jB_j^T=1\) and
\(\sum_{\substack{1\le j\le h\\\ell_j^{\ker}>0}}\ell_j^{\ker}B_j^TB_j=\mathrm{Id}_{\ker L}\).
The geometric Brascamp--Lieb inequality \cite[Proposition 2.8]{BCCT2008}
says \[
 \int_{\ker L}\prod_{\substack{1\le j\le h\\\ell_j^{\ker}>0}}
                    h_j(B_jw)^{\ell_j^{\ker}}\,dw
 \le\prod_{\substack{1\le j\le h\\\ell_j^{\ker}>0}}
                    \left(\int_{\mathbb R}h_j\right)^{\ell_j^{\ker}}.
\] For each \(1\le j\le h\) with \(\ell_j^{\ker}>0\), apply it to
\(h_j=f_j((L^Ty)_j+\sqrt{\ell_j^{\ker}}\,\cdot)^{1/\ell_j^{\ker}}\). Its
integral is at most \((\ell_j^{\ker})^{-1/2}K^{1/\ell_j^{\ker}-1}\). The
zero-coordinate maps contribute constants at most \(K\). Since
\(\sum_{j=1}^h(1-\ell_j^{\ker})=q\), the fiber bound is \[
 K^q\prod_{\ell_j^{\ker}>0}(\ell_j^{\ker})^{-\ell_j^{\ker}/2}\le K^qe^{q/2},
\] using \(-c\log c\le1-c\). This proves \eqref{l:eq:G.1} for the actual chosen
density representatives, with no exceptional fiber. Orthogonal
disintegration identifies it as a projection density. Multiplication by
the \(q\)-ball volume proves the first estimate in \eqref{l:eq:G.2}. More
explicitly, for \(Y=\operatorname{dist}(x,F)/\sqrt q\) and \(0<s<q\), \[
 \mathbb E Y^{-s}
 \le sC_K^q\int_0^{1/C_K}u^{q-s-1}\,du
    +s\int_{1/C_K}^\infty u^{-s-1}\,du
 =C_K^s\frac q{q-s}.
\] Take \(s=2p\le q/2\) and a \(p\)-th root to obtain the second bound.
\end{proof}

\begin{corollary}[Inverse Gram moments]\label{l:thm:IGRAM}

Assume \eqref{l:eq:M.1}. Let \(\mathsf M\in\mathbb R^{h\times m_{\mathrm{col}}}\)
have independent entries, with each column consisting of iid entries of
law \(\mu\) or \(N(0,1)\), the column types being fixed. If
\(1\le m_{\mathrm{col}}\le h/2\), then \(\mathsf G=\mathsf M^T\mathsf M>0\)
almost surely. For every \(J\subseteq[m_{\mathrm{col}}]\), write
\(J^c=[m_{\mathrm{col}}]\setminus J\). For \(t\ge0\) and
\(1\le p\le h/8\), \begin{equation}
\begin{aligned}
 \left\|\frac{\det \mathsf G_{J^c,J^c}}{\det \mathsf G}\right\|_{L^p}
 &\le(C_{\mathrm{inv}}(\mu)/h)^{|J|},\\
 \|\det(\mathrm{Id}_{m_{\mathrm{col}}}+th\mathsf G^{-1})\|_{L^p}
 &\le(1+C_{\mathrm{inv}}(\mu)t)^{m_{\mathrm{col}}}.
 \end{aligned}
 \label{l:eq:G.3}
\end{equation} Here \(C_{\mathrm{inv}}(\mu)>0\) depends only on \(\mu\), uniformly
over the column types, and the empty determinant is one. \end{corollary}

\begin{proof}

Each column has a density and hence almost surely avoids the span of its
predecessors; thus \(\mathsf G>0\). First expose the columns in \(J^c\), then
add those in \(J\) in a fixed order. The determinant ratio is the
product of the inverse squared distances of each new column to the
preceding span. Every such span has codimension at least
\(h-{m_{\mathrm{col}}}+1\). Apply \eqref{l:eq:G.2} to the last factor conditional
on all preceding columns, and repeat backwards. This bounds the moment
of the product without asserting independence of the distances. The
second estimate follows from the complementary-minor identity
\(\det(\mathrm{Id}_{m_{\mathrm{col}}}+th\mathsf G^{-1})=\sum_{J\subseteq[{m_{\mathrm{col}}}]}(th)^{|J|}
\det \mathsf G_{J^c,J^c}/\det \mathsf G\) and Minkowski's inequality. \end{proof}

\paragraph{Proof of the coordinate-volume
bound}\label{l:proof-of-the-coordinate-volume-bound}

\begin{proof}[Proof of Lemma \ref{l:thm:GV}]Write
\(\mathsf G_{\mathrm{del}}=\mathsf M_{I_0,:}^T\mathsf M_{I_0,:}\),
\(\mathsf G_{\mathrm{keep}}=\mathsf M_{I_0^c,:}^T\mathsf M_{I_0^c,:}\), and
\(h=d-s\). These Grams are independent because they involve disjoint
rows. For the small aspect ratios chosen below,
\(m_{\mathrm{col}}<h/2\), and Corollary \ref{l:thm:IGRAM} gives
\(\mathsf G_{\mathrm{keep}}>0\) almost surely. The matrix
\(\mathsf M(\mathsf G_{\mathrm{del}}+\mathsf G_{\mathrm{keep}})^{-1/2}\) has orthonormal
columns spanning \(\operatorname{Ran}\mathsf M\): its Gram is the
identity and the right factor is invertible. After restricting this
frame to \(I_0^c\), its Gram determinant, hence its squared volume, is
\(\det \mathsf G_{\mathrm{keep}}/\det(\mathsf G_{\mathrm{del}}+\mathsf G_{\mathrm{keep}})\);
taking the negative logarithm gives the equality below. Thus \begin{equation}
\Delta_{I_0}(\operatorname{Ran}\mathsf M)=\log\frac{\det(\mathsf G_{\mathrm{del}}+\mathsf G_{\mathrm{keep}})}
                              {\det \mathsf G_{\mathrm{keep}}}
 \le\log\det(\mathrm{Id}_{m_{\mathrm{col}}}+\tau d\mathsf G_{\mathrm{keep}}^{-1})
       +\frac{\operatorname{tr}\mathsf G_{\mathrm{del}}}{\tau d},\quad\tau>0.
 \label{l:eq:G.10}
\end{equation} To verify the inequality, add
\(\tau d\mathrm{Id}_{m_{\mathrm{col}}}\) to the retained Gram in the
numerator. Factor its determinant, apply
\(\log\det(\mathrm{Id}_{m_{\mathrm{col}}}+D)\le
\operatorname{tr}D\) for a positive semidefinite matrix \(D\), and use
\((\mathsf G_{\mathrm{keep}}+\tau d\mathrm{Id}_{m_{\mathrm{col}}})^{-1}\preceq(\tau d)^{-1}\mathrm{Id}_{m_{\mathrm{col}}}\).
No commutation of the two Grams is used.

Choose \(\lambda>0\) and \(C_{\mathrm{vol}}\ge1\), depending on \(\mu\),
such that \(C_{\mathrm{vol}}\ge4C_{\mathrm{inv}}/3\) and
\(\log\mathbb E e^{t\xi^2}\le C_{\mathrm{vol}}t\) for both scalar laws
and \(0\le t\le\lambda\). Set \[
 \tau_0=\frac a{4C_{\mathrm{vol}}},\quad
 q_0=\min\{3/32,\lambda\tau_0\},\quad
 \kappa_{\mathrm{vol}}=\min\{1/4,q_0,a\tau_0/(4C_{\mathrm{vol}})\}.
\] Set \(\tau=\tau_0s/{m_{\mathrm{col}}}\),
\(q=q_0s/{m_{\mathrm{col}}}\), \(p=qd\). The assumed aspect bound gives
\(s<{m_{\mathrm{col}}}<d/4\), \(h>3d/4\),
\(1\le {m_{\mathrm{col}}}\le p<h/8\), and \(q/\tau\le\lambda\). Thus
Corollary \ref{l:thm:IGRAM} applies at moment order \(p\). From \eqref{l:eq:G.10},
independence, and the entrywise square-exponential bound, \[
 \begin{aligned}
 \log\mathbb E e^{p\Delta_{I_0}(\operatorname{Ran}\mathsf M)}
 &\le p{m_{\mathrm{col}}}\log(1+C_{\mathrm{inv}}\tau d/h)
       +C_{\mathrm{vol}}qs{m_{\mathrm{col}}}/\tau\\
 &\le C_{\mathrm{vol}}q\tau d{m_{\mathrm{col}}}+C_{\mathrm{vol}}qs{m_{\mathrm{col}}}/\tau\\
 &=C_{\mathrm{vol}}q\tau_0ds+C_{\mathrm{vol}}q{m_{\mathrm{col}}}^2/\tau_0
 \le\tfrac12 aqds.
 \end{aligned}
\] Markov proves \eqref{l:eq:G.9} with \(c_{\mathrm{vol}}=aq_0/2\). \end{proof}

\paragraph{Singular-value estimates}\label{l:singular-value-estimates}

The following lemma separates two uses of singular values. Part (i)
does not require the aspect ratio to be small and supplies the energy bound
for the coarse comparison in Subsection \ref{l:sec:real_comparison_analysis}.
Part (ii) controls the small-singular-value count needed for the capped
coordinate-weight estimate, and requires \(m_{\mathrm{col}}/h\) to be small.

\begin{lemma}[Singular-value energy and small-value count]\label{l:thm:SVAL}

Let \(\mathsf M\in\mathbb R^{h\times {m_{\mathrm{col}}}}\),
\(1\le {m_{\mathrm{col}}}\le h\), have independent columns with
\(\mathsf M_{:,j}\sim\nu_j^{\otimes h}\), where the laws
\(\nu_j\in\{\mu,N(0,1)\}\) are fixed in advance. Write
\(s_1(\mathsf M)\ge\cdots\ge s_{m_{\mathrm{col}}}(\mathsf M)\ge0\) for
its singular values. There are positive constants
\(t_{\mathrm{en}},C_{\mathrm{en}},a_{\mathrm{sing}},
 c_{\mathrm{sing}},\alpha_{\mathrm{sing}}\), depending only on \(\mu\),
with the following properties.

\Needspace{6\baselineskip}
\textbf{(i)} For every integer \(1\le b\le {m_{\mathrm{col}}}\), \begin{equation}
\mathbb E\exp\left(t_{\mathrm{en}}\sum_{j=1}^bs_j(\mathsf M)^2\right)
 \le e^{C_{\mathrm{en}}(h+{m_{\mathrm{col}}})b}.
 \label{l:eq:G.4}
\end{equation}
\Needspace{7\baselineskip}
\noindent\textbf{(ii)} If \({m_{\mathrm{col}}}\le\alpha_{\mathrm{sing}}h\),
then the count \[
 N_{\mathrm{small}}:=\#\{1\le j\le {m_{\mathrm{col}}}:
               s_j(\mathsf M)<a_{\mathrm{sing}}\sqrt h\}
\] satisfies \begin{equation}
\mathbb P\{N_{\mathrm{small}}\ge b\}\le3e^{-c_{\mathrm{sing}}hb}
 \quad(1\le b\le {m_{\mathrm{col}}}).
 \label{l:eq:G.5}
\end{equation} \end{lemma}

\begin{proof}[\proofname\ of Lemma \ref{l:thm:SVAL}]\textbf{Energy estimate.} We use a
net argument for the mixed-column model; see also
\cite[Chapters 2 and 4]{V2018}. Choose \(v_\mu\ge1\) so that both
scalar laws obey
\(\mathbb E e^{t\xi}\le e^{v_\mu t^2/2}\). Gaussian linearization gives,
for a fixed orthonormal \({m_{\mathrm{col}}}\times b\) frame \(U\), \[
 \mathbb E e^{s\|\mathsf M U\|_{\mathrm{HS}}^2}
 \le(1-2v_\mu s)^{-hb/2}\quad(0<s<(2v_\mu)^{-1}).
\] The Stiefel manifold has an operator-norm \(\eta\)-net \(\mathcal N\)
of size at most \((1+2/\eta)^{{m_{\mathrm{col}}}b}\), by disjoint-ball
volume comparison in the ambient normed space. For any
\({m_{\mathrm{col}}}\times b\) matrix \(\Delta\),
\(\|\mathsf M\Delta\|_{\mathrm{HS}}\le
(\sum_{j=1}^bs_j(\mathsf M)^2)^{1/2}\|\Delta\|_{\mathrm{op}}\).
Approximating a maximizing frame by a \(1/4\)-net therefore gives \[
 \sum_{j=1}^bs_j(\mathsf M)^2\le(16/9)\max_{U\in\mathcal N}
 \|\mathsf M U\|_{\mathrm{HS}}^2.
\] Summing the fixed-frame moment bound over this net proves \eqref{l:eq:G.4}.
Consequently, for every fixed \(t_{\mathrm{tail}}>0\), some
\(K_{t_{\mathrm{tail}}}(\mu)\) satisfies
\(\mathbb P\{\sum_{j=1}^bs_j(\mathsf M)^2>K_{t_{\mathrm{tail}}}hb\}
\le e^{-t_{\mathrm{tail}}hb}\).

\Needspace{6\baselineskip}
\textbf{Small-singular-value count.} For a fixed frame \(U\), \eqref{l:eq:G.1}
applied independently to the \(h\) rows bounds the density of
\(\mathsf M U\) by \(C_\mu^{hb}\). Hence \[
 \mathbb P\{\|\mathsf M U\|_{\mathrm{HS}}\le a\sqrt{hb}\}
 \le(C'_\mu a)^{hb}\le e^{-2hb}
\] for a sufficiently small \(a>0\). Set \(a_{\mathrm{sing}}=a/2\) and
\(t_{\mathrm{tail}}=1\). On the energy event
\(\sum_{j=1}^bs_j(\mathsf M)^2\le K_1hb\), any frame of \(b\) small
singular directions has a net approximant with image norm at most
\(a\sqrt{hb}\), provided \(\eta\sqrt{K_1}\le a/2\). Fix such a net and
choose \(\alpha_{\mathrm{sing}}\) so that
\({m_{\mathrm{col}}}\log(1+2/\eta)\le h\). Its size is then at most
\(e^{hb}\). The union bound over the net and the energy complement prove
\eqref{l:eq:G.5}. \end{proof}

\paragraph{Proof of the capped coordinate-weight
bound}\label{l:proof-of-the-capped-coordinate-weight-bound}

We separate small singular directions from the rest. Their contribution
to the total coordinate mass is at most their number; the remaining
directions are controlled by independent row energies. This is why the
preceding lemma estimates a count, rather than only the least singular
value.

\begin{proof}[\proofname\ of Lemma \ref{l:thm:GE}]\textbf{Reduce to two exponential
moments.} Write
\(s_1(\mathsf M)\ge\cdots\ge s_{m_{\mathrm{col}}}(\mathsf M)\ge0\) for
the singular values, with orthonormal left and right singular vectors
\(\mathbf u_j\in\mathbb R^h\), \(\mathbf v_j\in\mathbb R^{m_{\mathrm{col}}}\),
satisfying \(\mathsf M\mathbf v_j=s_j\mathbf u_j\). Take
\(a_{\mathrm{sing}},c_{\mathrm{sing}},\alpha_{\mathrm{sing}}\) from
Lemma \ref{l:thm:SVAL}. Write \[
 N_{\mathrm{small}}=\#\{j:s_j(\mathsf M)<a_{\mathrm{sing}}\sqrt h\},
 \qquad q_i=\sum_{j=1}^{m_{\mathrm{col}}}\mathsf M_{ij}^2.
\] Define the contribution of the small positive singular directions by
\[
 b_i:=\sum_{\substack{1\le j\le {m_{\mathrm{col}}}\\0<s_j<a_{\mathrm{sing}}\sqrt h}}
                  |(\mathbf u_j)_i|^2,\qquad 1\le i\le h.
\] This is the \(i\)-th diagonal of the orthogonal projection onto their
span. Since every \(\mathbf u_j\) has norm one, \[
 \begin{aligned}
 \sum_{i=1}^h b_i
 &=\sum_{\substack{1\le j\le {m_{\mathrm{col}}}\\0<s_j<a_{\mathrm{sing}}\sqrt h}}
                         \sum_{i=1}^h|(\mathbf u_j)_i|^2\\
 &=\#\{1\le j\le {m_{\mathrm{col}}}:0<s_j<a_{\mathrm{sing}}\sqrt h\}
 \le N_{\mathrm{small}}.
 \end{aligned}
\] The inequality allows for zero singular values, which are counted by
\(N_{\mathrm{small}}\) but do not contribute to the range.

The left singular vectors with \(s_j>0\) form an orthonormal basis of
\(\operatorname{Ran}\mathsf M\). Consequently, \[
 P_{\operatorname{Ran}\mathsf M}
   =\sum_{\substack{1\le j\le {m_{\mathrm{col}}}\\s_j>0}}\mathbf u_j\mathbf u_j^T,
 \qquad
 \ell_i=b_i+\sum_{\substack{1\le j\le {m_{\mathrm{col}}}\\s_j\ge a_{\mathrm{sing}}\sqrt h}}
                               |(\mathbf u_j)_i|^2.
\] For these remaining directions, substitute
\((\mathbf u_j)_i=(\mathsf M\mathbf v_j)_i/s_j\) and use the lower bound on \(s_j\): \[
 \begin{aligned}
 \ell_i-b_i
 &=\sum_{\substack{1\le j\le {m_{\mathrm{col}}}\\s_j\ge a_{\mathrm{sing}}\sqrt h}}
                   \frac{|(\mathsf M\mathbf v_j)_i|^2}{s_j^2}\\
 &\le\frac{1}{a_{\mathrm{sing}}^2h}
       \sum_{\substack{1\le j\le {m_{\mathrm{col}}}\\s_j\ge a_{\mathrm{sing}}\sqrt h}}
                   |(\mathsf M\mathbf v_j)_i|^2\\
 &\le\frac{1}{a_{\mathrm{sing}}^2h}
                   \sum_{q=1}^{m_{\mathrm{col}}}|\mathsf M_{iq}|^2
 =\frac{q_i}{a_{\mathrm{sing}}^2h}.
 \end{aligned}
\] The last inequality is Bessel's inequality applied to the \(i\)-th
row: \((\mathsf M\mathbf v_j)_i\) is its component along \(\mathbf v_j\), and the
squared components along any subset of these orthonormal directions sum
to at most the squared row norm \(q_i\). Using
\(\min(x+y,\delta)^{3/2}\le
C\sqrt\delta\,x+C\min(y,\delta)^{3/2}\), we obtain \begin{equation}
\sum_{i=1}^h\min(\ell_i,\delta)^{3/2}
 \le C\sqrt\delta N_{\mathrm{small}}
       +C\sum_{i=1}^h\min(q_i/(a_{\mathrm{sing}}^2h),\delta)^{3/2}.
 \label{l:eq:G.7}
\end{equation} Set \(D_\theta=2C\theta\) and define \[
 g_h(y)=D_\theta h\min\left\{\frac{y}{a_{\mathrm{sing}}^2h},\delta\right\}^{3/2}.
\] Cauchy--Schwarz, followed by independence of the rows, now gives \[
 \begin{aligned}
 \log\mathbb E e^{\theta h\sum_{i=1}^h\min(\ell_i,\delta)^{3/2}}
 &\le\frac12\log\mathbb E e^{D_\theta h\sqrt\delta N_{\mathrm{small}}}
       +\frac12\sum_{i=1}^h\log\mathbb E e^{g_h(q_i)}.
 \end{aligned}
\] We control the singular-count term first, then the row terms.

\Needspace{6\baselineskip}
\textbf{The singular-count term.} Choose \(\delta\) small enough that
\(D_\theta\sqrt\delta\le c_{\mathrm{sing}}/2\). Summing \eqref{l:eq:G.5} over
integer levels gives \[
 \mathbb E e^{D_\theta h\sqrt\delta N_{\mathrm{small}}}
 \le1+3\sum_{b=1}^{m_{\mathrm{col}}} e^{-c_{\mathrm{sing}}hb/2}
 \le1+C_\mu e^{-c_\mu h}.
\] Thus the first logarithmic moment is \(O_\mu(e^{-c_\mu h})\).

\Needspace{6\baselineskip}
\textbf{The row terms and conclusion.} A common square-exponential
moment gives \(\lambda_\mu,K_{\mathrm{row}}>0\) such that \[
 \mathbb P\{q_i>y\}\le e^{-\lambda_\mu y/2}
 \qquad(y\ge K_{\mathrm{row}}{m_{\mathrm{col}}}).
\] Decreasing \(\delta\) if necessary ensures
\(g_h(y)\le\lambda_\mu y/4\); also
\(0\le g_h'(y)\le C_{\mu,\theta}\sqrt{y/h}\) almost everywhere. With
\(r_0=K_{\mathrm{row}}{m_{\mathrm{col}}}\), tail integration gives \[
 \begin{aligned}
 \mathbb E e^{g_h(q_i)}
 &\le e^{g_h(r_0)}+
       \int_{r_0}^{\infty}g_h'(y)e^{g_h(y)}\mathbb P\{q_i>y\}\,dy\\
 &\le \exp\{C_{\mu,\theta}{m_{\mathrm{col}}}^{3/2}/\sqrt h\}
       +\frac{C_{\mu,\theta}}{\sqrt h}
          \int_0^\infty\sqrt y\,e^{-\lambda_\mu y/4}\,dy.
 \end{aligned}
\] The integral is finite. Since \({m_{\mathrm{col}}}\ge1\) and
\(\log(e^x+u)\le x+u\) for \(x,u\ge0\), this implies
\(\log\mathbb E e^{g_h(q_i)}\le C_{\mu,\theta}{m_{\mathrm{col}}}^{3/2}/\sqrt h\).
Substitution into the initial moment decomposition gives a total of
\(C_{\mu,\theta}\sqrt h\,{m_{\mathrm{col}}}^{3/2}+O_\mu(e^{-c_\mu h})\).
Since \(\sqrt h\,{m_{\mathrm{col}}}^{3/2}\ge1\), the remainder is
absorbed. Taking \(\alpha_\theta\le\alpha_{\mathrm{sing}}\) and
\(\delta_\theta\) small enough for both cap conditions proves \eqref{l:eq:G.6}.
\end{proof}
\subsection{Proof of the observation
comparison}\label{l:sec:real_comparison_analysis}

We return to the general observation model of
Subsection~\ref{l:the-observation-comparison-theorem} and prove
Theorem~\ref{l:thm:AC}. Fix the auxiliary input \(c\), and write
\(L_c=dP_c/dG_0\) for the response likelihood and \(U_c\) for the frame
determined by the responses. The coarse comparison follows from the
exact likelihood, the projection-density bound \eqref{l:eq:G.1}, and
the energy estimate \eqref{l:eq:G.4}.

For the weighted comparison, the tilted density bound and a coupling
in the query order yield the entropy estimate of
Lemma~\ref{l:thm:ENT}: under its hypotheses,
\[
 \mathbb E_M\log L_c
 \le C_{\mathrm{ent}}d\mathbb E_M\Theta(U_c)
       +(1-\lambda)H(M\mid G_0),\qquad 0<\lambda<1.
\] Here \(M\) is an auxiliary law on response records and
\(H(M\mid G_0)=\int\log(dM/dG_0)\,dM\) is relative entropy.
For the event comparison, \(M\) is the original response law conditioned
on the event, with likelihood truncation when needed. The coefficient
\(1-\lambda<1\) leaves a negative entropy term and converts the average
coordinate-weight cost into a Gaussian exponential integral.
Subsection~\ref{l:sec:real_comparison_proof} derives this consequence;
Subsection~\ref{l:sec:entropy_gain} proves the entropy estimate from the
density bound and the coupling. Neither argument uses the stopping
estimate or the original-law link estimate.

\subsubsection{The exact observation
likelihood}\label{l:the-exact-observation-likelihood}

Fix the auxiliary input \(c\). The response laws are \(P_c,G_0\); the
symbols \(P,G\) continue to denote the joint laws above. Write
\(U_c(\mathbf z)=U(c,\mathbf z)\) for the query frame generated along
the response record \(\mathbf z\). Every such history is realizable: the
deterministic matrix \(X_{\mathbf z}=\mathbf zU_c(\mathbf z)^T\)
produces it, since \(X_{\mathbf z}\mathbf u_a(c,\mathbf z_{<a})=\mathbf z_a\) at
each step. The observation imposes \(XU_c(\mathbf z)=\mathbf z\);
\(X_{\mathbf z}\) is just the particular realization with all unobserved
coordinates set to zero. For a deterministic orthonormal frame
\(U\in\mathbb R^{h\times k}\) and \(y\in\mathbb R^k\), define the
density by its fiber integral \begin{equation}
p_U(y)=\int_{(\operatorname{Ran}U)^\perp}
                  \prod_{j=1}^h f((Uy+w)_j)\,dw.
 \label{l:eq:A.4}
\end{equation} A Borel orthogonal completion supplies measurable fiber coordinates;
the integral is independent of the completion. We use this specific
density representative throughout.

\begin{lemma}[Exact observation likelihood]\label{l:thm:LIKE}

For a response matrix \(\mathbf z\in\mathbb R^{d\times k}\), write
\(\mathbf z^{(i)}=(z_{i1},\ldots,z_{ik})^T\in\mathbb R^k\) for its
\(i\)-th row as a column vector. The response likelihood
\(L_c=dP_c/dG_0\) is \begin{equation}
L_c(\mathbf z)=\prod_{i=1}^d\frac{p_{U_c(\mathbf z)}(\mathbf z^{(i)})}{\phi_k(\mathbf z^{(i)})},
 \qquad\phi_k(y)=(2\pi)^{-k/2}e^{-\|y\|^2/2}.
 \label{l:eq:A.5}
\end{equation} \end{lemma} \begin{proof} At step \(a\), let
\(\mathbf r\in\mathbb R^{d\times(a-1)}\) be the preceding response
matrix and \(\mathbf V\in\mathbb R^{d\times(h-a+1)}\) the unrevealed
coordinates of the source rows. Using a fixed Borel completion, the next
query determines an orthogonal matrix
\(O_a(\mathbf r)\in\mathbb R^{(h-a+1)\times(h-a+1)}\) whose first row is
that query's coordinate vector in the remaining subspace. The change of
variables is \[
 (\mathbf r,\mathbf V)\longmapsto
       (\mathbf r,\mathbf V O_a(\mathbf r)^T).
\] The first column of the transformed array is the next response. For
each fixed history \(\mathbf r\), the map rotates each row of
\(\mathbf V\) orthogonally. Tonelli therefore shows that it preserves
the full Lebesgue integral, even for Borel update rules. After \(k\)
queries, \(U_c(\mathbf z)\) has \(k\) orthonormal columns. Fixing the
responses \(\mathbf z\) leaves each source row \(\mathbf x_i=X_{i,:}^T\)
in the \((h-k)\)-dimensional affine fiber \[
 \mathbf x_i=U_c(\mathbf z)\mathbf z^{(i)}+\mathbf w_i,\qquad
 \mathbf w_i\in(\operatorname{Ran}U_c(\mathbf z))^\perp.
\] Integrating over the unobserved coordinates \(\mathbf w_i\), rather
than fixing them to zero, gives the marginal response density and hence
\eqref{l:eq:A.4} and \eqref{l:eq:A.5}. For Gaussian sources, \(p_U=\phi_k\) pointwise, so
the responses are independent standard Gaussians. \end{proof}

\subsubsection{A pointwise bound for the fiber
density}\label{l:sec:real_fourier}

The likelihood formula evaluates the fiber density at a frame chosen
from the responses. We therefore need a bound for the specified
representative at every frame and point. We obtain it by tilting the
scalar entry densities and estimating the resulting Fourier integral.

Centering and the square-exponential moment imply, for some
\(v_\mu\ge1\), \begin{equation}
\Lambda_\mu(t)=\log\mathbb E e^{t\xi},\qquad
 0\le\Lambda_\mu(t)\le v_\mu t^2/2.
 \label{l:eq:A.6}
\end{equation}

\begin{lemma}[Pointwise density bound with coordinate-weight error]\label{l:thm:ENV}

There are \(C_{\mathrm{env}},\delta_{\mathrm{env}}>0\), depending only
on \(\mu\), such that, for every \(U\in\mathbb R^{h\times k}\) with
\(U^TU=\mathrm{Id}_k\) and coordinate weights
\(\ell_j=\sum_{a=1}^k U_{ja}^2\le\delta_{\mathrm{env}}\) for \(1\le j\le h\),
every \(y\in\mathbb R^k\), and
every \(\mathbf t=(t_1,\ldots,t_h)^T\in\operatorname{Ran}U\), \begin{equation}
\log p_U(y)\le-\tfrac k2\log(2\pi)-\langle\mathbf t,Uy\rangle
 +\sum_{j=1}^h\Lambda_\mu(t_j)
 +C_{\mathrm{env}}\sum_{j=1}^h(\ell_j^{3/2}+\sqrt{\ell_j}t_j^2).
 \label{l:eq:A.7}
\end{equation} \end{lemma}

We prove the bound directly for \eqref{l:eq:A.4}, with \(\log0=-\infty\). The
scalar estimate below retains the exact Gaussian normalization;
geometric Brascamp--Lieb then combines the source coordinates.

\paragraph{A scalar tilted Fourier
estimate}\label{l:a-scalar-tilted-fourier-estimate}

Define the tilted density and its characteristic function by \[
 f_t(x)=e^{tx-\Lambda_\mu(t)}f(x),\qquad
 \chi_t(u)=\int e^{iux}f_t(x)\,dx.
\]
The next lemma supplies the scalar factors for geometric
Brascamp--Lieb: we will apply it with \(c=\ell_j\) and \(t=t_j\), the
coordinate weight and the corresponding component of \(\mathbf t\).

\Needspace{12\baselineskip}

\begin{lemma}[Normalized tilted Fourier integral]\label{l:thm:FTILT}

There exist
\(\delta_{\mathrm{Four}}=\delta_{\mathrm{Four}}(\mu)\in(0,1/4]\) and
\(C_{\mathrm{Four}}=C_{\mathrm{Four}}(\mu)>0\) such that the
characteristic functions \(\chi_t\) defined above satisfy \begin{equation}
c\log\left[(2\pi)^{-1/2}
        \int_{\mathbb R}|\chi_t(\sqrt c\,z)|^{1/c}\,dz\right]
 \le C_{\mathrm{Four}}(c^{3/2}+\sqrt c\,t^2),
 \quad 0<c\le\delta_{\mathrm{Four}},\quad t\in\mathbb R.
 \label{l:eq:F.1}
\end{equation} \end{lemma}

\begin{proof}

Plancherel and \(f^2\le K_f f\) give \begin{equation}
\|\chi_t\|_2^2
 \le2\pi K_f e^{\Lambda_\mu(2t)-2\Lambda_\mu(t)}
 \le C_\mu e^{C_\mu t^2}.
 \label{l:eq:F.2}
\end{equation} On a sufficiently small fixed tilt interval \(|t|\le t_0\),
\(v(t)=\Lambda_\mu''(t)=1+O(|t|)\ge3/4\), and the tilted fourth moments
are uniformly bounded. These statements follow by differentiating under
the integral, dominated by the square-exponential moment on each compact
tilt interval. For independent \(Y,Y'\) of density \(f_t\), the
elementary cosine bound gives \[
 |\chi_t(u)|^2=\mathbb E\cos(u(Y-Y'))
 \le\exp\{-v(t)u^2+C_\mu u^4\}.
\] Choose a small fixed \(u_0>0\). On \(|z|\le u_0/\sqrt c\), comparison
with \(e^{-v(t)z^2/2}\) shows that the positive excess of
\(|\chi_t(\sqrt c z)|^{1/c}\) over this Gaussian is bounded by
\(C_\mu c z^4e^{-b_\mu z^2}\). To see the integrable decay, choose
\(u_0\) so the quartic remainder is at most a fixed fraction of the
quadratic term, then use \(e^v-1\le ve^v\). Its integral is \(O(c)\).

Absolute continuity implies \(|\chi_0(u)|<1\) for \(u\ne0\): equality
would force \(e^{iu\xi}\) to be constant almost surely, placing the law
on a countable set. The Riemann--Lebesgue lemma and compactness give
\(\sup_{|u|\ge u_0}|\chi_0(u)|<1\). Also
\(\|\chi_t-\chi_0\|_\infty\le\|f_t-f\|_1\to0\). Decreasing \(t_0\)
yields a common bound \(\rho_\mu<1\) on this region. By \eqref{l:eq:F.2}, the
remaining integral is at most \[
 c^{-1/2}\rho_\mu^{1/c-2}\|\chi_t\|_2^2=O(c).
\] The last assertion follows, for instance, by
\(e^{\lambda/c}\ge\lambda^2/(2c^2)\), \(\lambda=-\log\rho_\mu>0\). We
have proved \[
 (2\pi)^{-1/2}\int|\chi_t(\sqrt c z)|^{1/c}\,dz
 \le v(t)^{-1/2}+C_\mu c\le e^{C_\mu(c+|t|)}
 \qquad(|t|\le t_0).
\] Multiplying its logarithm by \(c\), and using \(c^2\le c^{3/2}\) and
\(c|t|\le(c^{3/2}+\sqrt c\,t^2)/2\), proves \eqref{l:eq:F.1} on this interval.

For \(|t|>t_0\), use \(|\chi_t|\le1\), \(1/c\ge2\) and \eqref{l:eq:F.2}: the same
normalized integral is at most \(C_\mu c^{-1/2}e^{C_\mu t^2}\). Its
logarithm times \(c\) is at most \(C'_\mu\sqrt c\,t^2\), since \(t_0>0\)
is fixed and \(c\log(1/\sqrt c)\le\sqrt c\). This completes \eqref{l:eq:F.1}.
\end{proof}

\paragraph{From the Fourier integral to every affine
fiber}\label{l:from-the-fourier-integral-to-every-affine-fiber}

Fourier inversion must recover the specified fiber integral at every
point. We establish its continuity for this purpose, then bound the
inverse integral and undo the tilt.

\begin{proof}[Proof of Lemma \ref{l:thm:ENV}]Fix \(U\in\mathbb R^{h\times k}\) with
\(U^TU=\mathrm{Id}_k\), \(\ell_j\le\delta_{\mathrm{Four}}\) for \(1\le j\le h\), and
\(\mathbf t\in\operatorname{Ran}U\). The tilted fiber density is \[
 p_{U,\mathbf t}(y)=\int_{\ker U^T}
             \prod_{j=1}^h f_{t_j}((Uy+w)_j)\,dw.
\]

\Needspace{6\baselineskip}
\textbf{Continuity of the fiber density.} Write \(r_j=(1-\ell_j)^{-1}\) for \(1\le j\le h\).
Boundedness of \(f\) gives \(\int f_{t_j}^{r_j}\le K_f^{r_j-1}
 e^{\Lambda_\mu(r_jt_j)-r_j\Lambda_\mu(t_j)}<\infty\). The
kernel-coordinate Brascamp--Lieb calculation in the proof of Lemma \ref{l:thm:GP},
applied to the telescoping difference of two products, gives \[
 \begin{aligned}
 |p_{U,\mathbf t}(y+v)-p_{U,\mathbf t}(y)|
 &\le C_U\sum_{j=1}^h
 \|f_{t_j}(\cdot+(Uv)_j)-f_{t_j}\|_{r_j}
       \prod_{\substack{1\le i\le h\\i\ne j}}\|f_{t_i}\|_{r_i},\\
 C_U&=\prod_{j=1}^h(1-\ell_j)^{-(1-\ell_j)/2}.
 \end{aligned}
\] The same inequality bounds each fiber integral by
\(C_U\prod_{j=1}^h\|f_{t_j}\|_{r_j}<\infty\). Translation continuity in
\(L^{r_j}\), with \(1\le r_j\le4/3\), therefore proves continuity at
every \(y\). This auxiliary bound is used only for continuity.

\Needspace{6\baselineskip}
\textbf{The Fourier bound.} Omit zero rows, whose
characteristic-function factors are \(|\chi_{t_j}(0)|=1\); products in
the following Brascamp--Lieb calculation run over \(1\le j\le h\) with
\(\ell_j>0\). Normalize each remaining row by setting
\(B_j=U_{j,:}/\sqrt{\ell_j}:\mathbb R^k\to\mathbb R\). Then \[
 B_jB_j^T=1,\qquad
 \sum_{\substack{1\le j\le h\\\ell_j>0}}\ell_j B_j^TB_j=\sum_{j=1}^hU_{j,:}^TU_{j,:}
 =U^TU=\mathrm{Id}_k,\qquad \sum_{j=1}^h\ell_j=k.
\] Thus the geometric Brascamp--Lieb weights are precisely \(\ell_j\).
For each \(1\le j\le h\) with \(\ell_j>0\), define \[
 H_j(s)=|\chi_{t_j}(\sqrt{\ell_j}\,s)|^{1/\ell_j},
 \qquad I_j=\int_{\mathbb R}H_j(s)\,ds.
\] These integrals are finite by \eqref{l:eq:F.2}, since \(1/\ell_j\ge2\). The
identity \(U_{j,:}z=\sqrt{\ell_j}B_jz\) gives \[
 |\chi_{t_j}(U_{j,:}z)|=H_j(B_jz)^{\ell_j}.
\] Consequently geometric Brascamp--Lieb
\cite[Proposition 2.8]{BCCT2008} yields \[
 \int_{\mathbb R^k}\prod_{j=1}^h|\chi_{t_j}(U_{j,:}z)|\,dz
 =\int_{\mathbb R^k}\prod_{\substack{1\le j\le h\\\ell_j>0}}H_j(B_jz)^{\ell_j}\,dz
 \le\prod_{\substack{1\le j\le h\\\ell_j>0}}\left(\int_{\mathbb R}H_j(s)\,ds\right)^{\ell_j}
 =\prod_{\substack{1\le j\le h\\\ell_j>0}}I_j^{\ell_j}.
\] For each \(1\le j\le h\) with \(\ell_j>0\), apply Lemma \ref{l:thm:FTILT} with \(c=\ell_j\) and \(t=t_j\): \[
 \ell_j\log\big[(2\pi)^{-1/2}I_j\big]
 \le C_{\mathrm{Four}}(\ell_j^{3/2}+\sqrt{\ell_j}t_j^2).
\] Sum these logarithmic bounds, use \(\sum_{j=1}^h\ell_j=k\), and
exponentiate to obtain \begin{equation}
\int_{\mathbb R^k}\prod_{j=1}^h|\chi_{t_j}(U_{j,:}z)|\,dz
 \le(2\pi)^{k/2}
       \exp\left\{C_{\mathrm{Four}}\sum_{j=1}^h(\ell_j^{3/2}+\sqrt{\ell_j}t_j^2)\right\}.
 \label{l:eq:F.3}
\end{equation} Independence gives the product characteristic function below.
Its integrability in \eqref{l:eq:F.3} permits Fourier inversion
\cite[Chapter 8]{F1999}. The inverse is a continuous density; the
continuity just proved makes it agree with our specified fiber integral
at every point. Thus, for every \(y\in\mathbb R^k\),
\[
 p_{U,\mathbf t}(y)=\frac1{(2\pi)^k}
 \int_{\mathbb R^k}e^{-i\langle\boldsymbol\omega,y\rangle}
       \prod_{j=1}^h\chi_{t_j}(U_{j,:}\boldsymbol\omega)
       \,d\boldsymbol\omega.
\]

\Needspace{6\baselineskip}
\textbf{Undoing the tilt.} The definition of the scalar tilt gives
\(f(x)=e^{-t_jx+\Lambda_\mu(t_j)}f_{t_j}(x)\). Multiplying over
coordinates, \[
 \prod_{j=1}^h f((Uy+w)_j)
 =e^{-\langle\mathbf t,Uy+w\rangle+\sum_{j=1}^h\Lambda_\mu(t_j)}
       \prod_{j=1}^h f_{t_j}((Uy+w)_j).
\] Since \(w\in\ker U^T=(\operatorname{Ran}U)^\perp\),
\(\langle\mathbf t,w\rangle=0\). Thus the exponential factor is constant
along the fiber, and integrating in \(w\) yields \[
 p_U(y)=e^{-\langle\mathbf t,Uy\rangle+\sum_{j=1}^h\Lambda_\mu(t_j)}p_{U,\mathbf t}(y).
\] Taking absolute values in the inverse formula and using \eqref{l:eq:F.3} proves
\eqref{l:eq:A.7}, with \(C_{\mathrm{env}}=C_{\mathrm{Four}}\) and
\(\delta_{\mathrm{env}}=\delta_{\mathrm{Four}}\), simultaneously for all
frames and points. \end{proof}

\subsubsection{Gaussian coupling in the query
order}\label{l:sec:real_transport}

We couple a response law \(M\) to fresh Gaussian columns, evaluating
all queries on the target response history. Entropy bounds the coupling
error; the following moment estimate controls Gaussian sums in this frame.

\Needspace{12\baselineskip}

\paragraph{A coupling adapted to the query
order}\label{l:a-coupling-adapted-to-the-query-order}

\begin{lemma}[Sequential Gaussian coupling]\label{l:thm:SEQCOUP}

Let \(G_0\) be the product standard Gaussian law on
\(\mathbb R^{d\times k}\), and let \(M\) be a probability law with
\(H(M\mid G_0)<\infty\). There is a coupling of a response matrix
\(\mathbf Z\sim M\) and independent vectors
\(\mathbf g_1,\ldots,\mathbf g_k\sim N(0,\mathrm{Id}_d)\) such that
\({\mathbf Z}_{:,\le j}\) is a Borel function of
\(\mathbf g_1,\ldots,\mathbf g_j\) for every \(j\), and \begin{equation}
\mathbb E\sum_{j=1}^k\|{\mathbf Z}_{:,j}-\mathbf g_j\|^2\le2H(M\mid G_0).
 \label{l:eq:C.1}
\end{equation} In particular, \(\mathbf g_j\) is independent of all preceding target
and Gaussian columns. Any query determined by \({\mathbf Z}_{:,<j}\) is
known before \(\mathbf g_j\) is revealed. \end{lemma}

\begin{proof}

Order the \(dk\) scalar coordinates by columns, and denote them by
\(\widehat Z_\ell,\widehat g_\ell\). Let \(\gamma=N(0,1)\), let \(\Phi\)
be its distribution function, and let \(K_\ell(z_{<\ell},\cdot)\) be
conditional laws under \(M\). Using independent
\(\widehat g_\ell\sim\gamma\), define recursively \[
 \widehat Z_\ell=
 F^{-1}_{K_\ell(\widehat Z_{<\ell},\cdot)}
       (\Phi(\widehat g_\ell)),\qquad 1\le\ell\le dk,
\] Here, for a probability law \(K\) on \(\mathbb R\),
\(F_K(x)=K((-\infty,x])\) and
\(F_K^{-1}(u)=\inf\{x:F_K(x)\ge u\}\), \(0<u<1\), are its
distribution function and generalized inverse.
The conditional quantile is jointly Borel, using rational thresholds in
its definition. The recursion has the prescribed conditional laws, hence
joint law \(M\). The one-dimensional Gaussian transport--entropy
inequality \cite{T1996} and the entropy chain rule
\cite[Section 1.4]{DE1997} give \[
 \mathbb E\sum_{\ell=1}^{dk}|\widehat Z_\ell-\widehat g_\ell|^2
 \le2\sum_{\ell=1}^{dk}\mathbb E
 H(K_\ell(\widehat Z_{<\ell},\cdot)\mid\gamma)
 =2H(M\mid G_0).
\] The conditional construction is also recorded in
\cite[Theorem 2.9 and Proposition 5.9]{BBLZ2017}. Finite total entropy
makes the conditional costs finite almost everywhere; null histories may
be assigned arbitrary probability kernels.

Regrouping coordinates into columns proves \eqref{l:eq:C.1}. The sequential
construction gives \[
 \sigma(\mathbf g_1,\ldots,\mathbf g_j,{\mathbf Z}_{:,\le j})=\sigma(\mathbf g_1,\ldots,\mathbf g_j),
\] so the next Gaussian column is independent of the entire past.
\end{proof}

\Needspace{12\baselineskip}

\paragraph{Third moments of predictable Gaussian
sums}\label{l:third-moments-of-predictable-gaussian-sums}

The next estimate allows coefficients chosen from earlier responses;
its random sum of squared coefficients will be a coordinate weight.

\begin{lemma}[Gaussian third moments and a random variance sum]\label{l:thm:GAUSS3}

Let \((\mathcal F_j)_{j=0}^k\) be a filtration, and let \(g_j\) be
\(\mathcal F_j\)-measurable scalar standard Gaussians
independent of \(\mathcal F_{j-1}\), and let bounded real \(a_j\) be
\(\mathcal F_{j-1}\)-measurable. For \(S_k=\sum_{j=1}^k a_jg_j\), \begin{equation}
\mathbb E|S_k|^3
 \le27\,\mathbb E\left(\sum_{j=1}^k a_j^2\right)^{3/2}.
 \label{l:eq:C.2}
\end{equation} Moreover, \[
 \mathbb E|S_k|^3+
 \mathbb E\left[\left(\sum_{j=1}^k a_j^2\right)^{1/2}|S_k|^2\right]
 \le46\,\mathbb E\left(\sum_{j=1}^k a_j^2\right)^{3/2}.
\] No independence between \(S_k\) and \(\sum_{j=1}^k a_j^2\) is
assumed. \end{lemma}

\begin{proof}

Write \(S_j=\sum_{a=1}^ja_ag_a\), \(V=\sum_{a=1}^k a_a^2\), and
\(S_* =\max_{0\le j\le k}|S_j|\). The elementary Taylor bound \[
 |x+y|^3\le |x|^3+3x|x|y+3|x|y^2+|y|^3
\] follows by integrating the second derivative \(6|x|\). Apply it to
each increment and take conditional expectations. The linear term
vanishes, \(\mathbb E|g_a|^3<2\), and
\(\sum_{a=1}^k|a_a|^3\le V^{3/2}\). Hence, with \(X=\mathbb E|S_k|^3\),
\(Y=\mathbb E V^{3/2}\), \[
 X\le3\mathbb E(S_*V)+2Y
 \le\frac92X^{1/3}Y^{2/3}+2Y.
\] The second inequality is H\"older together with Doob's \(L^3\) maximal
inequality \(\|S_*\|_3\le(3/2)\|S_k\|_3\). Young's inequality gives \[
 X\le\frac X3+
       \left[\frac23(9/2)^{3/2}+2\right]Y,
 \qquad X\le[(9/2)^{3/2}+3]Y\le13Y.
\] All moments are finite because the coefficients are bounded and the
sum is finite. This proves \eqref{l:eq:C.2}, with room in its stated constant.
Finally, \[
 \sqrt V|S_k|^2\le\tfrac13V^{3/2}+\tfrac23|S_k|^3
\] gives the second assertion; even the weaker constant \(27\) from
\eqref{l:eq:C.2} gives \(27+1/3+18<46\). \end{proof}

\subsubsection{Completion of the comparison
proof}\label{l:sec:real_comparison_proof}

We continue at fixed auxiliary input \(c\), and average over its common
law only after obtaining the conditional event bounds.

\paragraph{The coarse comparison}\label{l:the-coarse-comparison}

\begin{proof}[Proof of the coarse assertion of Theorem \ref{l:thm:AC}]\leavevmode\par
The projection
bound \eqref{l:eq:G.1} and \eqref{l:eq:A.5} give
\(L_c(Y)\le e^{C_\mu dk+\|Y\|_{\mathrm{HS}}^2/2}\). Under the original
source law, pointwise
\(\|XU\|_{\mathrm{HS}}^2\le\sum_{a=1}^k s_a(X)^2\). The energy estimate
\eqref{l:eq:G.4} of Lemma \ref{l:thm:SVAL}, applied to \(X\) when \(h\le d\), therefore gives,
for a sufficiently small fixed \(t>0\), \[
 \mathbb E_{G_0} L_c^{1+t}=\mathbb E_{P_c} L_c^t\le e^{C'_\mu dk}.
\] The equality first holds for nonnegative extended integrals; the
bound proves finiteness. H\"older yields \eqref{l:eq:A.2}. The moment estimate is
uniform in the fixed auxiliary input, so it also holds after averaging
over \(\mathsf C\). \end{proof}

\paragraph{The weighted comparison}\label{l:the-weighted-comparison}

Set \(\delta_\mu=\min\{1/4,\delta_{\mathrm{env}}\}\). We state the
entropy estimate and derive the weighted event bound from it. Its proof
in Subsection~\ref{l:sec:entropy_gain} combines the density and coupling
estimates established above.

\begin{lemma}[Entropy gain for the observation likelihood]\label{l:thm:ENT}

There are \(C_{\mathrm{ent}}>0\) and \(\lambda\in(0,1)\), depending only
on \(\mu\), such that the following holds for every fixed auxiliary
input \(c\). Let \(M\ll G_0\) be a probability law on response matrices
\(\mathbb R^{d\times k}\), with \(H(M\mid G_0)<\infty\) and
\(\log L_c\in L^1(M)\), satisfying \[
 M\{\mathbf z:L_c(\mathbf z)>0,
       \max_{1\le j\le h}\ell_j(U_c(\mathbf z))\le\delta_\mu\}=1.
\] Then \begin{equation}
\mathbb E_M\log L_c\le C_{\mathrm{ent}}d\mathbb E_M\Theta(U_c)
                         +(1-\lambda)H(M\mid G_0).
 \label{l:eq:A.13}
\end{equation} Here
\(\mathbb E_M\Theta(U_c)=\int\Theta(U_c(\mathbf z))\,M(d\mathbf z)\).
The functions \(L_c,U_c\) are the same observation functions for every
choice of \(M\). \end{lemma}

\begin{proof}[\proofname\ of the weighted assertion of Theorem \ref{l:thm:AC}]Let
\(F\) satisfy the coordinate-weight cap in the theorem, and use its
section \(F_c\). First suppose \(P_c(F_c)>0\) and \(\log L_c\) is
bounded on \(F_c\). Now choose \(M=P_c(\cdot\mid F_c)\). Its density and
entropy are \[
 \frac{dM}{dG_0}=\frac{\mathbf1_{F_c}L_c}{P_c(F_c)},\qquad
 H(M\mid G_0)=\mathbb E_M\log L_c-\log P_c(F_c).
\] In particular, \(M\) meets the hypotheses of Lemma \ref{l:thm:ENT}. Combining
this identity with \eqref{l:eq:A.13} gives \begin{equation}
\log P_c(F_c)\le C_{\mathrm{ent}}d\mathbb E_M\Theta(U_c)
                                  -\lambda H(M\mid G_0).
 \label{l:eq:A.14}
\end{equation} Define the weighted Gaussian integral over \(F_c\) by \[
 I_F(c)=\int_{F_c}
       e^{C_{\mathrm{ent}}d\Theta(U_c(\mathbf z))/\lambda}\,G_0(d\mathbf z).
\] The entropy variational inequality (see \cite[Section 1.4]{DE1997})
here follows directly by applying Jensen's inequality under \(M\) to
\(e^{C_{\mathrm{ent}}d\Theta(U_c)/\lambda}/(dM/dG_0)\). Using
\(M(F_c)=1\), we obtain \[
 \frac{C_{\mathrm{ent}}d}{\lambda}\mathbb E_M\Theta(U_c)-H(M\mid G_0)
 \le\log I_F(c).
\] Thus \(P_c(F_c)\le I_F(c)^\lambda\). The integral is finite since
\(0\le\Theta\le k\).

For an arbitrary section \(F_c\), use
\(F_{c,m}=F_c\cap\{e^{-m}\le L_c\le e^m\}\). When \(P_c(F_{c,m})>0\),
its conditional law meets the preceding hypotheses; null sections
satisfy the bound automatically. Each truncated Gaussian integral is at
most \(I_F(c)\). Since \(P_c\{0<L_c<\infty\}=1\), taking increasing
limits gives \(P_c(F_c)\le I_F(c)^\lambda\) for every section satisfying
the cap.

Finally average the section probabilities over the common law of \(\mathsf C\).
Concavity gives \[
 \begin{aligned}
 P(F)&=\int_{\mathcal C}P_c(F_c)\,\mathcal L(\mathsf C)(dc)\\
 &\le\mathbb E_{\mathsf C}[I_F(\mathsf C)^\lambda]
       \le[\mathbb E_{\mathsf C} I_F(\mathsf C)]^\lambda\\
 &=\left[\mathbb E_G\mathbf1_F
       e^{C_{\mathrm{ent}}d\Theta(U)/\lambda}\right]^\lambda.
 \end{aligned}
\] This is \eqref{l:eq:A.3}, with \(\theta_\mu=\lambda\) and
\(C_{\mathrm{weight}}=C_{\mathrm{ent}}/\lambda\). \end{proof}

\subsubsection{Proof of the entropy estimate}\label{l:sec:entropy_gain}

We now prove the entropy estimate, keeping the queries evaluated on the
target response record throughout the coupling.

\begin{proof}[\proofname\ of Lemma \ref{l:thm:ENT}]Fix \(c\) and a response law \(M\)
satisfying the hypotheses of the lemma.

\Needspace{6\baselineskip}
\textbf{Coupling and moments.} Use Lemma \ref{l:thm:SEQCOUP} to couple
\(\mathbf Z\sim M\) with independent standard Gaussian columns
\(\mathbf g_1,\ldots,\mathbf g_k\), and denote the coupling law by
\(\Pi\). Write \(\mathbb E_\Pi\) for expectation on this space. Every
query is evaluated on \(\mathbf Z\): write \(U=U_c(\mathbf Z)\),
\(\mathbf u_a=\mathbf u_a(c,{\mathbf Z}_{:,<a})\), and \(\ell_j=\ell_j(U)\). The
sequential coupling makes \(\mathbf u_a\) known before \(\mathbf g_a\) is
revealed. The matrix \(\mathbf Z\) is a response record with law \(M\);
the construction requires no source matrix. Define the \(d\times h\)
matrices \begin{equation}
\begin{aligned}
 N&=\sum_{a=1}^k \mathbf g_a\mathbf u_a^T,&
 E_{\mathrm{err}}&=\sum_{a=1}^k({\mathbf Z}_{:,a}-\mathbf g_a)\mathbf u_a^T,\\
 \mathbf ZU^T&=N+E_{\mathrm{err}},&
 \mathbb E_\Pi\|E_{\mathrm{err}}\|_{\mathrm{HS}}^2
 &=\mathbb E_\Pi\sum_{a=1}^k\|{\mathbf Z}_{:,a}-\mathbf g_a\|^2\le2H(M\mid G_0).
 \end{aligned}
 \label{l:eq:A.11}
\end{equation} Thus \(N\) uses Gaussian responses in the same query frame, and
\(E_{\mathrm{err}}\) is the coupling error. The norm equality follows
pathwise from orthonormality. Each row of \(N\) and
\(E_{\mathrm{err}}\), transposed, lies in \(\operatorname{Ran}U\). Write
\(g_{ia}=(\mathbf g_a)_i\) for the scalar coordinates. For
\(1\le i\le d\) and \(1\le j\le h\), the sum
\(N_{ij}=\sum_{a=1}^k (\mathbf u_a)_jg_{ia}\) has predictable coefficients and
variance sum \(\sum_{a=1}^k (\mathbf u_a)_j^2=\ell_j\). Lemma \ref{l:thm:GAUSS3} gives \begin{equation}
\mathbb E_\Pi|N_{ij}|^3+\mathbb E_\Pi[\sqrt{\ell_j}|N_{ij}|^2]
 \le46\mathbb E_\Pi\ell_j^{3/2}.
 \label{l:eq:A.12}
\end{equation}

\Needspace{6\baselineskip}
\textbf{Choice of tilt and entropy bound.} Write \(C_\mu\) for positive
constants depending only on \(\mu\), which may increase below and are
independent of \(0<\lambda<1\). The local expansion of \(\Lambda_\mu\)
and \eqref{l:eq:A.6} give \begin{equation}
\mathcal R_\mu(t):=\Lambda_\mu(t)-t^2/2
 \le C_\mu\min\{|t|^3,t^2\}\qquad(t\in\mathbb R).
 \label{l:eq:A.8}
\end{equation} Using the cubic bound when \(|u|\le|x|\) and the quadratic bound
otherwise yields \begin{equation}
\mathcal R_\mu(x+u)\le C_\mu(|x|^3+u^2).
 \label{l:eq:A.9}
\end{equation}
We keep the Gaussian part of the tilt and include only a small fraction
of the coupling error. The decrease in the quadratic likelihood term is
linear to first order in \(\lambda\), while the added remainder cost is
quadratic.
For each \(1\le i\le d\), write \(y={\mathbf Z}_{i,:}^T\) and take \[
 \mathbf t_i=N_{i,:}^T+\lambda(E_{\mathrm{err}})_{i,:}^T
           \in\operatorname{Ran}U,
\] leaving \(\lambda\) to be chosen. The pointwise bound \eqref{l:eq:A.7} permits
this tilt on the coupling space. Subtracting
\(\log\phi_k(y)=-\tfrac k2\log(2\pi)-\|y\|^2/2\) and expanding
\(\Lambda_\mu(s)=s^2/2+\mathcal R_\mu(s)\) gives \[
 \begin{aligned}
 \log\frac{p_U(y)}{\phi_k(y)}
 \le{}&\frac{\|y\|^2}{2}-\langle\mathbf t_i,Uy\rangle
          +\frac{\|\mathbf t_i\|^2}{2}
          +\sum_{j=1}^h\mathcal R_\mu((\mathbf t_i)_j)\\
 &+C_{\mathrm{env}}\Theta(U)
          +C_{\mathrm{env}}\sum_{j=1}^h\sqrt{\ell_j}(\mathbf t_i)_j^2.
 \end{aligned}
\] The Gaussian normalizing constants cancel. Since
\(Uy=(N+E_{\mathrm{err}})_{i,:}^T\) and \(\|Uy\|=\|y\|\), the three
quadratic terms equal \[
 \frac12\|Uy-\mathbf t_i\|^2
 =\frac{(1-\lambda)^2}{2}\|(E_{\mathrm{err}})_{i,:}\|^2.
\] Apply \eqref{l:eq:A.9} with \(x=N_{ij}\), \(u=\lambda(E_{\mathrm{err}})_{ij}\).
Also use \((a+b)^2\le2a^2+2b^2\) and \(\sqrt{\ell_j}\le1\). The two
remaining sums satisfy \[
 \begin{aligned}
 \sum_{j=1}^h\mathcal R_\mu((\mathbf t_i)_j)
 &\le C_\mu\sum_{j=1}^h|N_{ij}|^3
       +C_\mu\lambda^2\|(E_{\mathrm{err}})_{i,:}\|^2,\\
 \sum_{j=1}^h\sqrt{\ell_j}(\mathbf t_i)_j^2
 &\le2\sum_{j=1}^h\sqrt{\ell_j}N_{ij}^2
       +2\lambda^2\|(E_{\mathrm{err}})_{i,:}\|^2.
 \end{aligned}
\] Substitution gives, for a fixed sufficiently large \(C_\mu\ge1\), \[
 \log\frac{p_U(y)}{\phi_k(y)}
 \le C_\mu\!\left[\Theta(U)+\sum_{j=1}^h
         (|N_{ij}|^3+\sqrt{\ell_j}N_{ij}^2)\right]
       +a_\lambda\|(E_{\mathrm{err}})_{i,:}\|^2,
\] where the error coefficient is bounded by \begin{equation}
a_\lambda:=\frac12-\lambda+C_\mu\lambda^2
       \le\frac{1-\lambda}{2}<\frac12
       \quad\text{for sufficiently small fixed }\lambda\in(0,1).
 \label{l:eq:A.10}
\end{equation} Fix such a \(\lambda\), depending only on \(\mu\).

Take logarithms of the product \eqref{l:eq:A.5} and sum the row bounds. Since the frame is common to all rows,
\(\Theta(U)\) appears exactly \(d\) times, while
\(\sum_{i=1}^d\|(E_{\mathrm{err}})_{i,:}\|^2
=\|E_{\mathrm{err}}\|_{\mathrm{HS}}^2\). Thus \[
 \begin{aligned}
 \log L_c(\mathbf Z)
 &=\sum_{i=1}^d\log\frac{p_U({\mathbf Z}_{i,:}^T)}
                              {\phi_k({\mathbf Z}_{i,:}^T)}\\
 &\le C_\mu d\Theta(U)
       +C_\mu\sum_{i=1}^d\sum_{j=1}^h
           (|N_{ij}|^3+\sqrt{\ell_j}N_{ij}^2)
       +\frac{1-\lambda}{2}\|E_{\mathrm{err}}\|_{\mathrm{HS}}^2.
 \end{aligned}
\] Take \(\mathbb E_\Pi\) and apply \eqref{l:eq:A.11}--\eqref{l:eq:A.12}. Since
\(\mathbf Z\sim M\), expectations of functions of \(\mathbf Z\) alone
are \(\mathbb E_M\). The coupling cost \(2H(M\mid G_0)\) turns the error
coefficient \((1-\lambda)/2\) into \(1-\lambda\), yielding \eqref{l:eq:A.13} with a
law-dependent \(C_{\mathrm{ent}}\). All choices are uniform in the
auxiliary input \(c\) and the law \(M\). \end{proof}

The weighted comparison also holds with any fixed exponent
\(\gamma\in(1,3/2]\) on the coordinate weights, since
\(\ell^{3/2}\le\ell^\gamma\) on \([0,1]\). We retain \(3/2\), the power
supplied by the Gaussian third moment, throughout the stopping and
geometric estimates.
\section{Complex entries with zero pseudovariance}\label{l:sec:complex}

The complex case of Theorem~\ref{l:thm:MT} follows the same route as the
real case. We give the details at the few places where the field matters.
Write \(\mu(dz)=f(z)\,d^2z\), where \(f\) is an everywhere bounded
Borel density and \(d^2z=d(\operatorname{Re}z)\,d(\operatorname{Im}z)\)
is planar Lebesgue measure. Assume, for some $K<\infty$ and $t_0>0$,
\begin{equation}
0\le f\le K,\qquad
 \mathbb E\zeta=0,\qquad \mathbb E|\zeta|^2=1,\qquad
 \mathbb E\zeta^2=0,\qquad
 \mathbb E e^{t_0|\zeta|^2}<\infty,\qquad \zeta\sim\mu.
 \label{l:eq:Z.1}
\end{equation}
For geometry and comparison, we replace transposes by adjoints and
treat each entry as a two-dimensional real block. The Gaussian Arnoldi
residual squares now have Gamma laws, but their negative moments still
give the reference-link estimate needed in
Section~\ref{l:sec:real_stopping}. The deletion and stopping argument
can then be reused. The border calculation in
Subsection~\ref{l:sec:real_border} becomes a little simpler: its Gaussian
input is a single complex scalar rather than a real scalar or a
conjugate-pair block.

Zero pseudovariance enters through the identity
\(\operatorname{Cov}_{\mathbb R}(\zeta)=\tfrac12\mathrm{Id}_2\), where
\(\operatorname{Cov}_{\mathbb R}(\zeta)\) is the covariance matrix of
\((\operatorname{Re}\zeta,\operatorname{Im}\zeta)\). This covariance
matches the circular Gaussian reference in the comparison argument.
The real and imaginary parts may still be dependent, and the law need
not be rotationally invariant.

We use \(M^*\) for the conjugate transpose and
\(\langle x,y\rangle_{\mathbb R}=\Re(x^*y)\).
For \(v>0\), the law \(\mathcal N_{\mathbb C}(0,v)\) has density
\((\pi v)^{-1}e^{-|z|^2/v}\) and second absolute moment \(v\).
Unless we explicitly use real coordinates, dimension counts and Gram
determinants below are taken over the complex field.

\subsection{Geometry and observation comparison}
\label{l:projection-geometry-for-independent-complex-entries}
\label{l:a-pointwise-density-bound-for-two-dimensional-blocks}
\label{l:adaptive-comparison-with-complex-queries}

\begin{lemma}[Complex projection and geometric estimates]\label{l:thm:CGEO}

If \(x\in\mathbb C^h\) has independent entries with planar densities
\(f_j\le K\), then, for every complex coisometry
\(L:\mathbb C^h\to\mathbb C^q\), meaning \(LL^*=\mathrm{Id}_q\),
the canonical projection density satisfies \begin{equation}
\int_{\ker L}\prod_j f_j((L^*y)_j+w_j)\,dw\le(eK)^q
 \quad\text{for every }y\in\mathbb C^q.
 \label{l:eq:Z.3}
\end{equation} For independent-entry matrices whose column laws are \(\mu\) or
\(\mathcal N_{\mathbb C}(0,1)\), the conclusions of Corollary \ref{l:thm:IGRAM} and
Lemmas \ref{l:thm:GV} and \ref{l:thm:GE} hold with \(\mathsf G=\mathsf M^*\mathsf M\), complex ranges,
and law-dependent constants. The energy bound \eqref{l:eq:G.4} holds as well.
\end{lemma}

\begin{proof}
Write \(c_j=(P_{\ker L})_{jj}\). The normalized coordinate maps on
\(\ker L\), with weights \(c_j>0\), are geometric Brascamp--Lieb data
with two-dimensional real targets. The calculation in
Lemma~\ref{l:thm:GP}, using \cite[Proposition 2.8]{BCCT2008}, becomes
\[
 \int_{\ker L}\prod_{j=1}^h f_j((L^*y)_j+w_j)\,dw
 \le K^q\prod_{\substack{1\le j\le h\\c_j>0}}c_j^{-c_j}
 \le(eK)^q,
 \qquad \sum_{j=1}^h(1-c_j)=q.
\]
The volume of a complex \(q\)-dimensional ball and tail integration give
\begin{equation}
\mathbb P\{\operatorname{dist}(x,F)\le u\sqrt q\}
 \le\min\{1,(C_Ku)^{2q}\},\qquad
 \|\operatorname{dist}(x,F)^{-2}\|_{L^p}\le C'_K/q
 \quad(1\le p\le q/4),
 \label{l:eq:Z.4}
\end{equation}
for every fixed complex subspace \(F\) of codimension \(q\).
Conditioning on successive columns gives Corollary~\ref{l:thm:IGRAM}.
For Lemma~\ref{l:thm:GV}, the required complementary-minor and
positive-Gram identities also hold for Hermitian Grams.

For Lemma~\ref{l:thm:SVAL}, let \(U\) be a complex \(b\)-frame and
\(v\) a common subgaussian constant for the two column laws. The
fixed-frame estimates and an \(\eta\)-net \(\mathcal N_\eta\) satisfy
\[
 \begin{aligned}
 \mathbb E e^{s\|\mathsf M U\|_{\mathrm{HS}}^2}
 &\le(1-2vs)^{-hb},\qquad 0<s<(2v)^{-1},\\
 |\mathcal N_\eta|&\le(1+2/\eta)^{2m_{\mathrm{col}}b},\qquad
 \mathbb P\{\|\mathsf M U\|_{\mathrm{HS}}\le a\sqrt{hb}\}
 \le(Ca)^{2hb}.
 \end{aligned}
\]
Both the net exponent and the small-ball exponent double. The
maximizing-frame argument proves \eqref{l:eq:G.4} for
\(m_{\mathrm{col}}\le h\); the small-ball net argument proves
\eqref{l:eq:G.5} after reducing the aspect ratio.
Finally the singular-value decomposition and Bessel bound
\eqref{l:eq:G.7} are Hermitian identities. With
\(q_i=\sum_{j=1}^{m_{\mathrm{col}}}|\mathsf M_{ij}|^2\), the
small-singular-value count and row-energy tail integration in
Lemma~\ref{l:thm:GE} give the coordinate-weight bound.
\end{proof}

\paragraph{The two-dimensional Fourier bound.}
For comparison, normalize a source entry by
\[
 \zeta_{\mathrm{iso}}=\sqrt2\,\overline\zeta,
 \qquad f_{\mathrm{iso}}(z)=\tfrac12 f(\overline z/\sqrt2),
 \qquad \operatorname{Cov}_{\mathbb R}(\zeta_{\mathrm{iso}})=\mathrm{Id}_2.
\]
The conjugation converts the bilinear observations into Hermitian
projections below. Define the block tilt by
\[
 \Lambda(t)=\log\mathbb E e^{\langle t,\zeta_{\mathrm{iso}}\rangle_{\mathbb R}},
 \qquad f_{\mathrm{iso},t}(x)
 =e^{\langle t,x\rangle_{\mathbb R}-\Lambda(t)}f_{\mathrm{iso}}(x),
\]
Identifying \(\mathbb C\) with \(\mathbb R^2\), its characteristic
function is
\[
 \chi_t(u)=\int_{\mathbb C}e^{i\langle u,x\rangle_{\mathbb R}}
                         f_{\mathrm{iso},t}(x)\,d^2x,
 \qquad u,t\in\mathbb R^2.
\]
Thus each component \(t_j\) of a complex tilt vector is also interpreted
as a two-dimensional real vector when used in \(\Lambda(t_j)\).

\begin{lemma}[Two-dimensional tilt and projection bounds]\label{l:thm:CENV}

There are \(C>0\) and \(\delta\in(0,1/4]\) depending only on \(\mu\),
such that \begin{equation}
c\log\left[(2\pi)^{-1}\int_{\mathbb R^2}
       |\chi_t(\sqrt c\,u)|^{1/c}\,du\right]
 \le C(c^{3/2}+\sqrt c\,|t|^2),
 \qquad 0<c\le\delta,\quad t\in\mathbb R^2.
 \label{l:eq:Z.5}
\end{equation} For \(U\in\mathbb C^{h\times k}\), \(U^*U=\mathrm{Id}_k\), write
\(\ell_j=\sum_{a=1}^k|U_{ja}|^2\), and let
\(q_U(y)=\int_{\ker U^*}\prod_j f_{\mathrm{iso}}((Uy)_j+w_j)\,dw\). If
\(\max_j\ell_j\le\delta\), then for every \(y\in\mathbb C^k\) and
\(\mathbf t=(t_1,\ldots,t_h)^T\in\operatorname{Ran}U\), \begin{equation}
\log q_U(y)\le-k\log(2\pi)-\langle\mathbf t,Uy\rangle_{\mathbb R}
       +\sum_{j=1}^h\Lambda(t_j)
       +C\sum_{j=1}^h(\ell_j^{3/2}+\sqrt{\ell_j}|t_j|^2).
 \label{l:eq:Z.6}
\end{equation} \end{lemma}

\begin{proof}
In Lemma~\ref{l:thm:FTILT}, replace the scalar tilted variance by
\(\Sigma(t)=\nabla^2\Lambda(t)=\mathrm{Id}_2+O(|t|)\).
The local Gaussian integral is \((\det\Sigma(t))^{-1/2}\), and the
Fourier-tail change of variables contributes \(c^{-1}\).
The same cosine expansion, spectral-gap argument and Plancherel bound
therefore give
\[
 (2\pi)^{-1}\int_{\mathbb R^2}|\chi_t(\sqrt c\,u)|^{1/c}\,du
 \le(\det\Sigma(t))^{-1/2}+Cc\le e^{C(c+|t|)}
\]
for small \(t\), and \(Cc^{-1}e^{C|t|^2}\) for large \(t\).
Here absolute continuity excludes \(|\chi_0(u)|=1\) at \(u\ne0\),
which would confine the law to countably many parallel lines.
The final estimates in the proof of Lemma~\ref{l:thm:FTILT} give
\eqref{l:eq:Z.5}.

For the projection bound, the maps
\(B_jy=(Uy)_j/\sqrt{\ell_j}\), for \(\ell_j>0\), have real
two-dimensional targets and satisfy
\(\sum_{j:\ell_j>0}\ell_j B_j^*B_j=\mathrm{Id}_k\).
Geometric Brascamp--Lieb and \eqref{l:eq:Z.5} give
\[
 \int_{\mathbb C^k}\prod_{j=1}^h|\chi_{t_j}((Uz)_j)|\,dz
 \le(2\pi)^k
 e^{C\sum_{j=1}^h(\ell_j^{3/2}+\sqrt{\ell_j}|t_j|^2)}.
\]
The fiber-continuity argument in Lemma~\ref{l:thm:ENV} applies to these
blocks with the same exponents \((1-\ell_j)^{-1}\).
Fourier inversion in real dimension \(2k\) contributes
\((2\pi)^{-2k}\), leaving \((2\pi)^{-k}\).
We can undo the tilt just as in that lemma: its factor is constant on the fiber
because \(\mathbf t\in\operatorname{Ran}U\) and \(w\in\ker U^*\).
This proves \eqref{l:eq:Z.6} at every point.
\end{proof}

\paragraph{The observation comparison.}
\begin{theorem}[Complex observation comparison]\label{l:thm:CAC}

Let \(X\in\mathbb C^{d\times h}\) have iid law \(\mu\), and let \(\mathsf C\) be
an independent auxiliary input taking values in a standard Borel space
\(\mathcal C\). A Borel rule makes
predictable queries \(\mathbf u_a\in\mathbb C^h\), observes \(\mathbf y_a=X\mathbf u_a\), and
satisfies \(U^*U=\mathrm{Id}_k\) on every history, where
\(U=[\mathbf u_1,\ldots,\mathbf u_k]\in\mathbb C^{h\times k}\) and
\(Y=[\mathbf y_1,\ldots,\mathbf y_k]\in\mathbb C^{d\times k}\). Define
\(P=\mathcal L(\mathsf C,Y)\),
\(G=\mathcal L(\mathsf C)\otimes\mathcal N_{\mathbb C}(0,1)^{\otimes dk}\),
\(\ell_j=\sum_{a=1}^k|U_{ja}|^2\), and
\(\Theta=\sum_{j=1}^h\ell_j^{3/2}\). There are positive law-dependent
constants \(C_{\mathrm{coarse}},C_{\mathrm{weight}},\delta_\mu\) and
\(\theta_{\mathrm{coarse}},\theta_\mu\in(0,1)\) such that, for
\(h\le d\) and every Borel observation event
\(F\subseteq\mathcal C\times\mathbb C^{d\times k}\), \begin{equation}
\begin{aligned}
 P(F)&\le e^{C_{\mathrm{coarse}}dk}G(F)^{\theta_{\mathrm{coarse}}},\\
 P(F)&\le\left[\mathbb E_G\mathbf1_F
              e^{C_{\mathrm{weight}}d\Theta}\right]^{\theta_\mu}
       &&\text{if }F\subseteq\{\max_j\ell_j\le\delta_\mu\}.
 \end{aligned}
 \label{l:eq:Z.7}
\end{equation} The reference is realized by a standard circular complex Gaussian
source, and all constants are uniform over the query rule and auxiliary
law. \end{theorem}

\begin{proof}
Fix the auxiliary input \(c\). Set \(\mathbf Z=\sqrt2\,\overline Y\),
and evaluate the queries on the recovered history:
\(\widetilde U_c(\mathbf z)=U(c,\overline{\mathbf z}/\sqrt2)\).
For source rows written as columns,
\[
 \mathbf z^{(i)}=(\mathbf Z_{i,:})^T
 =U^*(\sqrt2\,\overline{X_{i,:}})^T.
\]
Here \(U=\widetilde U_c(\mathbf Z)\). Thus the responses are Hermitian projections of independent isotropic
blocks. This invertible transformation preserves likelihood ratios and
relative entropy; write \(\widetilde P_c,\widetilde G_0\) for the
transformed laws, and set \(L_c=d\widetilde P_c/d\widetilde G_0\).
The Gaussian reference \(\widetilde G_0\) has
\(2dk\) independent standard real Gaussian coordinates.
The unitary changes of variables in Lemma~\ref{l:thm:LIKE} preserve
Lebesgue measure and are determined by previous responses, so
\begin{equation}
L_c(\mathbf Z)=\prod_{i=1}^d
 \frac{q_{\widetilde U_c(\mathbf Z)}(\mathbf z^{(i)})}
 {\phi_{2k}(\mathbf z^{(i)})},
 \qquad\phi_{2k}(z)=(2\pi)^{-k}e^{-\|z\|^2/2}.
 \label{l:eq:Z.8}
\end{equation}
For Gaussian sources, \(q_U=\phi_{2k}\), giving the stated reference.

Choose \(\delta_\mu\le\min\{1/4,\delta\}\) from Lemma~\ref{l:thm:CENV}.
For the weighted bound, take a law \(M\) satisfying the entropy and
integrability hypotheses of Lemma~\ref{l:thm:ENT} on the flat-frame event.
Apply Lemma~\ref{l:thm:SEQCOUP} to the real observation coordinates,
ordered by complex columns, and write \(\Pi\) for the coupling law.
Its Gaussian innovations \(g_{ia}\) have
independent \(N(0,1)\) real and imaginary parts, and
\[
 \mathbf n_i=\sum_{a=1}^k\mathbf u_a g_{ia},\qquad
 \mathbf e_i=\sum_{a=1}^k\mathbf u_a(Z_{ia}-g_{ia}),\qquad
 U\mathbf z^{(i)}=\mathbf n_i+\mathbf e_i,
 \quad\mathbb E_\Pi\sum_{i=1}^d\|\mathbf e_i\|^2
 \le2H(M\mid\widetilde G_0).
\]
The query is known before either component of the current Gaussian
column is revealed. Multiplication by \(b+ic\) has real matrix
\(\left(\begin{smallmatrix}b&-c\\c&b\end{smallmatrix}\right)\), so
each component of \((\mathbf n_i)_j\) has predictable variance sum
\(\ell_j\). Lemma~\ref{l:thm:GAUSS3} therefore gives
\begin{equation}
\mathbb E_\Pi|(\mathbf n_i)_j|^3+
 \mathbb E_\Pi[\sqrt{\ell_j}|(\mathbf n_i)_j|^2]
 \le132\,\mathbb E_M\ell_j^{3/2}.
 \label{l:eq:Z.9}
\end{equation}
No independence of the original real and imaginary entry components is
used. Since \(\Lambda(t)-|t|^2/2\le C\min\{|t|^3,|t|^2\}\), the
tilt \(\mathbf t_i=\mathbf n_i+\lambda\mathbf e_i\) from the proof of
Lemma~\ref{l:thm:ENT} reduces the quadratic part to a term in the
coupling error, with all inner products taken over the reals. The error
coefficient is at most
\(1/2-\lambda+C\lambda^2\), as in \eqref{l:eq:A.10}.
Choose \(\lambda>0\) small; \eqref{l:eq:Z.6}, \eqref{l:eq:Z.9} and the
coupling cost give
\begin{equation}
\mathbb E_M\log L_c\le C_{\mathrm{ent}}d\mathbb E_M\Theta
                      +(1-\lambda)H(M\mid\widetilde G_0).
 \label{l:eq:Z.10}
\end{equation}
The entropy identity, Jensen inequality and likelihood truncation in the
weighted proof of Theorem~\ref{l:thm:AC} now give the second bound in
\eqref{l:eq:Z.7}.
For the coarse bound, \eqref{l:eq:Z.3} gives
\(L_c\le e^{Cdk+\|\mathbf Z\|_{\mathrm{HS}}^2/2}\), and
\(\|\mathbf Z\|_{\mathrm{HS}}^2\le2\sum_{a=1}^k s_a(X)^2\).
The energy bound of Lemma~\ref{l:thm:CGEO} lets us finish as in the
coarse proof of Theorem~\ref{l:thm:AC}. We then average over \(c\) and
return to the original responses.
\end{proof}

\subsection{Arnoldi links and stopping}
\label{l:complex-arnoldi-links-and-stopping}

Run Algorithms~\ref{l:thm:ARN} and \ref{l:thm:DEL} with Hermitian
projections and Grams:
\[
 \mathrm{Id}-\sum_{a=1}^j\mathbf v_a\mathbf v_a^*,\qquad
 \mathsf G(J)=V_{J,:}^*V_{J,:},\qquad
 \ell_j(J)=V_{j,:}\mathsf G(J)^{-1}V_{j,:}^*.
\]
The orthogonal completion rule, deletion order, weights and stopping
convention are unchanged; in particular \(W_k\) is still defined by
\eqref{l:eq:M.6}.

\begin{proposition}[Original-law complex links]\label{l:thm:CML}

For every \(0<\varepsilon\le1\), there are
\(\beta_{\mu,\varepsilon},\allowbreak c_{\mathrm{link}}(\mu,\varepsilon)>0\) such
that, for all sufficiently large \(d\), a deterministic integer
\(m\ge1\) satisfies \begin{equation}
2m<d,\qquad \beta_{\mu,\varepsilon}d\le2m,
 \qquad256m\le\varepsilon d,
 \qquad
 \mathbb P_{A,\mathbf s}\{W_{2m}\le(1-\varepsilon/4)^{m^2}\}
       \le e^{-c_{\mathrm{link}}d^2}.
 \label{l:eq:Z.11}
\end{equation} Here the complex core \(A\) and starter \(\mathbf s\) have
independent iid entries with law \(\mu\). \end{proposition}

\begin{proof}
Under a core with iid standard complex Gaussian entries,
the residual squares are independent and
\(R_j^2\sim\operatorname{Gamma}(d-j,1)\), with shape \(d-j\) and
rate one. For integers \(0\le s<d-j\),
\begin{equation}
\mathbb E(R_j^2)^{-s}=\prod_{\ell=1}^s(d-j-\ell)^{-1}.
 \label{l:eq:Z.12}
\end{equation}
The denominators are no smaller than in \eqref{l:eq:L.2}, so the same
moment calculation gives \eqref{l:eq:L.1}.
For mixed reference columns, with \(\phi(z)=\pi^{-1}e^{-|z|^2}\),
choose \(0<q\le\min\{1,t_0\}\). Write
\(L_{I_0}=\prod_{i=1}^d\prod_{j\in I_0}f(A_{ij})/\phi(A_{ij})\);
\(\mathbb E_{\mathrm G}\) denotes expectation over this standard complex
Gaussian core and the unchanged independent starter. Then
\[
 D_q=\int_{\mathbb C}f^{1+q}\phi^{-q}
 \le(\pi K)^q\mathbb E e^{q|\zeta|^2}<\infty,
 \qquad \mathbb E_{\mathrm G}L_{I_0}^{1+q}=D_q^{d|I_0|}.
\]
This is the likelihood moment used in Lemma~\ref{l:thm:LM}.

To use the stopping argument, we also need reconstruction and the
determinant-loss identities.
Reconstruction in Lemma~\ref{l:thm:RECONPROP} uses Hermitian
Gram--Schmidt; the determinant loss \eqref{l:eq:S.2} and range
monotonicity \eqref{l:eq:G.8} hold for Hermitian Grams.
The Krylov minors have the same Vandermonde witness, and a nonzero
complex polynomial has zero set of planar product measure zero.
Thus the reference-range inclusion \eqref{l:eq:P.2} still holds.
We can now follow the proofs of Propositions~\ref{l:thm:SS} and
\ref{l:thm:ML}, using Lemma~\ref{l:thm:CGEO} for the geometric
estimates and Theorem~\ref{l:thm:CAC} for comparison. The constants
may change, but the parameter choices and union bounds do not.
Setting \(k=2m\) gives \eqref{l:eq:Z.11}.
\end{proof}

\subsection{The border estimate}
\label{l:complex-gaussian-border-and-completion}

For the border calculation, each diagonal block is now a single complex
scalar. We first record its negative-moment bound.

\begin{lemma}[Complex Gaussian inverse determinants]\label{l:thm:CDN}

Let \(Z\in\mathbb C^m\) be a complex Gaussian vector with arbitrary
mean, Hermitian covariance \(\Gamma>0\), and zero centered
pseudocovariance: with \(Z_c=Z-\mathbb E Z\), these conditions mean
\(\Gamma=\mathbb E[Z_cZ_c^*]\) and \(\mathbb E[Z_cZ_c^T]=0\).
For every deterministic \(D\in\mathbb C^{m\times m}\),
\begin{equation}
\mathbb E|\det(\operatorname{diag}Z+D)|^{-1/2}
 \le C^m(\det\Gamma)^{-1/4}.
 \label{l:eq:Z.13}
\end{equation} In particular the same bound holds for
\(\mathbb E\prod_j|Z_j|^{-1/2}\). The constant is universal.
\end{lemma}

\begin{proof}
A circular Gaussian scalar of variance \(v\), with any mean, has
density at most \((\pi v)^{-1}\). Hence
\(\mathbb P\{|Z+a|<t\}\le\min\{1,t^2/v\}\), and tail integration
gives \(\mathbb E|Z+a|^{-1/2}\le Cv^{-1/4}\).
Apply the Schur-complement induction of Lemma~\ref{l:thm:DN}(ii),
integrating one complex coordinate at a time, conditional on all the
others. Its conditional
law remains circular, with variance \(v\) equal to the covariance
Schur complement, and \(\det\Gamma=v\det\Gamma_{\mathrm{rem}}\).
The complementary determinant is nonzero almost surely, since its
highest-degree term is the product of its diagonal variables.
The scalar bound proves \eqref{l:eq:Z.13}; set \(D=0\) for the product
assertion.
\end{proof}

\begin{proposition}[Complex good-core border probability]\label{l:thm:CBORDER}

Fix \(0\le r<R\le1\), \(a>0\), and an integer \(m\ge1\) with
\(2m<d=n-1\) and \((r/R)^m\le1/2\). For
every fixed complex core \(A\) and starter \(\mathbf s\) with
\(W_{2m}(A,\mathbf s)\ge a^{m^2}\), an independent border
\(\mathbf b=(\xi_{11},\mathbf x,\mathbf y)\sim\mu^{\otimes(2n-1)}\)
satisfies \begin{equation}
\mathbb P_{\mathbf b}\left\{\rho
    \begin{pmatrix}\xi_{11}&\mathbf y^T\\\mathbf x&A\end{pmatrix}\le r\sqrt n\right\}
 \le\exp\{C_{\mathrm{border}}(\mu)n+(m^2/4)\log(R/a)\}.
 \label{l:eq:Z.14}
\end{equation} \end{proposition}

\begin{proof}
We begin by replacing the border with a circular Gaussian one, as in
Subsection~\ref{l:sec:real_border}. For \(\kappa\ge t_0^{-1}\) and
\(\phi_\kappa(z)=(\pi\kappa)^{-1}e^{-|z|^2/\kappa}\), write
\(\mathbb P_\kappa=\mathcal N_{\mathbb C}(0,\kappa)^{\otimes(2n-1)}\)
for the law of the unnormalized Gaussian border, with the core fixed. Then
\[
 J_f=\int_{\mathbb C}f^2/\phi_\kappa
 \le\pi\kappa K\,\mathbb E e^{|\zeta|^2/\kappa}<\infty,
 \qquad \mathbb P_{\mathbf b}(E)
 \le J_f^{(2n-1)/2}\mathbb P_\kappa(E)^{1/2}.
\]
The core stays fixed in this comparison. Complete Arnoldi to a unitary
\(Q\); the normalized transformed matrix is
\[
 T=\begin{pmatrix}g_{\mathrm b}&\mathbf x_{\mathrm b}^T\\
                  \mathbf u_{\mathrm b}&B\end{pmatrix},
 \qquad B=(Q^*AQ)^T/\sqrt n.
\]
Its core is lower Hessenberg, with \(B_{j,j+1}=R_j/\sqrt n\), and
its Gaussian border entries remain independent
\(\mathcal N_{\mathbb C}(0,\kappa/n)\) by unitary invariance.
The ordinary transpose in this characteristic-polynomial identity is
retained; projections and covariances use adjoints.

The cofactor and evaluation identities in
Lemmas~\ref{l:lem:hessenberg-cofactor} and
\ref{l:lem:evaluation-determinant}(i) are algebraic over \(\mathbb C\).
The polynomial \(Q_B\) is now complex monic of degree
\(d_{\mathrm{tail}}=d-2m\). Its highest forward difference on the
same real grid is still \(d_{\mathrm{tail}}!\) times the grid spacing
to that power. The triangle inequality therefore gives the choice of
\(w\) in Lemma~\ref{l:lem:evaluation-determinant}(ii), with the same
determinant identity and lower bound \eqref{l:eq:D.8}--\eqref{l:eq:D.9}.
No conjugate pairing of polynomial values is needed.

We can now use the two conditional integrations from
Lemma~\ref{l:lem:border-negative-moment}. Integrate
\(\mathbf x_*=(\mathbf x_{\mathrm b})_{[m]}\) first, fixing the
other border coordinates. The evaluation vector has covariance
\((\kappa/n)HH^*\). Then integrate
\(\mathbf u_*=(\mathbf u_{\mathrm b})_{[m+1,2m]}\), fixing its
complement. In
\(H=\operatorname{diag}(E_{\mathrm{suf}}\mathbf u_*)E_{\mathrm{pre}}+H_0\),
the Gaussian diagonal has covariance
\((\kappa/n)E_{\mathrm{suf}}E_{\mathrm{suf}}^*\).
Lemma~\ref{l:thm:CDN} gives
\begin{equation}
\begin{aligned}
 \mathbb E_{\mathbf x_*}\prod_{h=1}^m|P_T(z_h)|^{-1/2}
 &\le C^m(n/\kappa)^{m/4}|\det H|^{-1/2},\\
 \mathbb E_{\mathbf u_*}|\det H|^{-1/2}
 &\le C^m(n/\kappa)^{m/4}
          |\det E_{\mathrm{pre}}\det E_{\mathrm{suf}}|^{-1/2}.
 \end{aligned}
 \label{l:eq:Z.15}
\end{equation}
Both bounds are uniform in the fixed means. Finiteness of the second
expectation gives the almost-sure invertibility needed in the first.
Tonelli combines the two integrals. The calculation in
\eqref{l:eq:D.13}--\eqref{l:eq:D.14} yields
\begin{equation}
\begin{aligned}
 \mathbb P_\kappa\{\rho(T)\le r\}
 &\le C^{2m}\left[
   \frac{(n/\kappa)^m(w+r^m)^n}
        {m^mw^{m-1}W_{2m}|Q_B(w)|}\right]^{1/2}\\
 &\le e^{C_\kappa n+(m^2/2)\log(R/a)}.
 \end{aligned}
 \label{l:eq:Z.16}
\end{equation}
Returning to the original border halves the logarithmic term and adds
only a linear cost, proving \eqref{l:eq:Z.14}.
\end{proof}

\begin{proof}[\proofname\ of Theorem \ref{l:thm:MT}, complex case]Fix \(r<1\), set \(\varepsilon=1-r\),
\(a=1-\varepsilon/4\), and \(R=(r+a)/2\). Proposition \ref{l:thm:CML} provides a
deterministic \(m\ge\beta_{\mu,\varepsilon}d/2\) and a bad-core
probability at most \(e^{-c_{\mathrm{link}}d^2}\). For large \(d\),
\((r/R)^m\le1/2\). On the complementary event, Proposition \ref{l:thm:CBORDER}
bounds the conditional border probability by \[
 \exp\{C_{\mathrm{border}}(\mu)n-\delta_rm^2\},
 \qquad\delta_r=\tfrac14\log(a/R)>0.
\] The independent starter is absent from the spectral event. Averaging
over the core and starter and adding the bad-core probability proves
\eqref{l:eq:M.2}, since \(d=n-1\) and \(m\) is proportional to \(d\). \end{proof}

\raggedbottom
\bibliographystyle{amsalpha}
\bibliography{spectral_radius_references}
\end{document}